\documentclass[12pt,reqno]{amsart}
\usepackage[margin=1in]{geometry}
\usepackage{amsmath,amssymb,amsthm,graphicx,amsxtra, setspace}
\usepackage[utf8]{inputenc}
\usepackage{mathrsfs}
\usepackage{hyperref}
\usepackage{upgreek}
\usepackage{mathtools}
\usepackage[dvipsnames]{xcolor}
\usepackage[mathcal]{euscript}
\usepackage{amsmath,units}
\usepackage{pgfplots}
\usepackage{tikz}
\usepackage{verbatim}
\allowdisplaybreaks
\usepackage{dsfont}
\usepackage{fontenc}
\usepackage{textcomp}
\usepackage{marvosym}
\usepackage{eurosym}
\usepackage{upgreek}
\usepackage[pagewise]{lineno}

\DeclareMathAlphabet{\mathpzc}{OT1}{pzc}{m}{it}

\usepackage[cyr]{aeguill}

\colorlet{darkblue}{blue!50!black}

\hypersetup{
	colorlinks,%
	citecolor=blue,%
	filecolor=red,%
	linkcolor=red,%
	urlcolor=blue,%
	pdfnewwindow=true,%
	pdfstartview={FitH}
}

\newtheorem{theorem}{Theorem}[section]

\newtheorem{proposition}[theorem]{Proposition}

\newtheorem{definition}[theorem]{Definition}

\newtheorem{remark}[theorem]{Remark}

\let\originalleft\left
\let\originalright\right
\renewcommand{\left}{\mathopen{}\mathclose\bgroup\originalleft}
\renewcommand{\right}{\aftergroup\egroup\originalright}

\renewcommand{\d}{\/\mathrm{d}\/}

\def\L{\mathbb{L}}

\def\I{\mathrm{I}}

\def\C{\mathrm{C}}
\def\f{\boldsymbol{f}}

\def\D{\mathrm{D}}
\def\y{\mathbf{y}}
\def\q{\boldsymbol{q}}
\def\Y{\boldsymbol{Y}}
\def\Z{\boldsymbol{Z}}

\def\X{\boldsymbol{X}}
\def\x{\boldsymbol{x}}

\def\p{\boldsymbol{p}}

\def\y{\boldsymbol{y}}
\def\z{\boldsymbol{z}}
\def\v{\boldsymbol{v}}

\def\V{\mathbb{V}}
\def\R{\mathbb{R}}
\def\wi{\widetilde}

\def\U{\mathbb{U}}

\def\u{\boldsymbol{u}}

\def\H{\mathbb{H}}

\newcommand{\eps}{\varepsilon}

\renewcommand{\d}{\/\mathrm{d}\/}

\newcommand{\Addresses}{{% additional braces for segregating \footnotesize
		\footnote{
		\noindent \textsuperscript{1,2}Department of Mathematics, Indian Institute of Technology Roorkee-IIT Roorkee,
		Haridwar Highway, Roorkee, Uttarakhand 247667, INDIA.\par\nopagebreak
		\noindent  \textit{e-mail:} \texttt{Manil T. Mohan: maniltmohan@ma.iitr.ac.in, maniltmohan@gmail.com}
		
		\textit{e-mail:} \texttt{Sagar Gautam: sagar\_g@ma.iitr.ac.in, sagargautamkm@gmail.com}
		
		\noindent \textsuperscript{*}Corresponding author.

			\textit{Key words:} convective Brinkman-Forchheimer equations, viscosity solutions, Hamilton-Jacobi-Bellman equations, verification theorem, feedback synthesis, Pontryagin maximum principle
			
			Mathematics Subject Classification (2020): Primary 76D55, 49L25; Secondary 76B75, 49N35, 76S05

}}}

\begin{document}
	
	%	\linenumbers
	
	\title[A verification theorem for optimal control of CBF equations]{A Verification Theorem for an Optimal Control Problem Governed by the Convective Brinkman--Forchheimer Equations
		\Addresses}
		\author[S. Gautam and M. T. Mohan]
	{Sagar Gautam\textsuperscript{1} and Manil T. Mohan\textsuperscript{2*}}

	\maketitle

\begin{abstract}
	This article establishes a verification theorem for an optimal control problem governed by the two- and three-dimensional convective Brinkman--Forchheimer equations on the $d$-dimensional torus, $d\in\{2,3\}$:
	$$\frac{\partial\mathfrak{u}}{\partial t}
	-\mu\Delta\mathfrak{u}
	+(\mathfrak{u}\cdot\nabla)\mathfrak{u}
	+\alpha\mathfrak{u}
	+\beta|\mathfrak{u}|^{r-1}\mathfrak{u}
	+\nabla\mathfrak{p}
	=\boldsymbol{f},
	\qquad
	\nabla\cdot\mathfrak{u}=0,$$
	where $\mu,\alpha,\beta>0$ and $r\in[1,\infty)$. We derive the Pontryagin maximum principle and develop a verification framework for the associated control problem, a topic that has received comparatively little attention for fluid models of Navier--Stokes type. A major challenge in establishing the verification theorem and the corresponding feedback characterization for the CBF system is that the analysis requires a substantially different regularity framework from that used for the two-dimensional Navier--Stokes equations. In particular, the present approach relies on strong solution theory, a delicate treatment of the nonlinear absorption term, novel estimates in negative-order Sobolev spaces, and continuous dependence estimates in stronger topologies, especially in the three-dimensional setting. A distinctive feature of the present work is that the verification framework is developed not only in two dimensions, but also in the three-dimensional supercritical regime, corresponding to $r\in(3,5]$, and in the critical case $r=3$ under the condition $2\beta\mu\geq1$. Consequently, the feedback characterization and verification arguments can be rigorously justified in both two and three dimensions. 
\end{abstract}

	%\tableofcontents
	
	\section{Introduction}\setcounter{equation}{0}
Optimal control problems for fluid flow equations have attracted considerable attention over the past several decades, particularly for systems governed by the Navier--Stokes equations (NSE). While the classical NSE provide an accurate description of the motion of viscous incompressible fluids in open domains (that is, fluid flows  without any solid obstacles or porous structure interacting with it; for example, water flowing in a river, flow through channel, air flowing around a wing, etc.), they fail to adequately model flows through porous media. Such flows arise in numerous practical applications, including groundwater transport in aquifers, oil and gas recovery from reservoirs, and industrial filtration processes. In these settings, the fluid is subjected not only to viscous effects but also to resistance induced by the porous structure, exhibiting linear drag (Darcy's law) at low velocities and nonlinear drag (Forchheimer's law) at higher velocities. Since these phenomena are not captured by the classical NSE, more general fluid models incorporating porous medium effects are required.

Among the various porous media flows, the convective Brinkman-Forchheimer (CBF) equations have been introduced as an extension of the classical NSE and provide an important framework that incorporates viscous diffusion, inertial convection, Darcy resistance, and nonlinear drag effects. Let us now provide the mathematical formulation of the CBF system. Denote by $t$, the initial time with $t\in[0,\infty)$, and by $T$, the terminal time, with $T\in(t,\infty)$. We denote by $\mathbb{T}^{d}=\big(\R/\mathbb{Z}\big)^{d}$, a $d$-dimensional torus with dimension $d\in\{2,3\}$. 
%     This control problem has vast number of industrial and engineering applications (see \cite{SSS}). In such control problems the state variable can be taken as one of the physical quantity such as velocity or vorticity vector field. The state equations are governed by the convective Brinkman-Forchheimer (or \emph{damped Navier-Stokes equations}). 
Then the unknowns are the velocity vector field $\mathfrak{u}(\cdot):[t, T]\times\mathbb{T}^d\to\R^d$ (the state variable) and pressure $p(\cdot):[t, T]\times\mathbb{T}^d\to\R$ which satisfy the following set of equations (known as CBF equations):

\begin{equation}\label{CBF}
	\left\{
	\begin{aligned}
		\frac{\partial\mathfrak{u}}{\partial s}-\mu \Delta\mathfrak{u}+(\mathfrak{u}\cdot\nabla)\mathfrak{u}+\alpha\mathfrak{u}+\beta|\mathfrak{u}|^{r-1}\mathfrak{u}+\nabla\mathfrak{p}&=\mathfrak{g}(s)+\mathfrak{a}(s), \ \text{ in } \ (t,T)\times\mathbb{T}^{d}, \\ \nabla\cdot\mathfrak{u}&=0, \ \text{ in } \ [t,T]\times\mathbb{T}^{d}, \\
		\mathfrak{u}(t)&=\mathfrak{u}_0 \ \text{in} \ \mathbb{T}^{d},\\
		\int_{\mathbb{T}^d}\mathfrak{p}(t,x)\d x&=0, \ \text{ in } \ (t,T),
	\end{aligned}
	\right.
\end{equation}
where $\mathfrak{a}(\cdot)$ is a given measurable function that acts as a control 
%and taking values in some control set, say $\Theta$ (a complete metric space)
 and $\mathfrak{g}(\cdot,\cdot):[0, T]\to\R^d$ represents an external forcing. The final condition in \eqref{CBF} is included to ensure the uniqueness of the pressure.
 Furthermore, the velocity field $\mathfrak{u}(\cdot,\cdot)$ and pressure $\mathfrak{p}(\cdot,\cdot)$ are subject to the following periodic boundary conditions:
\begin{align}\label{pbc}
	\mathfrak{u}(x+e_{i},\cdot) = \mathfrak{u}(x,\cdot)  \ \text{ and }  \ \mathfrak{p}(x+e_{i},\cdot) = \mathfrak{p}(x,\cdot),
\end{align}
for all $x\in\R^{d}$ and $i=1,\ldots,d,$ where $\{e_i\}_{i=1}^d$ denotes the standard basis of $\R^{d}.$ 
The viscous term $-\mu\Delta\mathfrak{u}$ with $\mu>0$ (Brinkman coefficient) accounts for the viscous stress arising from the pore
microstructure. The nonlinear term $(\mathfrak{u}\cdot\nabla)\mathfrak{u}$ describes the convective transport and the linear term $\alpha\mathfrak{u}$ with $\alpha>0$ (Darcy coefficient) represents the viscous resistance at low flow rates. The nonlinear term $\beta|\mathfrak{u}|^{r-1}\mathfrak{u}$ with $\beta>0$ (Forchheimer coefficient) and $r\geq1$ (absorption exponent) represents the Forchheimer correction capturing inertial effects at higher flow rates. For a detailed discussion of the physical background and further properties of the CBF model, we refer the reader to \cite{SGKKMTM}. 

The case $r=3$ is called the \emph{critical exponent} because, when $\alpha=0$, the CBF equations exhibit the same scaling properties as the classical NSE (\cite[Proposition 1.1]{KWH}). 
Compared to classical NSE (which can be recovered from \eqref{CBF} when $\alpha=\beta=0$) the CBF system possesses a richer dissipative structure that
allows for a substantially more complete analytical theory, particularly in
three spatial dimensions. A central role in this theory is played by the
absorption exponent $r\geq 1$, which governs the strength of the
Forchheimer nonlinearity relative to the convective term and determines the
global analytical behaviour of solutions. The following regimes are distinguished:

\begin{itemize}
	\item \textbf{Subcritical case} $1\leq r<3$ and the critical case $r=3$ with $2\beta\mu<1$. The Forchheimer damping
	is weaker relative to the convective nonlinearity. In this regime,
	the analysis is closer in spirit to the pure NSE theory and global regularity in three dimensions remains an open problem. The special case $r=2$ of the CBF equations \eqref{CBF} is commonly referred to as the incompressible Toner--Tu equations; see \cite{HBYPKK}.
	
	\item \textbf{Supercritical case} $r>3$ and the critical case $r=3$ with $2\beta\mu\geq 1$. The Forchheimer drag
	dominates the convective nonlinearity, and the equations enjoy
	significantly better analytical properties, including global
	existence and uniqueness of strong solutions in both two and three
	dimensions.
\end{itemize} 
We emphasize that, in sharp contrast to the three-dimensional Navier--Stokes equations, for which Leray--Hopf weak solutions are only known to satisfy the energy inequality, weak solutions of the CBF system \eqref{CBF} in the critical $(r=3)$ and supercritical $(r>3)$ regimes satisfy the energy equality. Moreover, under the conditions $r=3$ with $2\beta\mu\geq1$ and $r>3$, global weak solutions are unique, a property that remains open for the three-dimensional NSE. The well-posedness theory for \eqref{CBF}, together with the corresponding energy and regularity estimates, has been extensively investigated in \cite{SMTM} and the references therein. Building upon these foundational results, the present work is devoted to the optimal control theory associated with the CBF system \eqref{CBF}.

\subsection{Optimal control problems for NSE and related models} 
The systematic study of optimal control problems governed by NSE was initiated by Lions \cite{JLL} and has since evolved into a comprehensive theory through the contributions of Fursikov \cite{AVf}, Gunzburger \cite{MDG}, Sritharan \cite{SSS}, and many others \cite{FART,VBoc,TRM}, see also the references therein. A fundamental characterization of optimality is provided by the Pontryagin Maximum Principle (PMP), which yields necessary conditions involving the Hamiltonian and an adjoint variable. For finite dimensional systems it was established by Pontryagin and his group \cite{LPVRE}. Extending the PMP to infinite dimensional system governed by PDEs is considerably more involved due to the unboundedness of the underlying operators. We refer to \cite{HOFHF,HOF,XJJMY,JPRHZ,JYXYZ} for representative contributions in this direction. In the context of fluid flow models, Sritharan \cite{SSS1,SSS} established the PMP in the framework of  distributed control problems (see also the works \cite{HOSS1,HOSS,dppfeedmtm,SSS2,SSS4,Hyu}) for the NSE in two dimensions. 
%Since 3D NSE is not known to have global well-posed solutions for arbitrary controls, unlike the 2D case. Their main methodological contribution is rather than treating the state equation as constraint, they approximate the whole problem (state equation plus state constraint) via a penalty functional which depends on both state and control variable. This converts a non wellposed problem into a sequence of wellposed problems 
A notable contribution in the context of the 3D NSE is due to Wang \cite{GengW}, who established the PMP through a penalization approach despite the absence of a global well-posedness theory. The analysis covers optimal control problems subject to three classes of state constraints, namely integral constraints, two-point boundary (time-variable) constraints, and periodic constraints. Furthermore, the authors in \cite{BTAD} derived the PMP for the 3D NSE with pointwise control constraints using the spike variation (needle perturbation) method. Their approach relies on the global regularity of strong solutions, which is ensured under a suitable smallness condition involving the variation parameter and the energy norm of the external forcing.

Another fundamental approach to optimal control is provided by dynamic programming, which characterizes optimality through the value function. By Bellman's principle of optimality, the value function formally satisfies the Hamilton--Jacobi--Bellman (HJB) equation, a fully nonlinear first-order PDE. However, even in finite dimensions, smooth solutions to the HJB equation are generally unavailable (see \cite{MBCD}). This difficulty is overcame by the notion of viscosity solutions, introduced by Crandall and Lions \cite{MGL,MGEL,Hsh1} in the finite-dimensional setting and subsequently extended to infinite-dimensional spaces in \cite{MGL1,MGL2,MGL3}, with further developments due to Tataru \cite{DT1,DT2,DT3}. For the NSE, the dynamic programming approach was initiated by Sritharan \cite{SSS1,SSS}. A major subsequent advance was achieved in \cite{FGSSA1}, where the uniqueness of viscosity solutions for the 2D NSE was established, a considerably more delicate and challenging problem.

The verification theorem, which is the primary focus of the present work, constitutes a cornerstone of the dynamic programming approach. It ensures that a candidate solution of the HJB equation, typically understood in the viscosity sense, coincides with the value function and generates an optimal feedback control. Verification theorems have been widely investigated in both finite- and infinite-dimensional frameworks. Without attempting an exhaustive account, we refer the reader to \cite{VBGDP1,ABGDP,PCGDP,GFAS1,VENR,WHF1,Mgar,JYXYZ,XYZ} and the references therein. These results provide sufficient conditions for optimality and often yield the existence and, in certain cases, an explicit characterization of optimal feedback controls. In the context of the 2D NSE, verification theorems were established by Sritharan \cite{SSS1,SSS}; see also \cite{VBSSSH,HOSS1,HOSS,HOF}.

\subsubsection{Optimal control problems for CBF equations} We now turn to optimal control problems governed by the CBF equations. Despite the well-developed analytical theory for the CBF system, the corresponding optimal control theory remains comparatively less developed than that for the NSE. In two dimensions, the existence of optimal controls and the derivation of first-order necessary optimality conditions were established in \cite{Op1,Op3}. The PMP for the time-optimal control problem was obtained in \cite{TiOp2}, while the corresponding result for CBF equations driven by time-periodic inputs was established in \cite{Op2}. More recently, \cite{KKFC} studied the velocity tracking problem for the 3D critical CBF equations $(r=3)$ and established the existence of optimal controls together with first-order necessary conditions by exploiting the regularity properties of weak solutions. In contrast, the dynamic programming theory for the CBF equations is still far from complete. In particular, although the associated HJB equation was shown in \cite{smtm2} to admit a viscosity solution framework in both two and three dimensions, the corresponding verification theory and feedback characterization of optimal controls have remained open. The main objective of the present work is to bridge this gap by establishing a verification theorem for the CBF system.

\begin{table}[ht]
	\begin{tabular}{|c|c|c|c|c|}
		\hline
		\textbf{Case}&Dimension &$ r$& Conditions on 
		$\mu$ \& $\beta$ \\
		\hline
		\textbf{I}&$d=2$ &$r\in(1,\infty)$&  for any  
		$\mu>0,\alpha>0$ and $\beta>0$  \\
		\hline
		\textbf{II}&$d=3$ &$r\in(3,5]$&  for any  
		$\mu>0,\alpha>0$ and $\beta>0$  \\
		\hline
		\textbf{III}&$d=3$ &$r=3$&for $\mu>0,\alpha>0$ and
		$\beta>0$ with $2\beta\mu\geq1$ \\
		\hline
	\end{tabular}
	\vskip 0.1 cm
	\caption{Values of $\mu,\alpha$ and $\beta$, and $r$ for the verification result.}
	\label{Table1}
\end{table}

\subsection{Mathematical challenges, approaches and main contributions}
%The current work establishes the PMP and verification theorem for the optimal control problem of the CBF equations \eqref{CBF} in both 2D and 3D, inspired by the work \cite{SSS1,SSS}, which was carried out for the 2D NSE. Although the underlying infinite-dimensional framework closely follows \cite{SSS1,SSS} (for the 2D NSE), extending it to the CBF system \eqref{CBF}, especially in dimension three, is far from straightforward. The presence of the absorption term fundamentally changes the analytical structure of the problem and introduces several new difficulties. At the same time, the absorption term provides stronger regularisation effects (in the supercritical regime), which, when carefully exploited, allow us to establish the verification-type results in both two and three space dimensions. We describe the principal mathematical difficulties and the approaches introduced to overcome them below.
The primary objective of the present work is to establish a verification theorem for the optimal control problem governed by the CBF system \eqref{CBF}. A fundamental ingredient in our analysis is the global well-posedness of strong solutions. In fact, several key aspects of the verification theory developed herein, including the derivation of estimates in negative-order Sobolev spaces, continuous dependence results in stronger topologies, the study of the linearized system, and the feedback characterization of optimal controls, depend crucially on strong-solution regularity. For the CBF equations on the torus $\mathbb{T}^d$, $d\in{2,3}$, it is known that the nonlinear absorption term in the critical and supercritical regimes $(r=3$ and $r>3)$ provides sufficient dissipation to guarantee the global existence and uniqueness of strong solutions in both two and three dimensions.The global solvability relies on the identity
\begin{align}\label{torusequality}
	\int\limits_{\mathbb{T}^d}(-\Delta \mathfrak{u}(x))\cdot|\mathfrak{u}(x)|^{r-1}\mathfrak{u}(x)\d x= \int\limits_{\mathbb{T}^d}|\nabla\mathfrak{u}(x)|^2|\mathfrak{u}(x)|^{r-1}\d x +\left[\frac{4(r-1)}{(r+1)^2}\right]
	\int\limits_{\mathbb{T}^d}|\nabla|\mathfrak{u}(x)|^{\frac{r+1}{2}}|^2\d x,
\end{align}
which holds without boundary terms on the torus. By contrast, the situation on a bounded domain $\mathcal{O}\subset\mathbb{R}^d$ is considerably more delicate. In this setting, the Leray--Helmholtz projection $\mathcal{P}$ and the Laplacian $-\Delta$ do not, in general, commute, the quantity $\mathcal{P}(|\mathfrak{u}|^{r-1}\mathfrak{u})$ need not satisfy homogeneous boundary conditions, and the pressure term can no longer be eliminated; see \cite{KT2}. Furthermore, unlike the case of mean-zero vector fields, the Poincar\'e inequality is not directly applicable since the spatial average of $\mathfrak{u}$ need not vanish. Consequently, one must work with the full $\mathbb{H}^1$-norm rather than the seminorm induced by the gradient. Another essential difference stems from the spectral properties of the Stokes operator $\mathcal{A}$ on the torus $\mathbb{T}^d$: since $\boldsymbol{0}$ belongs to the spectrum of $\mathcal{A}$, the operator is not invertible. To avoid difficulties arising from the lack of invertibility of $\mathcal A$, we work with the shifted operator $\mathcal A+\mathrm I$ throughout the analysis.

%A key ingredient in the feedback analysis of the associated optimal control problem is the continuity properties of the auxiliary functional $\mathcal{W}$ (see Section \ref{feedbackana}). For this, initially, we first need to establish the following continuous dependence estimates 
%\begin{align}\label{zz12strong}
%	\sup\limits_{t\in[0,T]}\|\Z_1-\Z_2\|_{\V}\leq C\|\Z_1(0)-\Z_2(0)\|_{\V},
%\end{align}
%for some constant $C$, where $\Z_1$ and $\Z_2$ are two strong solutions of the system \eqref{stapPMP}. We show that the estimate
%can be established by carefully exploiting the coercivity of the absorption
%term, energy estimates \eqref{eqncont2} for the strong solutions, and Sobolev embedding $\V\hookrightarrow\wi\L^{r+1}$ for $r$ in Table \ref{Table1} (see the proof of \eqref{supVnormest} in Proposition \ref{supVnorm}). For $r\leq 3$,
%the analysis is similar to that of NSE, and the required regularity cannot be propagated (where such an estimate \eqref{zz12strong} cannot be established globally with the regularity of the strong solution). 

A cornerstone of the analysis required for verification result is the continuous dependence estimate in the negative Sobolev norm $\|\cdot\|_{\D(\mathcal{A}+\I)^{-\frac12}}$ (equivalently, the $\V^{*}$-norm), which is ultimately used to establish the locally Lipschitz property of the value function with respect to the initial data (see Section \ref{anoptimaL}) and further used in feedback analysis (Section \ref{feedbackana}). For the 2D NSE, such an estimate follows directly from the energy inequality satisfied by weak solutions. For the CBF system, however, weak solutions alone do not provide sufficient regularity to close the required estimates. Specifically, one must estimate the difference $\Z_1-\Z_2$ of two solutions of the CBF system \eqref{stapPMP} in the negative norm by controlling the paring 
\begin{align}\label{bz1z2diff}
	(\mathfrak{M}(\Z_1)-\mathfrak{M}(\Z_2),(\mathcal{A}+\I)^{-1}(\Z_1-\Z_2)),
\end{align}
where $\mathfrak{M}(\cdot)=\mathfrak{B}(\cdot)$ or $\mathfrak{C}(\cdot)$. For the 2D NSE, the term \eqref{bz1z2diff} with $\mathfrak{M}(\cdot)=\mathfrak{B}(\cdot)$ is estimated in \cite{SSS} by combining the Ladyzhenskaya inequality with the energy bounds available for weak solutions. In three dimensions, however, the energy class no longer provides sufficient regularity to control the trilinear term in \eqref{bz1z2diff}; see \eqref{bz12}. Consequently, the derivation of the required estimate necessitates that $\Z_1$ and $\Z_2$ be strong solutions, which constitutes a genuine difficulty absent in the two-dimensional NSE theory. A further complication arises from the nonlinear absorption term, which introduces an additional contribution corresponding to $\mathfrak{M}(\cdot)=\mathfrak{C}(\cdot)$, a feature that has no analogue in the classical NSE framework. To estimate \eqref{bz1z2diff}, we exploit the embedding $\V\hookrightarrow\widetilde{\mathbb{L}}^{r+1}$ together with the strong-solution energy estimates \eqref{eqncont2}. This strategy is valid for the range of exponents $r$ listed in Table \ref{Table1} (see the proof of \eqref{z12negsob}) and leads to the negative-order estimate \eqref{z12negsob} established in Proposition \ref{supVnorm}. To the best of our knowledge, such estimates for the CBF system have not previously appeared in the literature. For completeness and the convenience of the reader, we provide detailed and self-contained proofs.

The feedback analysis constitutes one of the central components of the present work. A key ingredient in this analysis is the introduction of the auxiliary functional $\mathcal{W}$, defined in \eqref{wfunc} (see Section \ref{feedbackana}), whose differentiability properties play a crucial role in identifying the adjoint variable and deriving the PMP from the HJB equation. While this program was successfully carried out for the 2D NSE in \cite{SSS1,SSS}, the corresponding analysis for the CBF system is substantially more delicate, particularly in the three-dimensional supercritical regime $3<r\leq 5$.
We now highlight some of the specific challenges encountered in the feedback analysis. A crucial step in establishing the differentiability properties of the auxiliary functional $\mathcal{W}$ is the study of the linearized state equation associated with \eqref{stapocp}. In particular, the analysis requires a careful investigation of the following linearization of the state system:
\begin{align*}
-\frac{\d\Psi}{\d t}+\mathcal{L}(\Psi)=0, \ 
\mathcal{L}(\Psi):=\mathcal{A}\Psi+\mathfrak{B}'(\Z)\Psi+\alpha\Psi
+\beta\mathfrak{C}'(\Z)\Psi,
\end{align*}
where $\mathfrak{B}'(\cdot)$ and $\mathfrak{C}'(\cdot)$ denote the Fr\'echet derivatives of the nonlinear operators $\mathfrak{B}(\cdot)$ and $\mathfrak{C}(\cdot)$, respectively, supplemented with an appropriate initial condition (see Section \ref{feedbackana}). Here, $\Z$ denotes the strong solution of the state system \eqref{stapPMP}. A crucial step in the analysis is to prove that $\mathcal{D}_{\z}\mathcal{W}$ belongs to $\V$ and is locally Lipschitz continuous (see Proposition \ref{WpropertyLip}). To achieve this, it is necessary to derive suitable energy estimates together with continuous dependence estimates in the negative Sobolev norm $\|\cdot\|_{\D(\mathcal{A}+\I)^{-1/2}}$. This, in turn, requires estimating the following terms:
\begin{align*}
	(\mathfrak{M}'(\Z)\Psi,(\mathcal{A}+\I)^{-1}\Psi) \ \text{ and } \ 
	(\mathfrak{M}'(\Z_1)\Psi_1-\mathfrak{M}'(\Z_2)\Psi_2,
	(\mathcal{A}+\I)^{-1}\Psi_1-(\mathcal{A}+\I)^{-1}\Psi_2).
\end{align*}
The derivation of these estimates is highly nontrivial. In particular, establishing the negative-order energy estimates requires a delicate treatment of both the convective and absorption nonlinearities through duality arguments involving the operator $(\mathcal{A}+\I)^{-1}$. A complete and systematic analysis of these estimates is carried out in Proposition \ref{phinegatest}. Although analogous estimates are available for the 2D NSE, their extension to the CBF system is far from straightforward owing to the presence of the nonlinear damping term. To the best of our knowledge, such estimates have not previously been developed in a systematic manner in the setting of optimal control and verification theory.

\subsection{Significance and novelties of the present work}
 At first glance, the feedback analysis for the CBF equations might appear to be a straightforward adaptation of the NSE results. However, this is not the case for several reasons:
\vskip 2mm
\noindent
 \textbf{Lack of existing verification results:} At present, no rigorous verification theorem for the CBF system appears to be available in the literature. In particular, a proof that the value function is a viscosity solution of the associated HJB equation and that the corresponding optimal control satisfies the PMP through the HJB framework has not been established. The additional dissipation induced by the Forchheimer term $|\mathfrak{u}|^{r-1}\mathfrak{u}$ enables us to develop this theory not only in two dimensions, but also in the physically relevant three-dimensional supercritical regime $3<r\leq 5$. Consequently, the present work provides the first comprehensive dynamic programming framework, including verification and feedback synthesis, for this class of porous-media flow models.

 \vskip 2mm
 \noindent
 \textbf{Three‑dimensional relevance:} The works \cite{SSS1,SSS} are confined to the 2D NSE, reflecting the absence of a global well-posedness theory for strong solutions in three dimensions. By contrast, for the CBF system with $r$ belonging to the range specified in Table \ref{Table1}, the enhanced dissipation induced by the Forchheimer term guarantees the global existence and uniqueness of strong solutions even in three dimensions. This marks a fundamental distinction between the CBF equations and the 3D NSE.  Indeed, although verification and feedback results are available for the 2D NSE, their extension to the three-dimensional setting is hindered by the longstanding open problem of global regularity. As a consequence, several key ingredients underlying the dynamic programming approach cannot presently be justified globally in time for the 3D NSE. In contrast, the global strong solvability of the CBF system enables us to carry out the feedback analysis and establish the verification theorem in the physically more relevant three-dimensional setting. This constitutes one of the principal novelties and advantages of the present work. 
 
 Moreover, in the existing literature, most optimal control problems have been studied for the two-dimensional CBF equations (see \cite{TiOp2,Op1,Op3,Op2}) and for the three-dimensional critical CBF equations (see \cite{KKFC}). Therefore, the present work significantly extends the optimal control results established in these earlier works.
 
 \vskip 2mm
 \noindent
 \textbf{Detailed handling of the supercritical regime:} The exponent $r\in(3,5]$ is precisely the range in which the CBF system exhibits supercritical damping while still allowing the required Sobolev embeddings. Moreover, the negative‑norm estimates and continuous dependence results (see the results in Sections \ref{abscon} and \ref{feedbackana}) are carefully tuned to this range, and we provide all the algebraic details that are often omitted in the NSE literature. This level of detail is necessary because the CBF nonlinearity does not enjoy the same algebraic simplifications as the NSE.
 \subsection{Organization of the article} 
 The remainder of the paper is organized as follows. In Section \ref{Sec-2}, we introduce the functional framework and recall several properties of the bilinear and Forchheimer nonlinearities. Section \ref{abscon} presents the abstract formulation of the CBF system, reviews the corresponding well-posedness theory, and establishes continuous dependence estimates in both Sobolev and negative Sobolev norms (Proposition \ref{supVnorm}). In Section \ref{LADJsys}, we introduce the linearized and adjoint systems associated with \eqref{stapPMP} and recall their well-posedness properties (Theorems \ref{soLvaLin}--\ref{soLvaadj}). Section \ref{anoptimaL} is devoted to the formulation of the optimal control problem and a discussion of the associated value function, including Bellman's principle of optimality. In Section \ref{feedbackana}, we introduce the auxiliary functional $\mathcal{W}$ defined in \eqref{wfunc} and establish its differentiability and local Lipschitz properties. Finally, Section \ref{Pontryagin} contains the main results of the paper. There, we derive the Pontryagin maximum principle from the viscosity solution property and prove the verification theorem (Theorem \ref{verifmain}), which characterizes the superdifferential of the value function $\mathpzc{V}$ and yields the corresponding optimal feedback law (Theorem \ref{PMPpr}).

	\section{Functional Framework}\label{Sec-2}\setcounter{equation}{0}
In this section, we introduce the functional framework and the linear and nonlinear operators that will be used throughout the paper. The notation and analytical setting adopted here are based on those developed in \cite{JCR1,Te1}.
	\subsection{Function spaces}\label{zerofunc}
	Let us denote by $\C_{\mathrm{p}}^{\infty}(\mathbb{T}^d;\R^d)$, the space of all infinitely differentiable  functions $\mathfrak{u}$ satisfying \eqref{pbc}. 
	The completion of $\C_{\mathrm{p}}^{\infty}(\mathbb{T}^d;\R^d)$  with respect to the $\H^s$-norm is the Sobolev space $\H_{\mathrm{p}}^s(\mathbb{T}^d):=\mathrm{H}_{\mathrm{p}}^s(\mathbb{T}^d;\mathbb{R}^d)$.  From \cite[Proposition 5.39]{JCR1}, the Sobolev space of periodic functions, that is $\H_{\mathrm{p}}^s(\mathbb{T}^d)$, is same as the following: 
	$$\left\{\mathfrak{u}:\mathfrak{u}=\sum_{k\in\mathbb{Z}^d}\mathfrak{u}_{k}\mathrm{e}^{2\pi i k\cdot x},\ \overline{\mathfrak{u}}_{k}=\mathfrak{u}_{-k}, \  \|\mathfrak{u}\|_{{\H}^s_f}:=\left(\sum_{k\in\mathbb{Z}^d}(1+|k|^{2s})|\mathfrak{u}_{k}|^2\right)^{1/2}<\infty\right\}.$$ 
   \begin{remark}
   	We do not impose the zero mean condition on $\mathfrak{u}$ because the absorption term $\beta|\mathfrak{u}|^{r-1}\mathfrak{u}$ does not satisfies this property (see \cite{MTT}). As a result, the standard Poincar\'e inequality is not applicable, and we must work with the full $\H^1$-norm instead.
   \end{remark}
	Let us define 
	\begin{align*} 
		\mathcal{V}:=\{\mathfrak{u}\in\C_{\mathrm{p}}^{\infty}(\mathbb{T}^d;\R^d):\nabla\cdot\mathfrak{u}=0\}.
	\end{align*}
	We denote by $\H$, the closure of $\mathcal{V}$ in the Lebesgue space $\L^2(\mathbb{T}^d):=\mathrm{L}^2(\mathbb{T}^d;\R^d)$ and by $\widetilde{\L}^{p}$, the closure of $\mathcal{V}$ in the Lebesgue space $\L^p(\mathbb{T}^d):=\mathrm{L}^p(\mathbb{T}^d;\R^d)$ for $p\in(2,\infty]$, respectively. We endow the space $\H$ with the inner product and norm of $\L^2(\mathbb{T}^d),$ and are denoted by $(\cdot,\cdot)$ and $\|\cdot\|_{\H}$, respectively. 
	For $p\in(2,\infty)$ and for $p=\infty$, the spaces $\widetilde{\L}^{p}$ and $\widetilde{\L}^{\infty}$ are, respectively, equipped with the $\|\cdot\|_{\widetilde{\L}^p}-$norm of $\L^p(\mathbb{T}^d)$ and with the $\|\cdot\|_{\widetilde{\L}^{\infty}}-$ norm of $\L^{\infty}(\mathbb{T}^d)$.
	
	We also define the space $\V$ as the closure of $\mathcal{V}$ in the Sobolev space $\H^1_{\mathrm{p}}(\mathbb{T}^d)$. We equip the space $\V$ with the inner product $(\cdot,\cdot)_{\V}:=(\cdot,\cdot)+(\nabla\cdot,\nabla\cdot)$ and norm $\|\cdot\|_{\V}^2:=\|\cdot\|_{\H}^2+\|\nabla\cdot\|_{\H}^2$, respectively. 
	Let $\langle \cdot,\cdot\rangle $ denote the duality pairing between the spaces $\V$  and its dual $\V^{*}$, as well as between $\widetilde{\L}^{r+1}$ and its dual $\widetilde{\L}^{\frac{r+1}{r}}$. Note that $\H$ can be identified with its own dual $\H^{*}$. According to \cite[Subsection 2.1]{FKS}, the sum space $\V^{*}+\widetilde{\L}^{\frac{r+1}{r}}$ is well defined and forms a Banach when equipped with the norm 
	\begin{align*}
		\|\mathfrak{u}\|_{\V^{*}+\widetilde{\L}^{\frac{r+1}{r}}}&:=
		\inf\{\|\mathfrak{u}_1\|_{\V^{*}}+\|\mathfrak{u}_2\|_{\wi\L^{\frac{r+1}{r}}}:
		\mathfrak{u}=\mathfrak{u}_1+\mathfrak{u}_2, \mathfrak{u}_1\in\V^{*} \ \text{and} \ \mathfrak{u}_2\in\wi\L^{\frac{r+1}{r}}\}\nonumber\\&=
		\sup\left\{\frac{|\langle\mathfrak{u},\mathfrak{v}
			\rangle|}{\|\mathfrak{v}\|_{\V\cap\widetilde{\L}^{r+1}}}:
			\boldsymbol{0}\neq
		\mathfrak{v}\in\V\cap\widetilde{\L}^{r+1}\right\},
	\end{align*}
	where the norm $\|\cdot\|_{\V\cap\widetilde{\L}^{r+1}}$ on the intersection space $\V\cap\wi\L^{r+1}$ is defined by
	$\|\cdot\|_{\V\cap\widetilde{\L}^{r+1}}:=\max\{\|\cdot\|_{\V}, \|\cdot\|_{\wi\L^{r+1}}\},$ which is equivalent to the norms $\|\cdot\|_{\V}+\|\cdot\|_{\wi\L^{r+1}}$ and $\sqrt{\|\cdot\|_{\V}^2+\|\cdot\|_{\wi\L^{r+1}}^2}$.

	\subsection{Stokes operator}\label{linope}
	 	The Stokes operator is defined by 
	\begin{equation}\label{stokesope}
		\left\{
		\begin{aligned}
	\mathcal{A}\mathfrak{u}&:=-\mathcal{P}\Delta\mathfrak{u}=-\Delta\mathfrak{u},\;\mathfrak{u}\in\D(\mathcal{A}),\\
			\D(\mathcal{A})&:=\V\cap{\H}^{2}_\mathrm{p}(\mathbb{T}^d),
		\end{aligned}
		\right.
	\end{equation}
	where \(\mathcal{P}:\mathbb{L}^{2}(\mathbb{T}^{d})\to\H\) is the Helmholtz--Hodge projection (cf. \cite{DFHM}), that is, the orthogonal projection onto the space \(\H\). By the definition of the $\|\cdot\|_{\H^2_\mathrm{p}}-$norm and an application of Parseval's identity, it follows that the norms $\|\cdot\|_{\H^2_\mathrm{p}}$ and $\|\cdot\|_{\H}+\|\mathcal{A}\cdot\|_{\H}$ are equivalent. Moreover,  $\D(\I+\mathcal{A})=\H^2_\mathrm{p}(\mathbb{T}^d)\cap\V=:\V_2$.

	\subsection{Bilinear operator}\label{bilinope}
		Let us define the trilinear form $b(\cdot,\cdot,\cdot):\V\times\V\times\V\to\R$ by
		\begin{align*}
			b(\mathfrak{u},\mathfrak{v},\mathfrak{w})=
			\int_{\mathbb{T}^d}(\mathfrak{u}(x)\cdot\nabla)\mathfrak{v}(x)\cdot
			\mathfrak{w}(x)\d x.
		\end{align*} 
		It is known that $b(\cdot,\cdot,\cdot)$ is continuous (see \cite{Te1}) and satisfies
		\begin{align}\label{skewsym}
		b(\mathfrak{u},\mathfrak{v},\mathfrak{w})=-
		b(\mathfrak{u},\mathfrak{w},\mathfrak{v}), \ \text{ for all } \ \mathfrak{u}\in\V, \ \mathfrak{v}, \mathfrak{w}\in\H^1_{\mathrm{p}}(\mathbb{T}^d).
		\end{align} 
		We define a bilinear map $\mathfrak{B}(\cdot,\cdot):\V\times\V\to\R$ such that $$\langle\mathfrak{B}(\mathfrak{u},\mathfrak{v}),\mathfrak{w}\rangle=b(\mathfrak{u},\mathfrak{v},\mathfrak{w}) \ \text{ for all } \ \mathfrak{u},\mathfrak{v},\mathfrak{w}\in\V.$$ 
		A direct consequence of \eqref{skewsym} is the following identity
		\begin{align}\label{syymB}
			\langle\mathfrak{B}(\mathfrak{u},\mathfrak{v}),\mathfrak{v}\rangle=0,
			\ \text{ for any } \ \mathfrak{u},\mathfrak{v}\in\V.
		\end{align}

		\begin{remark}\label{BLrem}
		In view of \eqref{syymB}, along with H\"older's and Young's inequalities, the bilinear operator $\mathfrak{B}(\cdot)$ satisfies the following inequalities, for $r>3$ (see \cite[Theorem 2.5]{SMTM}):
			\begin{align}
				|\langle\mathfrak{B}(\mathfrak{u})-\mathfrak{B}(\mathfrak{v}),
				\mathfrak{u}-\mathfrak{v}\rangle|&\leq
				\frac{\mu }{2}\|\nabla(\mathfrak{u}-\mathfrak{v})\|_{\H}^2 +\frac{\beta}{4}\||\mathfrak{v}|^{\frac{r-1}{2}}(\mathfrak{u}-\mathfrak{v})\|_{\H}^2 +\varrho\|\mathfrak{u}-\mathfrak{v}\|_{\H}^2,\label{3.4}\\
		    \text{ and } \	|(\mathfrak{B}(\mathfrak{u}),\mathcal{A}\mathfrak{u})|&\leq\frac{\mu}{2}\|\mathcal{A}\mathfrak{u}\|_{\H}^2+\frac{\beta}{4} \||\mathfrak{u}|^{\frac{r-1}{2}}\nabla\mathfrak{u}\|_{\H}^2+\varrho\|\nabla\mathfrak{u}\|_{\H}^2,\label{syymB3}
			\end{align}
			where \begin{align}\label{eqn-varrho}
				\varrho:=\frac{r-3}{2\mu(r-1)}\left[\frac{4}{\beta\mu (r-1)}\right]^{\frac{2}{r-3}}.
				\end{align}
		\end{remark}

	 \subsection{Nonlinear operator}\label{nonlinope}
		We define the operator $$\mathfrak{C}(\mathfrak{u}):=\mathcal{P}(|\mathfrak{u}|^{r-1}\mathfrak{u})\ \text{ for }\ \mathfrak{u}\in\V\cap\L^{r+1}.$$  From \cite[Remark 1.6]{Te}, the operator $\mathfrak{C}(\cdot):\V\cap\widetilde{\L}^{r+1}\to\V^{*}+\widetilde{\L}^{\frac{r+1}{r}}$ is well-defined. Moreover, it satisfies the identity $\langle\mathfrak{C}(\mathfrak{u}),\mathfrak{u}\rangle =\|\mathfrak{u}\|_{\widetilde{\L}^{r+1}}^{r+1}.$ Furthermore, for all $\mathfrak{u}\in\V\cap\L^{r+1}$, the map $\mathfrak{C}(\cdot):\V\cap\widetilde{\L}^{r+1}\to\V^{*}+\widetilde{\L}^{\frac{r+1}{r}}$ is Gateaux differentiable with Gateaux derivative given by 
		\begin{align}\label{C1d}
			\mathfrak{C}'(\mathfrak{u})\mathfrak{v}&=\left\{\begin{array}{cl}\mathcal{P}(\mathfrak{v}),&\text{ for }r=1,\\ \left\{\begin{array}{cc}\mathcal{P}(|\mathfrak{u}|^{r-1}\mathfrak{v})+(r-1)\mathcal{P}\left(\frac{\mathfrak{u}}{|\mathfrak{u}|^{3-r}}(\mathfrak{u}\cdot\mathfrak{v})\right),&\text{ if }\mathfrak{u}\neq \mathbf{0},\\\mathbf{0},&\text{ if }\mathfrak{u}=\mathbf{0},\end{array}\right.&\text{ for } 1<r<3,\\ \mathcal{P}(|\mathfrak{u}|^{r-1}\mathfrak{v})+(r-1)\mathcal{P}(\mathfrak{u}|\mathfrak{u}|^{r-3}(\mathfrak{u}\cdot\mathfrak{v})), &\text{ for }r\geq 3,\end{array}\right.
		\end{align}
		for all $\mathfrak{v}\in\V\cap\widetilde{\L}^{r+1}$. 
%		For $r\geq3$, we also have the existence of second order Gateaux derivative given by 
%		\begin{align}\label{C2d}
%			&\mathfrak{C}''(\mathfrak{u})(\mathfrak{v}\otimes\mathfrak{w})
%			\nonumber\\ &=\left\{\begin{array}{cl} \left\{\begin{array}{cc}(r-1)\mathcal{P}\{|\mathfrak{u}|^{r-3}
%					\left(\mathfrak{u}\cdot\mathfrak{v})\mathfrak{v}+(\mathfrak{u}\cdot\mathfrak{v})\mathfrak{v}+(\mathfrak{v}\cdot\mathfrak{v})\mathfrak{u}\right)\},&\text{ if }\mathfrak{u}\neq \mathbf{0},\\\mathbf{0},&\text{ if }\mathfrak{u}=\mathbf{0},
%					\end{array}\right.&\text{ for } 3<r<5,\\ (r-1)\mathcal{P}\{|\mathfrak{u}|^{r-3}\left(\mathfrak{u}\cdot\mathfrak{v})\mathfrak{v}+(\mathfrak{u}\cdot\mathfrak{v})\mathfrak{v}+(\mathfrak{v}\cdot\mathfrak{v})\mathfrak{u}\right)\}\\
%				+(r-1)(r-3)\mathcal{P}\{|\mathfrak{u}|^{r-5}(\mathfrak{u}\cdot\mathfrak{v})(\mathfrak{u}\cdot\mathfrak{v})\mathfrak{u}\}, &\text{ for }r\geq5,\end{array}\right.
%		\end{align}
%		for all $\mathfrak{u},\mathfrak{v},\mathfrak{w}\in\V\cap{\wi\L}^{r+1}$.
\begin{remark}
	1.) (\textbf{Monotonicity of} $\mathfrak{C}(\cdot)$) For every $r\geq1$ and for all $\mathfrak{u},\mathfrak{v}\in\widetilde{\L}^{r+1}$, we have following estimate ({cf. \cite[Subsection 2.4]{SMTM}}):
	\begin{align}\label{monoC2}
		\langle\mathfrak{C}(\mathfrak{u})-\mathfrak{C}(\mathfrak{v}),\mathfrak{u}-\mathfrak{v}\rangle\geq \frac{1}{2}\||\mathfrak{u}|^{\frac{r-1}{2}}(\mathfrak{u}-\mathfrak{v})\|_{\H}^2+\frac{1}{2}\||\mathfrak{v}|^{\frac{r-1}{2}}(\mathfrak{u}-\mathfrak{v})\|_{\H}^2\geq\frac{1}{2^{r-1}}\|\mathfrak{u}-\mathfrak{v}\|_{\wi\L^{r+1}}^{r+1}.
	\end{align}
	
	2.) The following equality holds on $\mathbb{T}^d$ (see 
	\cite[Lemma 2.1]{KWH}):
	\begin{align}\label{torusequ}
		(\mathfrak{C}(\mathfrak{u}),\mathcal{A}\mathfrak{u})
		=\||\mathfrak{u}|^{\frac{r-1}{2}}\nabla\mathfrak{u}\|_{\H}^{2} +4\left[\frac{r-1}{(r+1)^2}\right]
		\|\nabla|\mathfrak{u}|^{\frac{r+1}{2}}\|_{\H}^{2}.
	\end{align}
\end{remark}

\begin{remark}[Frequently used inequalities] Throughout this article, we frequently used the Sobolev embedding, Agmon and Ladyzhenskaya inequalities. To avoid  repeated references, we fix constant as follows:
	
	1.) Let $C_e>0$ denote a generic constant associated with the Sobolev embedding $\V\hookrightarrow\wi\L^{r+1}$ namely
	\begin{align*}
		\|\cdot\|_{\L^{r+1}}\leq C_e\|\cdot\|_{\V},
	\end{align*}
	where $r\geq1$ in $d=2$ and $r\in[1,5]$ in $d=3$. 
	
	2.) Let $C_a>0$ denote the constant appearing in Agmon's inequality
	\begin{align*}
		\|\cdot\|_{\wi\L^{\infty}}\leq C_a\|\cdot\|_{\H}^{1-\frac{d}{4}} \|\cdot\|_{\H^2_\mathrm{p}(\mathbb{T}^d)}^{\frac{d}{4}}, \ \ d\in\{2,3\}.
	\end{align*} 
	
	3.) Let $C_{\ell}$ denote the constant appearing in the Ladyzhenskaya inequality:
	\begin{align*}
		\|\cdot\|_{\wi\L^4}\leq C_{\ell}\|\cdot\|_{\V}^{\frac{d}{4}}
		\|\cdot\|_{\H}^{1-\frac{d}{4}},
	\end{align*}
	where $C_{\ell}:=2^{\frac{d-1}{4}}$, $d\in\{2,3\}$. For convenience, we shall use the generic symbol $C$ to denote any of the above constants whenever no confusion can arise.
	
	4.) Unless explicitly stated otherwise, $C$ denotes a generic positive constant whose value may vary from line to line. Such constants typically arise from applications of H\"older's inequality, Young's inequality, interpolation inequalities, and embedding results, and their precise values are immaterial for the subsequent analysis.
\end{remark}

 \begin{definition}[\cite{SSS2}]\label{superdff}
 	Let $\mathbb{X}$ be a separable Hilbert space and let $\mathpzc{f}:\mathbb{X}\to\R$ be a locally Lipschitz function. The superdifferential of $\mathpzc{f}$ at $\x\in\mathbb{X}$ is the subset of $\mathbb{X}^{*}$
 	\begin{align*}
 	\mathcal{D}^{+}\mathpzc{f}(\x):=
 	\left\{\boldsymbol{\zeta}\in\mathbb{X}^{*}:\limsup\limits_{\y\to\x}\left[
 	\frac{\mathpzc{f}(\y)-\mathpzc{f}(\x)-\langle\boldsymbol{\zeta},
 		\y-\x\rangle}{\|\y-\x\|_{\mathbb{X}}}
 	\right]
 	\leq0\right\}.
 	\end{align*}
Equivalently, $\boldsymbol{\zeta}\in\mathbb{X}^{*}$ belongs to $\mathcal{D}^{+}\mathpzc{f}(\x)$, $\x\in\V$, if
 	\begin{align*}
 	\mathpzc{f}(\x+\boldsymbol{h})-\mathpzc{f}(\x)\leq\langle\boldsymbol{\zeta},
 	\boldsymbol{h}\rangle+
 	o(\|\boldsymbol{h}\|_{\mathbb{X}})  \ \text{ as } \ \|\boldsymbol{h}\|_{\mathbb{X}}\to0,
 	\end{align*}
 	for all $\boldsymbol{h}\in\mathbb{X}$. 
 \end{definition}
 
	\section{Abstract formulation and Auxiliary results}\label{abscon} \setcounter{equation}{0}
	We begin this section by presenting an abstract formulation of the system \eqref{CBF} and establishing uniform energy estimates, followed by proofs of continuous dependence results. 
   Let us set $\Z(\cdot):=\mathcal{P}\mathfrak{u}(\cdot)$, $\mathcal{P}\mathfrak{u}_0:=\z$,  $\mathcal{P}\mathfrak{g}=\f$, $\mathcal{P}\mathfrak{a}=\u$. On projecting the first equation in \eqref{CBF}, we obtain the following abstract controlled CBF equations:
	\begin{equation}\label{stapPMP}
		\left\{
		\begin{aligned}
			\frac{\d\Z(s)}{\d s}&= -\mu\mathcal{A}\Z(s)-\mathfrak{B}(\Z(s))-\alpha\Z(s)-
			\beta\mathfrak{C}(\Z(s))+
			\f(s)+\u(s),  \  \text{ in } \ (t,T)\times\H, \\
			\Z(t)&=\z\in\H.
		\end{aligned}
		\right.
	\end{equation}
	Let us recall some well-posedness results for the system \eqref{stapPMP}. For the purpose of this analysis, we assume that the control $\u$ belongs to $\mathrm{L}^2(t,T;\H)$. A precise definition of the admissible control class and the control set will be introduced later in the optimal control section (Section \ref{anoptimaL}). At this stage this assumption is sufficient from the viewpoint of well-posedness. 
	\begin{definition}\label{def-var-strong} Let $\u\in\mathrm{L}^2(t,T;\H)$. A function $$\Z(\cdot)\in\mathrm{C}([t,T];\H)\cap\mathrm{L}^2(t,T;\V)\cap\mathrm{L}^{r+1}(t,T;\wi\L^{r+1}),$$
		with $\frac{\d\Z}{\d s}\in\mathrm{L}^2(t,T;\V^{*})+ \mathrm{L}^{\frac{r+1}{r}}(t,T;\wi\L^{\frac{r+1}{r}})$ is called a \emph{weak solution} to \eqref{stapPMP}, 
%		if for a given $\f(\cdot)\in\mathrm{L}^2(t,T;\V^{*})$ and $\Z(t)=\z\in\H$, we have 
%		$$\sup_{s\in[t,T]}\|\Z(s)\|_{\H}^2+\int_t^T \|\Z(s)\|_{\V}^2\d s + \int_t^T\|\Z(s)\|_{\wi\L^{r+1}}^{r+1}\d s<+\infty,$$ and 
		if the function $\Z(\cdot)$ satisfies the following weak formulation:
		\begin{align*}
			(\Z(s),\boldsymbol{\phi})&=(\z,\boldsymbol{\phi})+\int_t^s  \langle-\mu\mathcal{A}\Z(\tau)-\mathfrak{B}(\Z(\tau))-\alpha\Z(\tau)-
			\beta\mathfrak{C}(\Z(\tau))+\f(\tau)+\u(\tau),\boldsymbol{\phi}\rangle\d\tau,
		\end{align*} 
		  for every $\boldsymbol{\phi}\in\V\cap\wi\L^{r+1}$ and every $s\in[t,T]$.
	    Furthermore, if $\f\in\mathrm{L}^2(t,T;\H)$ and $\z\in\V$, then a function 
		$$\Z(\cdot)\in\mathrm{C}([t,T];\V)\cap\mathrm{L}^2(t,T;\D(\mathcal{A}))\cap
		\mathrm{L}^{r+1}(t,T;\wi\L^{p(r+1)})$$ is called a \emph{strong solution} of \eqref{stapPMP} with $\Z(t)=\z\in\V$ if $\Z(\cdot)$ is a weak solution and  
		\begin{align*}
			\Z(s)&=\z+\int_t^s  \left[-\mu\mathcal{A}\Z(\tau)-\mathfrak{B}(\Z(\tau))-\alpha\Z(\tau)-
			\beta\mathfrak{C}(\Z(\tau))+\f(\tau)+\u(\tau)\right]\d\tau,
		\end{align*} 
		 is satisfied as an equality in $\H$ for all  $s\in[t,T]$.
%		and satisfies $$\sup_{s\in[t,T]}\|\Z(s)\|_{\V}^2+\int_t^T \|\mathcal{A}\Z(s)\|_{\H}^2\d s+\int_0^T\|\Z(s)\|_{\wi\L^{p(r+1)}}^{r+1}\d s<+\infty,$$  where $p\in[2,\infty)$ for $d=2$ and $p=3$ for $d=3$.
	%	Moreover, \eqref{stapPMP} is satisfied as an equality in $\H$ for a.e. $s\in[t,T]$.
	\end{definition}
    For a given time interval $s\in[t, T]$, a control $\u(\cdot)$, and an initial datum $\z$, we denote by $\Z(\cdot;t,\z,\u)$, the (weak or strong) solution of \eqref{stapPMP} . For brevity, we simply write $\Z(\cdot)$ when no ambiguity arises. 
    We now state the following well-posedness results for weak and strong solutions of the system \eqref{stapPMP}, which are adopted from \cite[Theorem 3.5 and Theorem 4.2]{SMTM}.
     \begin{proposition}\label{weLLp}
      Let $t\in[0,T]$ and $\u\in\mathrm{L}^2(t,T;\H)$. Assume that $r\geq1$ in $d=2$, and $r>3$ and $r=3$ with $2\beta\mu\geq1$ in $d=3$. Then, for a given initial data $\z\in\H$ and forcing $\f\in\mathrm{L}^2(t,T;\V^{*})$, there exists a \emph{unique weak solution} $\Z(\cdot)$ of the system \eqref{stapPMP}, in the sense of Definition \ref{def-var-strong}, satisfying the following energy estimate: 
      	\begin{align}\label{eqncont1}
      	&\sup\limits_{s\in[t,T]}\|\Z(s)\|_{\H}^2
      	+\int_t^{T}\|\Z(\zeta)\|_{\V}^2\d\zeta+
      	\int_t^{T}\|\Z(\zeta)\|_{\widetilde{\L}^{r+1}}^{r+1}\d\zeta
      	\nonumber\\&\leq C(\mu,\beta,T)
      	\left(\|\z\|_{\H}^2+\int_t^T\|\f(\zeta)\|_{\V^{*}}^2\d\zeta+
      	\int_t^T\|\u(\zeta))\|_{\H}^2\d \zeta\right).
      \end{align}
      Moreover, for $\f\in\mathrm{L}^2(t,T;\H)$ and $\z\in\V$, there exists a \emph{unique strong solution} $\Z(\cdot)$ of the system \eqref{stapPMP} for which the following uniform energy estimate holds:
     	 \begin{align}\label{eqncont2}
     		&\sup\limits_{s\in[t,T]}\|\Z(s)\|_{\V}^2
     		+\int_t^{T}\|(\mathcal{A}+\I)\Z(\zeta)\|_{\H}^2\d\zeta+
     		\int_t^T\|\Z(\zeta)\|_{\wi\L^{r+1}}^{r+1}\d\zeta+
     		\int_t^{T}\||\Z(\zeta)|^{\frac{r-1}{2}}\nabla\Z(\zeta)\|_{\H}^{2}\d\zeta
     		\nonumber\\&\leq C(\mu,\alpha,\beta,T)
     		\left(\|\z\|_{\V}^2+\int_t^T\|\f(\zeta)\|_{\H}^2\d\zeta 
     		+\int_t^T\|\u(\zeta)\|_{\H}^2\d\zeta\right).
     	\end{align}
     \end{proposition}
     The forcing term $\f$ does not play any essential role in the analysis, we assume without loss of generality that $\f\equiv\boldsymbol{0}$. Hence, the control $\u$ represents the only forcing in the system \eqref{stapPMP}.
     The following result provides a key estimate, namely a supremum bound for the difference of two solutions in the $\V$-norm, and it will be essential in the derivation of the Pontryagin maximum principle. Moreover, we also establish the continuous dependence estimate for the solution of the system \eqref{stapPMP} in negative Sobolev norms, which will be useful to study the continuity properties of the value function.
     \begin{proposition}\label{supVnorm}
      Let $\Z_1(\cdot)$ and $\Z_2(\cdot)$ be two strong solutions of the system \eqref{stapPMP} such that $\Z_1(t)=\z_1\in\V$ and $\Z_2(t)=\z_2\in\V$. Then, for the values $r$ given in Table \ref{Table1}, the following bound holds:
       \begin{align}\label{supVnormest}
      	\sup\limits_{s\in[t,T]}\|(\Z_1-\Z_2)(s)\|_{\V}^2+\int_t^T
      	\|(\mathcal{A}+\I)(\Z_1-\Z_2)(\zeta)\|_{\H}^2\,\d \zeta
      	 \leq
      	C\|\z_1-\z_2\|_{\V}^2.
      \end{align}
      Moreover, the following continuous dependence estimate for the solution of the system \eqref{stapPMP} in negative Sobolev norm holds:
      \begin{align}\label{z12negsob}
      	\sup\limits_{s\in[t,T]}
      	\|(\mathcal{A}+\I)^{-\frac12}(\Z_1-\Z_2)(s)\|_{\H}^2+
      	\int_{t}^T \|(\Z_1-\Z_2)(\zeta)\|_{\H}^2\d\zeta\leq 
      	C\|(\mathcal{A}+\I)^{-\frac12}(\z_1-\z_2)\|_{\H}^2,
      \end{align}
       where  $C=C\big(\mu,\alpha,\beta,T,\|\z_1\|_{\V},\|\z_2\|_{\V},\|\u\|_{\mathrm{L}^2(t,T;\H)}\big)$.
     \end{proposition}

    \begin{proof}
    Let us set $\mathfrak{Z}(\cdot):=\Z_1(\cdot)-\Z_2(\cdot)$. Then, from \eqref{stapPMP}, we infer that  $\mathfrak{Z}(\cdot)$ satisfies the following system:
    \begin{equation}\label{redstapdff1}
    	\left\{
    	\begin{aligned}
    		\frac{\d\mathfrak{Z}(s)}{\d s}+\mu\mathcal{A}\mathfrak{Z}(s)+ \mathfrak{B}(\Z_1(s))&-\mathfrak{B}(\Z_2(s))+
    		\alpha\mathfrak{Z}(s)\\+\beta\mathfrak{C}(\Z_1(s))-\beta\mathfrak{C}(\Z_2(s))&= \boldsymbol{0},  \  \text{ for a.e. } \ s\in[t,T],\\
    		\mathfrak{Z}(t)&=\z_1-\z_2\in\V.
    	\end{aligned}
    	\right.
    \end{equation}
    The rest of the proof is divided into the following two steps: 
    \vskip 2mm
    \noindent
    \textbf{Step I: Proof of \eqref{supVnormest}.} By taking the inner product of the first equation of \eqref{redstapdff1} with $(\mathcal{A}+\I)\mathfrak{Z}$ and using the absolute continuity of $t\mapsto\|\mathfrak{Z}(t)\|_{\V}^2$, we obtain
    \begin{align}\label{unqA}
    	&\frac12\frac{\d}{\d s}\|\mathfrak{Z}\|_{\V}^2+
    	\mu\|(\mathcal{A}+\I)\mathfrak{Z}\|_{\H}^2+\alpha\|\mathfrak{Z}\|_{\V}^2
    	\nonumber\\&=
    	\mu\|\mathfrak{Z}\|_{\V}^2
    	-\big(\mathfrak{B}(\Z_1)-\mathfrak{B}(\Z_2),
    	(\mathcal{A}+\I)\mathfrak{Z}\big)-\beta \big(\mathfrak{C}(\Z_1)-\mathfrak{C}(\Z_2),
    	(\mathcal{A}+\I)\mathfrak{Z}\big).
    \end{align}
    Let us first estimate the term containing bilinear operator in the right hand side of \eqref{unqA} . By an application of the Cauchy--Schwarz and Agmon inequalities, we calculate
    \begin{align*}
    	&\big(\mathfrak{B}(\Z_1)-\mathfrak{B}(\Z_2),
    	(\mathcal{A}+\I)\mathfrak{Z}\big)
    	\nonumber\\&=
    	\big(\mathfrak{B}(\mathfrak{Z},\Z_1),(\mathcal{A}+\I)\mathfrak{Z}\big)+
    	\big(\mathfrak{B}(\Z_2,\mathfrak{Z}),(\mathcal{A}+\I)\mathfrak{Z}\big)
    	\nonumber\\&\leq
    	\|\mathfrak{B}(\mathfrak{Z},\Z_1)\|_{\H}\|(\mathcal{A}+\I)\mathfrak{Z}\|_{\H}+
    	\|\mathfrak{B}(\Z_2,\mathfrak{Z})\|_{\H}\|(\mathcal{A}+\I)\mathfrak{Z}\|_{\H}
    	\nonumber\\&\leq
    	C\|\mathfrak{Z}\|_{\H}^{1-\frac{d}{4}}
    	\|(\mathcal{A}+\I)\mathfrak{Z}\|_{\H}^{1+\frac{d}{4}}
    	\|\nabla\Z_1\|_{\H}+C
    	\|\Z_2\|_{\H}^{1-\frac{d}{4}}\|(\mathcal{A}+\I)\Z_2\|_{\H}^{\frac{d}{4}}
    	\|\nabla\mathfrak{Z}\|_{\H}\|(\mathcal{A}+\I)\mathfrak{Z}\|_{\H}.
    \end{align*}
    On integrating the above inequality  from $t$ to $T$, and then applying H\"older's and Young's inequalities, we find
    \begin{align}\label{unqA1}
    	&\int_t^T \big(\mathfrak{B}(\Z_1(\zeta))-\mathfrak{B}(\Z_2(\zeta)),
    	(\mathcal{A}+\I)\mathfrak{Z}(\zeta)\big)\,\d \zeta
    	\nonumber\\&\leq
    	C\|\Z_1\|_{\mathrm{L}^{\infty}(t,T;\V)}
    	\|\mathfrak{Z}\|_{\mathrm{L}^2(t,T;\V)}^{1-\frac{d}{4}}
    	\|(\mathcal{A}+\I)\mathfrak{Z}\|_{\mathrm{L}^2(t,T;\H)}^{1+\frac{d}{4}}
    	\nonumber\\&\quad+
    	C\|\Z_2\|_{\mathrm{L}^{\infty}(t,T;\H)}^{1-\frac{d}{4}}
    	\left(\int_t^T \|(\mathcal{A}+\I)\Z_2(\zeta)\|_{\H}^{\frac{d}{2}}\|\mathfrak{Z}(\zeta)\|_{\V}^2\,
    	\d \zeta \right)^{\frac12}
    	\|(\mathcal{A}+\I)\mathfrak{Z}\|_{\mathrm{L}^2(t,T;\H)}
    	\nonumber\\&\leq
    	C\|\Z_1\|_{\mathrm{L}^{\infty}(t,T;\V)}^{\frac{8}{4-d}}
    	\|\mathfrak{Z}\|_{\mathrm{L}^2(t,T;\V)}^2+
    	\frac{\mu}{4}\|(\mathcal{A}+\I)\mathfrak{Z}\|_{\mathrm{L}^2(t,T;\H)}^2
    	\nonumber\\&\quad+
    	C\|\Z_2\|_{\mathrm{L}^{\infty}(t,T;\H)}^{\frac{4-d}{2}}
    	\int_t^T \|(\mathcal{A}+\I)\Z_2(\zeta)\|_{\H}^{\frac{d}{2}}
    	\|\mathfrak{Z}(\zeta)\|_{\V}^2\,\d \zeta.
    \end{align}
%    where
%    \begin{align*}
%    C_1=\left(\frac{4+d}{\mu}\right)^{\frac{4+d}{4-d}}\left(\frac{4-d}{8}\right)
%    C_a^{\frac{8}{4-d}} \ \text{ and } \ C_2=\frac{2C_a^2}{\mu}.
%    \end{align*}
   An application of Taylor's formula yields
    \begin{align}\label{Ctayest}
    \mathfrak{C}(\Z_1)-\mathfrak{C}(\Z_2)&=
    	\int_0^1\mathfrak{C}'(\theta\Z_1+(1-\theta)\Z_2)\mathfrak{Z}\,\d\theta
    	\nonumber\\&=
    	\int_0^1 |\theta\Z_1+(1-\theta)\Z_2|^{r-1}\mathfrak{Z}\,\d\theta
    	\nonumber\\&\quad+
    	(r-1)\int_0^1 (\theta\Z_1+(1-\theta)\Z_2)|\theta\Z_1+(1-\theta)\Z_2|^{r-3}
    	(\theta\Z_1+(1-\theta)\Z_2)\cdot\mathfrak{Z}\,\d\theta
    	\nonumber\\&\leq
    	r(|\Z_1|+|\Z_2|)^{r-1}|\mathfrak{Z}|.
    \end{align}
    By using \eqref{Ctayest} and making the use of the Cauchy--Schwarz and H\"older inequalities, we obtain
    \begin{align*}
    	\big(\mathfrak{C}(\Z_1)-\mathfrak{C}(\Z_2),
    	(\mathcal{A}+\I)\mathfrak{Z}\big)
    	&\leq 
    	r\||\Z_1|+|\Z_2|\|_{\wi\L^{3(r-1)}}^{r-1}\|\mathfrak{Z}\|_{\wi\L^{6}}
    	\|(\mathcal{A}+\I)\mathfrak{Z}\|_{\H}
    	\nonumber\\&\leq
    	rC\||\Z_1|+|\Z_2|\|_{\wi\L^{3(r-1)}}^{r-1}\|\mathfrak{Z}\|_{\V}
    	\|(\mathcal{A}+\I)\mathfrak{Z}\|_{\H}.
    \end{align*}
   On integrating above from $t$ to $T$, and then applying H\"older's and Young's inequalities, we derive 
    \begin{align}\label{unqA2}
    	&\beta\int_t^T \big(\mathfrak{C}(\Z_1(\zeta))-\mathfrak{C}(\Z_2(\zeta)),
    	(\mathcal{A}+\I)\mathfrak{Z}(\zeta)\big)\,\d \zeta
    	\nonumber\\&\leq
    	rC\beta\left(\int_t^T\||\Z_1(\zeta)|+|\Z_2(\zeta)|\|_{\wi\L^{3(r-1)}}^{2(r-1)}
    	\|\mathfrak{Z}(\zeta)\|_{\V}^2\,\d \zeta\right)^{\frac12}
    	\|(\mathcal{A}+\I)\mathfrak{Z}\|_{\mathrm{L}^2(t,T;\H)}
    	\nonumber\\&\leq
    	C\int_t^T\||\Z_1(\zeta)|+|\Z_2(\zeta)|\|_{\wi\L^{3(r-1)}}^{2(r-1)}
    	\|\mathfrak{Z}(\zeta)\|_{\V}^2\,\d \zeta+
    	\frac{\mu}{4}\|(\mathcal{A}+\I)\mathfrak{Z}\|_{\mathrm{L}^2(t,T;\H)}^2.
    \end{align}
 On integrating \eqref{unqA} from $t$ to $s$ and utilizing the estimates \eqref{unqA1} and \eqref{unqA2}, and rearranging the terms, we deduce the following:
    \begin{align*}
    &\|\mathfrak{Z}(s)\|_{\V}^2+\mu\int_t^s
    \|(\mathcal{A}+\I)\mathfrak{Z}(\zeta)\|_{\H}^2\,\d \zeta+2\alpha\int_t^s \|\mathfrak{Z}(\zeta)\|_{\V}^2\,\d \zeta
    \nonumber\\&\leq
    \|\z_1-\z_2\|_{\V}^2+\int_t^s \bigg[2\mu+C\|\Z_1\|_{\mathrm{L}^{\infty}(t,T;\V)}^{\frac{8}{4-d}}
    +C\|\Z_2\|_{\mathrm{L}^{\infty}(t,T;\H)}^{\frac{4-d}{2}}
    \|(\mathcal{A}+\I)\Z_2(\zeta)\|_{\H}^{\frac{d}{2}}
    \nonumber\\&\hspace{4cm}+
    C\||\Z_1(\zeta)|+|\Z_2(\zeta)|\|_{\wi\L^{3(r-1)}}^{2(r-1)}
    \bigg]\|\mathfrak{Z}(\zeta)\|_{\V}^2\,\d \zeta,
    \end{align*}
    for all $s\in[t,T]$.  By an application of Gr\"onwall's inequality, we obtain
    \begin{align}\label{unqA3}
    	&\|\mathfrak{Z}(s)\|_{\V}^2+\mu\int_t^s
    	\|(\mathcal{A}+\I)\mathfrak{Z}(\zeta)\|_{\H}^2\,\d \zeta+2\alpha\int_t^s \|\mathfrak{Z}(\zeta)\|_{\V}^2\,\d \zeta
    	\nonumber\\&\leq
    	\|\z_1-\z_2\|_{\V}^2\,\exp\bigg(\int_t^s \bigg[2\mu+C\|\Z_1\|_{\mathrm{L}^{\infty}(t,T;\V)}^{\frac{8}{4-d}}
    	+C\|\Z_2\|_{\mathrm{L}^{\infty}(t,T;\H)}^{\frac{4-d}{2}}
    	\|(\mathcal{A}+\I)\Z_2(\zeta)\|_{\H}^{\frac{d}{2}}
    	\nonumber\\&\hspace{4cm}+
    	C\||\Z_1(\zeta)|+|\Z_2(\zeta)|\|_{\wi\L^{3(r-1)}}^{2(r-1)}
    	\bigg]\,\d \zeta\bigg),
    \end{align}
    for all $s\in[t,T]$. Now, we need to justify the integral quantity appearing in the exponential term in \eqref{unqA3}. Note that, by H\"older' inequality, the following bound follows immediately:
    \begin{align}\label{unqA7}
    	\int_t^T \|(\mathcal{A}+\I)\Z_2(\zeta)\|_{\H}^{\frac{d}{2}}\,\d\zeta
    	\leq T^{\frac{4-d}{4}} \|(\mathcal{A}+\I)\mathfrak{Z}\|_{\mathrm{L}^2(t,T;\H)}^\frac{d}{2}.
    \end{align}
    Therefore, we only need to justify the term 
    $\int_t^T \||\Z_1(\zeta)|+|\Z_2(\zeta)|\|_{\wi\L^{3(r-1)}}^{2(r-1)}\,\d\zeta$. We consider the following three cases:
    \vskip 2mm
    \noindent
    \textbf{Case I:} \emph{For $r\geq1$ in $d=2$.} In view of the Sobolev embedding $\V\hookrightarrow\wi\L^{r+1}$ for any $r\geq1$, the following bound is immediate:
    \begin{align*}
    \int_t^T \||\Z_1(\zeta)|+|\Z_2(\zeta)|\|_{\wi\L^{3(r-1)}}^{2(r-1)}\,\d\zeta
    \leq 2^{2r-3}C \big[\|\Z_1\|_{\mathrm{L}^{\infty}(t,T;\V)}^{2(r-1)}+
    \|\Z_2\|_{\mathrm{L}^{\infty}(t,T;\V)}^{2(r-1)}\big]T.
    \end{align*} 
    \vskip 2mm
    \noindent
    \textbf{Case II:} \emph{For $r\in(3,5)$ and $r=3$ with $2\beta\mu\geq1$ in $d=3$.}
    For $r\in(3,5)$, by an application of the interpolation and H\"older inequalities, we estimate 
    \begin{align}
    	&\int_t^T \||\Z_1(\zeta)|+|\Z_2(\zeta)|\|_{\wi\L^{3(r-1)}}^{2(r-1)}\,\d\zeta
    	\nonumber\\&\leq
    	\int_t^T  \||\Z_1(\zeta)|+|\Z_2(\zeta)|\|_{\wi\L^{r+1}}^{2} \||\Z_1(\zeta)|+|\Z_2(\zeta)|\|_{\wi\L^{3(r+1)}}^{2(r-2)}\,\d\zeta
    	\label{unqA4}\\&\leq
    	\left(\int_t^T  \||\Z_1(\zeta)|+|\Z_2(\zeta)|\|_{\wi\L^{r+1}}^{\frac{2(r+1)}{5-r}}
    	\,\d\zeta\right)^{\frac{5-r}{r+1}} \left(\int_t^T\||\Z_1(\zeta)|+|\Z_2(\zeta)|\|_{\wi\L^{3(r+1)}}^{r+1}
    	\,\d\zeta\right)^{\frac{2(r-2)}{r+1}}
    	\nonumber\\&\leq 2^{\frac{3(r-1)}{r+1}}
    	\left(\int_t^T  \|\Z_1(\zeta)\|_{\wi\L^{r+1}}^{\frac{2(r+1)}{5-r}}
    	\,\d\zeta+
    	\int_t^T  \|\Z_2(\zeta)\|_{\wi\L^{r+1}}^{\frac{2(r+1)}{5-r}}
    	\,\d\zeta\right)^{\frac{5-r}{r+1}}
    	\nonumber\\&\quad\times \left(\int_t^T\||\Z_1(\zeta)|+|\Z_2(\zeta)|\|_{\wi\L^{3(r+1)}}^{r+1}
    	\,\d\zeta\right)^{\frac{2(r-2)}{r+1}}
    	\nonumber\\&\leq 2^{\frac{3(r-1)}{r+1}} C
    	\left(\int_t^T  \|\Z_1(\zeta)\|_{\V}^{\frac{2(r+1)}{5-r}}
    	\,\d\zeta+
    	\int_t^T  \|\Z_2(\zeta)\|_{\V}^{\frac{2(r+1)}{5-r}}
    	\,\d\zeta\right)^{\frac{5-r}{r+1}}
    	\nonumber\\&\qquad\times \left(\int_t^T\||\Z_1(\zeta)|+|\Z_2(\zeta)|\|_{\wi\L^{3(r+1)}}^{r+1}
    	\,\d\zeta\right)^{\frac{2(r-2)}{r+1}}
    	\nonumber\\&\leq 4 C T^{\frac{5-r}{r+1}}
    	\big[\|\Z_1\|_{\mathrm{L}^{\infty}(t,T;\V)}^2+
    	\|\Z_2\|_{\mathrm{L}^{\infty}(t,T;\V)}^2\big] \left(\int_t^T\||\Z_1(\zeta)|+|\Z_2(\zeta)|\|_{\wi\L^{3(r+1)}}^{r+1}
    	\,\d\zeta\right)^{\frac{2(r-2)}{r+1}}.\label{unqA5}
    \end{align}
    \vskip 2mm
    \noindent
    \textbf{Case III:} \emph{For $r=5$ in $d=3$.} From \eqref{unqA4}, we conclude 
    \begin{align}\label{unqA6}
    	&\int_t^T \||\Z_1(\zeta)|+|\Z_2(\zeta)|\|_{\wi\L^{12}}^{8}\,\d\zeta
    	\nonumber\\&\leq
    	\int_t^T  \||\Z_1(\zeta)|+|\Z_2(\zeta)|\|_{\wi\L^{6}}^{2} \||\Z_1(\zeta)|+|\Z_2(\zeta)|\|_{\wi\L^{18}}^{6}\,\d\zeta
    	\nonumber\\&\leq
    	2C
    	\big[\|\Z_1\|_{\mathrm{L}^{\infty}(t,T;\V)}^2+\|\Z_2\|_{\mathrm{L}^{\infty}(t,T;\V)}^2\big]\left(\int_t^T \||\Z_1(\zeta)|+|\Z_2(\zeta)|\|_{\wi\L^{18}}^{6}\,\d\zeta\right).
    \end{align} 
    Combining \eqref{unqA6}--\eqref{unqA7} with \eqref{eqncont2} and substituting the resulting estimates into \eqref{unqA3}, we obtain \eqref{supVnormest}.
    \vskip 2mm
    \noindent
    \textbf{Step II: Proof of \eqref{z12negsob}.}
   On taking the inner product with $(\mathcal{A}+\I)^{-1}\mathfrak{Z}$ in the first equation of \eqref{redstapdff1}, we obtain
   \begin{align}\label{z2z3}
   	&\frac12\frac{\d}{\d s}\|(\mathcal{A}+\I)^{-\frac12}\mathfrak{Z}\|_{\H}^2+
   	\mu\|\mathfrak{Z}\|_{\H}^2+
   	\alpha\|(\mathcal{A}+\I)^{-\frac12}\mathfrak{Z}\|_{\H}^2
   	\nonumber\\&=
   	\mu\|(\mathcal{A}+\I)^{-\frac12}\mathfrak{Z}\|_{\H}^2
   	-\underbrace{\big(\mathfrak{B}(\Z_1)-\mathfrak{B}(\Z_2),
   		(\mathcal{A}+\I)^{-1}\mathfrak{Z}\big)}_{:=\mathcal{T}_1}
   	\nonumber\\&\quad-
   	\beta\underbrace{\big(\mathfrak{C}(\Z_1)-\mathfrak{C}(\Z_2),
   		(\mathcal{A}+\I)^{-1}\mathfrak{Z}\big)}_{:=\mathcal{T}_2}.
   \end{align}
   We now estimate $\mathcal{T}_1$ and $\mathcal{T}_2$, separately. By the properties of bilinear operator along with H\"older's and Young's inequalities, we estimate $\mathcal{T}_1$ as
   \begin{align}\label{bz12}
   	\big|\mathcal{T}_1\big|&\leq
   	\big|\big(\mathfrak{B}(\Z_1,\mathfrak{Z}),(\mathcal{A}+\I)^{-1}\mathfrak{Z}
   	\big)\big|+
   	\big|\big(\mathfrak{B}(\mathfrak{Z},\Z_2),(\mathcal{A}+\I)^{-1}\mathfrak{Z}
   	\big)\big|
   	\nonumber\\&\leq
   	\|\mathfrak{Z}\|_{\H}\|(\mathcal{A}+\I)^{-1}\mathfrak{Z}\|_{\wi\L^4}
   	\big[\|\nabla\Z_1\|_{\wi\L^4}+\|\nabla\Z_2\|_{\wi\L^4}\big]
   	\nonumber\\&\leq
   	\frac{\mu}{4}\|\mathfrak{Z}\|_{\H}^2+C
   	\big[\|\nabla\Z_1\|_{\wi\L^4}^2+\|\nabla\Z_2\|_{\wi\L^4}^2\big]
   	\|(\mathcal{A}+\I)^{-\frac12}\mathfrak{Z}\|_{\H}^2.
   \end{align}
    From \eqref{Ctayest}, together with the Sobolev embedding  $\V\hookrightarrow \widetilde{\mathbb{L}}^{r+1}$ for all values of $r$  considered in Cases I and II of Table \ref{Table1}, and by applying Young's  inequality, we estimate $\mathcal{T}_2$ as
    \begin{align}\label{cz12}
   	\big|\mathcal{T}_2\big|
   	&\leq
   		rC\||\Z_1|+|\Z_2|\|_{\wi\L^{2(r+1)}}^{r-1}\|\mathfrak{Z}\|_{\H}
   	\|(\mathcal{A}+\I)^{-\frac12}\mathfrak{Z}\|_{\H}
   	\nonumber\\&\leq
   	\frac{\mu}{4}\|\mathfrak{Z}\|_{\H}^2+
   	C\||\Z_1|+|\Z_2|\|_{\wi\L^{2(r+1)}}^{2(r-1)}
   	\|(\mathcal{A}+\I)^{-\frac12}\mathfrak{Z}\|_{\H}^2.
   \end{align}
   On substituting \eqref{bz12}-\eqref{cz12} into \eqref{z2z3}, we obtain
   \begin{align*}
   	&\frac12\frac{\d}{\d s}\|(\mathcal{A}+\I)^{-\frac12}\mathfrak{Z}\|_{\H}^2+
   \frac{\mu}{2}\|\mathfrak{Z}\|_{\H}^2+
   	\alpha\|(\mathcal{A}+\I)^{-\frac12}\mathfrak{Z}\|_{\H}^2
   	\nonumber\\&\leq
   	\mu\|(\mathcal{A}+\I)^{-\frac12}\mathfrak{Z}\|_{\H}^2+
   C
   \big[\|\nabla\Z_1\|_{\wi\L^4}^2+\|\nabla\Z_2\|_{\wi\L^4}^2\big]
   \|(\mathcal{A}+\I)^{-\frac12}\mathfrak{Z}\|_{\H}^2
   	\nonumber\\&\quad+
   	C\||\Z_1|+|\Z_2|\|_{\wi\L^{2(r+1)}}^{2(r-1)}
   	\|(\mathcal{A}+\I)^{-\frac12}\mathfrak{Z}\|_{\H}^2.
   \end{align*}
   On integrating above from $t$ to $s$ and then employing Gr\"onwall's inequality we arrive at 
   {\small
   	\begin{align}\label{z12negsob12}
   		&\|(\mathcal{A}+\I)^{-\frac12}\mathfrak{Z}(s)\|_{\H}^2+
   		\mu\int_{t}^s\|\mathfrak{Z}(\zeta)\|_{\H}^2\,\d\zeta+
   		2\alpha\int_{t}^s\|(\mathcal{A}+\I)^{-\frac12}\mathfrak{Z}(\zeta)\|_{\H}^2\,
   		\d\zeta
   		\nonumber\\&\leq
   		\|(\mathcal{A}+\I)^{-\frac12}(\z_1-\z_2)\|_{\H}^2
   		\nonumber\\&\quad\times
   		\exp\bigg(2\mu T+C\int_{t}^s
   		\big[\|\nabla\Z_1(\zeta)\|_{\wi\L^4}^2+\|\nabla\Z_2(\zeta)\|_{\wi\L^4}^2\big]
   		\d\zeta
   		+
   		C\int_{t}^s
   		\big[\||\Z_1(\zeta)|+|\Z_2(\zeta)|\|_{\wi\L^{2(r+1)}}^{2(r-1)}\big]
   		\d\zeta\bigg),
   \end{align}}
for all $s\in[t,T].$ The integral terms appearing in the exponent on the right-hand side of \eqref{z12negsob12} are uniformly bounded in view of \eqref{eqncont2}. Hence, using \eqref{supVnormest}, we finally deduce \eqref{z12negsob}.
    \end{proof}

The following proposition establishes higher order energy estimates in the $\D(\mathcal{A})-$norm and additional regularity properties for the strong solution of the system \eqref{stapPMP} .
    \begin{proposition}\label{extregu}
    Assume that $\z\in\D(\mathcal{A})$ and $\u\in\mathrm{L}^2(t,T;\H)$. Let $\Z(\cdot)$ be the unique strong solution of the system \eqref{stapPMP} with $\Z(t)=\z$.  
    Then, for $r$ in Table \ref{Table1} the following uniform energy estimate holds: 
    \begin{align}\label{extreg1}
    	\|\mathcal{A}\Z(s)\|_{\H}^2+
    	\int_t^{s}\|\mathcal{A}^{\frac32}\Z(\zeta)\|_{\H}^2\d\zeta\leq 
    	C\left(\mu,\beta,T,\|\z\|_{\D(\mathcal{A})},
    	\|\u\|_{\mathrm{L}^2(t,T;\H)}\right),
    \end{align}
    for all $s\in[t, T]$. Moreover, if $\u\in\C([t,T];\H)$, 
    then
    \begin{align}\label{extreg2}
    \Z\in\C([t,T];\D(\mathcal{A})) \ \text{ and } \ \frac{\d\Z}{\d s}\in\C([t,T];\H).
    \end{align} 
    \end{proposition}
    \begin{proof}
    The first part \eqref{extreg1} has been established in \cite[Appendix B]{smtm2}. We will address only the second part. In view of \eqref{extreg1} and the fact that $\Z(\cdot)$ is a strong solution, we have
    \begin{align*}
    	\Z\in\mathrm{L}^2(t,T;\D(\mathcal{A}^{\frac32})) \ \text{ and  } \ 
    	\ \frac{\d\Z}{\d s}\in\mathrm{L}^2(t,T;\H).
    \end{align*}
    Since, the embedding 
    $\D(\mathcal{A}^{\frac32})\hookrightarrow\D(\mathcal{A})
    \hookrightarrow\H$. Therefore, from \cite[Lemma 1.2, Chapter III]{Te}, we conclude  
    \begin{align}\label{zc0tda}
    	\Z\in\C([t,T];\D(\mathcal{A})).
    \end{align}
    Let us now prove the continuity of the time derivative. By using Agmon's inequality, we compute
    \begin{align}\label{bc0t}
    &\|\mathfrak{B}(\Z(s_1))-\mathfrak{B}(\Z(s_2))\|_{\H}
    \nonumber\\&=
    \|\mathfrak{B}(\Z(s_1)-\Z(s_2),\Z(s_2))\|_{\H}+
    \|\mathfrak{B}(\Z(s_1),\Z(s_1)-\Z(s_2))\|_{\H}
    \nonumber\\&\leq
    C\|\Z(s_1)-\Z(s_2)\|_{\H}^{1-\frac{d}{4}}
    \|(\mathcal{A}+\I)(\Z(s_1)-\Z(s_2))\|_{\H}^{\frac{d}{4}}
    \|\nabla\Z(s_2)\|_{\H}^{\frac{d}{4}}
    \nonumber\\&\quad+
     C\|\Z(s_1)\|_{\H}^{1-\frac{d}{4}}
    \|(\mathcal{A}+\I)\Z(s_1)\|_{\H}^{\frac{d}{4}}
    \|\nabla(\Z(s_1)-\Z(s_2))\|_{\H}^{\frac{d}{4}},
    \end{align}
    for all $s_1,s_2\in[t,T]$. Moreover, from \eqref{Ctayest}, and using the Sobolev embeddings $\D(\mathcal{A})\hookrightarrow\wi\L^{r+1}$ for $r\geq1$ and $\V\hookrightarrow\wi\L^6$, we find
    \begin{align}\label{cc0t}
  	&\|\mathfrak{C}(\Z(s_1))-\mathfrak{C}(\Z(s_2))\|_{\H}
  	\nonumber\\&\leq r\||\Z(s_1)|+|\Z(s_2)|\|_{\wi\L^{3(r-1)}}^{r-1}\|\Z(s_1)-\Z(s_2)\|_{\wi\L^{6}}
  	\nonumber\\&\leq rC\||\Z(s_1)|+|\Z(s_2)|\|_{\D(\mathcal{A})}^{r-1}\|\Z(s_1)-\Z(s_2)\|_{\V},
  \end{align}
  for all $s_1,s_2\in[t,T]$.
  From \eqref{zc0tda} and \eqref{extreg1}, it follows from \eqref{bc0t}-\eqref{cc0t} that 
  \begin{align}\label{bcc0t}
  	\mathfrak{B}(\Z(\cdot))\in\C([t,T];\H)  \ \text{ and } \
  		\mathfrak{C}(\Z(\cdot))\in\C([t,T];\H).
  \end{align} 
 Consequently, for $\u\in\C([t,T];\H)$, it follows from \eqref{stapPMP}, together with \eqref{zc0tda}, \eqref{bcc0t}, that $\frac{\d\Z}{\d s}\in\C([t,T];\H).$
    \end{proof}
    
    \section{The linearized and adjoint systems for the CBF equations}\label{LADJsys} \setcounter{equation}{0}
    In this section, we introduce the linearized and adjoint systems associated with the CBF system \eqref{stapPMP} and recall the corresponding well posedness results related to \eqref{stapPMP} (see \cite{Op1} for 2D CBF equations with $r=1,2,3$). These results play an important role in the derivation of the Pontryagin maximum principle and the verification type results (Section \ref{Pontryagin}).
    
    \subsection{The linearized system}\label{sysLN} 
    Let  $\z\in\V$ and $\Z(\cdot)$ be a strong solution to the system \eqref{stapPMP}.
    Then, we consider the following linearized system associated with the system \eqref{stapPMP}  in $\V^{*}+\wi\L^{\frac{r+1}{r}}$:
    \begin{equation}\label{stapLNR}
    	\left\{
    	\begin{aligned}
    		\frac{\d\Y(s)}{\d s}+ \mu\mathcal{A}\Y(s)+\mathfrak{B}'(\Z(s))\Y(s)+
    		\alpha\Y(s)+\beta\mathfrak{C}'(\Z(s))\Y(s)&=
    		\mathfrak{f}_1(s),  \\
    		\Y(t)&=\y\in\H,
    	\end{aligned}
    	\right.
    \end{equation}
 for a.e.  $ s\in[t,T]$.  Let us formally derive a priori estimates for \eqref{stapLNR}. Since  $\Z(\cdot)$ is a strong solution to the system \eqref{stapPMP}, the  following calculations are justified. A rigorous proof can be obtained by  employing a standard Faedo--Galerkin approximation method. Taking the inner product with $\Y(\cdot)$ in \eqref{stapLNR} and utilizing \eqref{C1d}, we obtain
    \begin{align}\label{fy}
    	&\frac12\frac{\d}{\d s}\|\Y\|_{\H}^2
    	+\mu\|\nabla\Y\|_{\H}^2+\alpha\|\Y\|_{\H}^2+
    	\beta\||\Z|^{\frac{r-1}{2}}\Y\|_{\H}^2
    	+\beta(r-1)
    	\||\Z|^{\frac{r-3}{2}}(\Y\cdot\Z)\|_{\H}^2
    	\nonumber\\&=-
    	\langle\mathfrak{B}(\Y,\Z),\Y\rangle+\langle\mathfrak{f}_1,\Y\rangle.
    \end{align}
   By using the properties \eqref{syymB} of bilinear operator and applying the Cauchy Schwarz inequality, we find
    \begin{align}\label{Best1}
    	\big|\langle\mathcal{B}(\Y,\Z),\Y\rangle\big|= \big|-\langle\mathcal{B}(\Y,\Y),\Z\rangle\big|
    	\leq
    	\frac{\mu}{4}\|\nabla\Y\|_{\H}^2+
    	\frac{1}{\mu}\|\Y\Z\|_{\H}^2.
    \end{align}
    For $r>3$, by an application of H\"older's and Young's inequalities, we estimate the last term in \eqref{Best1} as follows:
    \begin{align}\label{Best2}
    	\|\Y(s)\Z(s)\|_{\H}^2\leq
    	\frac{\beta\mu}{2}\||\Z(s)|^{\frac{r-1}{2}}\Y(s)\|_{\H}^2
    	+\varrho_1\|\Y(s)\|_{\H}^2,
    \end{align}
    where $\varrho_1:=\frac{r-3}{r-1}\left[\frac{4}{\beta\mu (r-1)}\right]^{\frac{2}{r-3}}.$
    	Moreover, we estimate $\langle\mathfrak{f},\Y\rangle$ as follows:
    \begin{align}\label{fy1}
    	\langle\mathfrak{f},\Y\rangle\leq\|\mathfrak{f}\|_{\V^{*}}\|\Y\|_{\V}
    	&\leq\|\mathfrak{f}\|_{\V^{*}}(\|\Y\|_{\H}+\|\nabla\Y\|_{\H})
    	\nonumber\\&\leq
    	\left(\frac{1}{2\alpha}+\frac{1}{\mu}\right)\|\mathfrak{f}_1\|_{\V^{*}}^2
    	+\frac{\alpha}{2}\|\Y\|_{\H}^2+\frac{\mu}{4}\|\nabla\Y\|_{\H}^2.
    \end{align}
    On substituting \eqref{Best1}-\eqref{fy1} into \eqref{fy}, we obtain
    \begin{align}\label{fy2}
    	&\frac12\frac{\d}{\d s}\|\Y\|_{\H}^2
    	+\frac{\alpha}{2}\|\Y\|_{\H}^2+\frac{\mu}{2}\|\nabla\Y\|_{\H}^2+
    	\frac{\beta}{2}\||\Z|^{\frac{r-1}{2}}\Y\|_{\H}^2
    	+\beta(r-1)
    	\||\Z|^{\frac{r-3}{2}}(\Y\cdot\Z)\|_{\H}^2
    	\nonumber\\&\leq
    	\left(\frac{1}{2\alpha}+\frac{1}{\mu}\right)\|\mathfrak{f}_1\|_{\V^{*}}^2
    	+\frac{\varrho_1}{\mu}\|\Y\|_{\H}^2.
    \end{align} 
   When $r=3$, taking the inner product in \eqref{stapLNR} with $\Y$ and using \eqref{C1d}, we arrive at 
     \begin{align}\label{fy3}
    	&\frac12\frac{\d}{\d s}\|\Y\|_{\H}^2
    	+\mu\|\nabla\Y\|_{\H}^2+\alpha\|\Y\|_{\H}^2+
    	\beta\|\Z\Y\|_{\H}^2+2\beta\|\Z\cdot\Y\|_{\H}^2
    	\nonumber\\&=-
    	\langle\mathfrak{B}(\Y,\Z),\Y\rangle+\langle\mathfrak{f}_1,
    	\Y\rangle.
    \end{align}
    Similar to \eqref{Best1}, we find
    \begin{align}\label{Best11}
    \big|\langle\mathcal{B}(\Y,\Z),\Y\rangle\big|
    \leq
    \frac{\mu}{2}\|\nabla\Y\|_{\H}^2+
    \frac{1}{2\mu}\|\Y\Z\|_{\H}^2.
    \end{align}
    Substituting \eqref{fy1} and \eqref{Best11} into \eqref{fy3}, we obtain
    \begin{align}\label{fy4}
   	&\frac12\frac{\d}{\d s}\|\Y\|_{\H}^2
   	+\frac{\mu}{4}\|\nabla\Y\|_{\H}^2+\frac{\alpha}{2}\|\Y\|_{\H}^2+
   	\left(\beta-\frac{1}{2\mu}\right)\|\Z\Y\|_{\H}^2
   	+2\beta\|\Z\cdot\Y\|_{\H}^2
   	\nonumber\\&\leq
   	\left(\frac{1}{2\alpha}+\frac{1}{\mu}\right)\|\mathfrak{f}_1\|_{\V^{*}}^2.
   \end{align}
    An application of Gr\"onwall's inequality to the estimates  \eqref{fy2} and \eqref{fy4} yields
    \begin{align}\label{fengest}
    	&\|\Y(s)\|_{\H}^2+\mu\int_t^s\|\nabla\Y(\tau)\|_{\H}^2\d\tau+
    	\alpha\int_t^s\|\Y(\tau)\|_{\H}^2
    	\d\tau+\beta\int_t^s \||\Z(\tau)|^{\frac{r-1}{2}}\Y(\tau)\|_{\H}^2\d\tau
    	\nonumber\\&\quad+
    	2\beta(r-1)\int_t^s \||\Z(\tau)|^{\frac{r-3}{2}}(\Y(\tau)\cdot\Z(\tau))\|_{\H}^2\d\tau
    	\nonumber\\&\leq
    	\bigg[\|\y\|_{\H}^2+
    	\left(\frac{1}{2\alpha}+\frac{1}{\mu}\right)\int_t^s \|\mathfrak{f}_1(\tau)\|_{\V^{*}}^2\d\tau\bigg]e^{k s},
    \end{align}
    for all $s\in[t,T]$, where the constant $k$ is defined as
    \begin{equation}\label{value-k}
    k=\left\{
    \begin{aligned}
    	\frac{2\varrho_1}{\mu},  \ &\text{ when } \ r>3 \ \text{ for all } \ \beta,\mu>0,\\
    	0, \ &\text{ when } \ r=3 \ \text{ with } \ 2\beta\mu\geq1.
    \end{aligned}
    \right.
    \end{equation} 
    We now derive a bound on the derivative of the solution $\Y$ of the system \eqref{stapLNR}. 
%    For this, we consider the following two cases:
%    \vskip 0.2cm
%    \noindent
%    \textbf{When $d=3$}. 
We consider the values of $r$ given in Table \ref{Table1}. We first note  from \eqref{fengest} that 
    \begin{align*}
    	\Y(\cdot)\in\mathrm{L}^{\infty}(t,T;\H)\cap\mathrm{L}^2(t,T;\V).
    \end{align*}
    We then calculate
    \begin{align}\label{bdere}
    	\left|\int_t^T\langle\mathfrak{B}'(\Z(s))\Y(s),\v(s)\rangle\d s\right|
    	&\leq\int_t^T\|\Z(s)\|_{\wi\L^4}\|\Y(s)\|_{\wi\L^4}\|\nabla\v(s)\|\d s
    	\nonumber\\&\leq
    	C\|\Z\|_{\mathrm{L}^{\infty}(t,T;\V)}
    	\|\Y\|_{\mathrm{L}^2(t,T;\V)}\|\v\|_{\mathrm{L}^2(t,T;\V)},
    \end{align}
    for all $\v\in\mathrm{L}^2(0,T;\V)$. Moreover, from \eqref{eqncont2} and the Sobolev embedding $\V\hookrightarrow\wi\L^{r+1}\hookrightarrow\wi\L^{\frac{r+1}{r}}\hookrightarrow\V^{*}$ for the values of $r$ given in Table \ref{Table1}, we find
    \begin{align}\label{cdere}
    	\left|\int_t^T\langle\mathfrak{C}'(\Z(s))\Y(s),\v(s)\rangle\d s\right|
    	&\leq
    	\int_t^T\|\mathfrak{C}'(\Z(s))\Y(s)\|_{\V^{*}}\|\v(s)\|_{\V}\d s
    	\nonumber\\&\leq
    	C\int_t^T \|\mathfrak{C}'(\Z(s))\Y(s)\|_{\wi\L^{\frac{r+1}{r}}}\|\v(s)\|_{\V}\d s
    	\nonumber\\&\leq
    	C\int_t^T\|\Z(s)\|_{\wi\L^{r+1}}^{r-1}\|\Y(s)\|_{\wi\L^{r+1}}\|\v(s)\|_{\V}\d s
    	\nonumber\\&\leq
    	C\int_t^T\|\Z(s)\|_{\V}^{r-1}\|\Y(s)\|_{\V}\|\v(s)\|_{\V}\d s
    	\nonumber\\&\leq
    	C\|\Z\|_{\mathrm{L}^{\infty}(t,T;\V)}^{r-1}
    	\|\Y\|_{\mathrm{L}^2(t,T;\V)}\|\v\|_{\mathrm{L}^2(t,T;\V)},
    \end{align}
    for all $\v\in\mathrm{L}^2(0,T;\V)$. Thus, from \eqref{eqncont2}, \eqref{fengest}-\eqref{cdere}, we deduce the following estimate:
    \begin{align}\label{ddere}
    	&\left|\int_t^T\left\langle\frac{\d\Y}{\d s}, \v(s)\right\rangle\d s\right|
    	\nonumber\\&=
    	\left|\int_t^T\langle-\mu\mathcal{A}\Y(s)-\mathfrak{B}'(\Z(s))\Y(s)-\alpha\Y(s)-\beta\mathfrak{C}'(\Z(s))(\Y(s))+\mathfrak{f}(s),\v(s)\rangle\d s\right|
    	\nonumber\\&\leq
    	C\left(\mu,\alpha,\beta,C,\|\y\|_{\H},
    	\|\f\|_{\mathrm{L}^2(t,T;\H)},\|\u\|_{\mathrm{L}^2(t,T;\H)},
    	\|\mathfrak{f}_1\|_{\mathrm{L}^2(t,T;\V^{*})}\right)\|\v\|_{\mathrm{L}^2(t,T;\V)},
    \end{align}
    for all $\v\in\mathrm{L}^2(t,T;\V)$. Thus, $\frac{\d\Y}{\d s} \in\mathrm{L}^2(t,T;\V^{*})$. Consequently, from \cite[Lemma 1.2, Chapter III]{Te}, we infer
    \begin{align*}
    	\Y(\cdot)\in\C([t,T];\H). 
    \end{align*}
%    \vskip 0.2cm
%    \noindent
%    \textbf{When $d=2$}. Note that in this case, in view of the Sobolev embedding $\V\hookrightarrow\wi\L^{r+1}$ for $r\geq1$, \eqref{fengest} yields following:
%    \begin{align*}
%    	\Y(\cdot)\in\mathrm{L}^{\infty}(t,T;\H)\cap\mathrm{L}^2(t,T;\V)
%    	\cap\mathrm{L}^{r+1}(t,T;\wi\L^{r+1}).
%    \end{align*}
%    Then, the bounds \eqref{bdere}-\eqref{cdere} and consequently \eqref{ddere} holds true for all $\v\in\mathrm{L}^2(t,T;\V)\cap\mathrm{L}^{r+1}(t,T;\wi\L^{r+1})$. Thus, we have $\frac{\d\Y}{\d s}\in \mathrm{L}^2(t,T;\V^{*})+\mathrm{L}^{\frac{r+1}{r}}(t,T;\wi\L^{\frac{r+1}{r}})$.
    
    The existence of the weak solution to \eqref{stapLNR} can be established by using the Faedo-Galerkin approximation techniques. Moreover, the uniqueness of weak solutions is immediate from \eqref{fengest} due to the linearity of the system \eqref{stapLNR} and the regularity of $\Z(\cdot)$. We thus arrive at the following result: 
    \begin{theorem}\label{soLvaLin}
Let $\Z(\cdot)$ be the strong solution to the system \eqref{stapPMP}. Let $\y\in\H$ and $\mathfrak{f}_1\in\mathrm{L}^2(t,T;\V^{*})$ be given. Then, there exists a unique weak solution $\Y(\cdot)$ to the system \eqref{stapLNR} satisfying 
\begin{equation*}
\Y(\cdot)\in
\mathrm{L}^{\infty}(t,T;\H)\cap\mathrm{L}^2(t,T;\V) \ \text{ and  }\  
\frac{\d\Y}{\d s}\in
\mathrm{L}^2(t,T;\V^{*}) \ \text{ in } d=2,3.
\end{equation*}
   Moreover, it satisfies the following energy equality:
    	\begin{align}\label{fyeng}
    		&\|\Y(s)\|_{\H}^2
    		+2\mu\int_t^s\|\nabla\Y(\zeta)\|_{\H}^2\d \zeta+2\alpha\int_t^s\|\Y(\zeta)\|_{\H}^2\d \zeta
    		\nonumber\\&\quad+
    		2\beta\int_t^s\||\Z(\zeta)|^{\frac{r-1}{2}}\Y(\zeta)\|_{\H}^2\d \zeta+2\beta(r-1)\int_t^s
    		\||\Z(\zeta)|^{\frac{r-3}{2}}(\Y(\zeta)\cdot\Z(\zeta))\|_{\H}^2\d \zeta
    		\nonumber\\&=\|\y\|_{\H}^2-2\int_t^s
    		\langle\mathfrak{B}(\Y(\zeta),\Z(\zeta)),\Y(\zeta)\rangle\d \zeta+
    		2\int_t^s\langle\mathfrak{f}_1(\zeta),\Y(\zeta)\rangle\d \zeta,
    	\end{align}
    	for all $s\in[t,T]$.
    \end{theorem}
    \subsection{The adjoint system}\label{sysADJ}
   To characterize the optimal control, it is essential to derive the adjoint system associated with \eqref{stapPMP}. The optimal control can then be expressed in terms of the corresponding adjoint variable. Let $\p$ denote the adjoint variable satisfying the following adjoint system   in $\V^{*}+\wi\L^{\frac{r+1}{r}}$:
    \begin{equation}\label{stapadj}
    	\left\{
    	\begin{aligned}
    		-\frac{\d\p(s)}{\d s}+ \mu\mathcal{A}\p(s)+\big(\mathfrak{B}'(\Z(s))\big)^{*}\p(s)+
    		\alpha\p(s)+\beta\mathfrak{C}'(\Z(s))\p(s)&=
    		\mathfrak{f}_2(s),  \\
    		\p(T)&=\p_T\in\H,
    	\end{aligned}
    	\right.
    \end{equation}
  for a.e.  $s\in[t,T]$. Here $(\mathfrak{B}'(\cdot))^{*}$ is defines as
    	\begin{align*}
    		\langle(\mathfrak{B}'(\Z))^{*}\p,\q\rangle=
    		\langle\p,\mathfrak{B}'(\Z)\q\rangle=
    		\langle\p,\mathfrak{B}(\Z,\q)\rangle+\langle\p,\mathfrak{B}(\q,\Z)\rangle.
    \end{align*}
    Similar to  Theorem \ref{soLvaLin}, the following result regarding the solvability of the adjoint system \eqref{stapadj} can also be established by using the standard Faedo-Galerkin approximation procedure.
    \begin{theorem}\label{soLvaadj}
    	Let $\Z(\cdot)$ be the strong solution to the system \eqref{stapPMP}. Let $\p_T\in\H$ and $\mathfrak{f}_2\in\mathrm{L}^2(t,T;\V^{*})$ be given. Then, there exists a unique weak solution $\p(\cdot)$ to the adjoint system \eqref{stapadj} satisfying 
    	\begin{equation*}
    	\p(\cdot)\in
    	\mathrm{L}^{\infty}(t,T;\H)\cap\mathrm{L}^2(t,T;\V) \ \text{ and }\  
    	\frac{\d\p}{\d s}\in\mathrm{L}^2(t,T;\V^{*}). 
    \end{equation*}
    	Moreover, the following energy equality holds:
    	\begin{align}\label{fyengadj}
    		&\|\p(s)\|_{\H}^2
    		+2\mu\int_s^T\|\nabla\p(\zeta)\|_{\H}^2\d \zeta+2\alpha\int_s^T\|\p(\zeta)\|_{\H}^2\d \zeta
    		\nonumber\\&\quad+
    		2\beta\int_s^T\||\Z(\zeta)|^{\frac{r-1}{2}}\p(\zeta)\|_{\H}^2\d \zeta+2\beta(r-1)\int_s^T
    		\||\Z(\zeta)|^{\frac{r-3}{2}}(\p(\zeta)\cdot\Z(\zeta))\|_{\H}^2\d \zeta
    		\nonumber\\&=\|\p_T\|_{\H}^2-2\int_s^T
    		\langle\mathfrak{B}(\p(\zeta),\Z(\zeta)),\p(\zeta)\rangle\d \zeta+
    		2\int_s^T\langle\mathfrak{f}_2(\zeta),\p(\zeta)\rangle\d \zeta,
    	\end{align}
    	for all $s\in[t,T]$.
    \end{theorem}

\section{An optimal control problem and the Hamilton--Jacobi--Bellman equation}
\setcounter{equation}{0}\label{anoptimaL}
In this section, we study the infinite-dimensional optimal control problem associated with the state equation \eqref{stapPMP} with forcing $\f\equiv\boldsymbol{0}$. We assume that the control takes values in a closed ball 
$\U=\mathrm{B}_{\H}(0,\mathtt{R})\subset\H$. Given 
$\z\in\D(\mathcal{A})$ and a control $\u\in\mathrm{L}^2(t,T;\H)$ such that 
\begin{align*}
\u(\zeta)\in\U \ \text{ for a.e. } \ \zeta\in[t,T],
\end{align*}
we consider the following cost functional associated with \eqref{stapPMP}:
\begin{align} \label{costFverf}
	\Upsilon(\Z; \u) =\frac12\int_t^T 
	\big[\|\nabla\Z(\zeta)\|_{\H}^2+\|\u(\zeta)\|_{\H}^2\big]\d\xi
	+\frac12\|\Z(T)\|_{\H}^2.
\end{align}
We define the admissible control class as 
\begin{align*}
\mathscr{U}:=\{\u\in\mathrm{L}^2(t,T;\H): \u(\zeta)\in\U \ \text{ for a.e. } \ \zeta\in[t,T]\}.
\end{align*}
 Moreover, we define the admissible class of solutions as follows:
	\begin{align*}
			\mathscr{A}:=\{(\Z,\u): \ \Z(\cdot)\ &\text{ is the unique strong solution to the system } \ \eqref{stapPMP} \\
			\ &\text{ with the control } \ \u(\cdot)\in\mathscr{U}\}.
		\end{align*} 
	In view of Theorem \ref{weLLp}, the system \eqref{stapPMP} admits a unique strong solution $\Z(\cdot)$ for any $\u(\cdot)\in\mathscr{U}$. Consequently, the admissible class $\mathscr{A}$ is non-empty. Thus, the optimal control problem is formulated as 
	\begin{align}\label{redopt}
			\min\limits_{(\Z,\u)\in\mathscr{A}}
			\Upsilon(\Z; \u).
		\end{align}
		\vskip 2mm
		\noindent
		\textbf{Physical interpretation of the cost functional \eqref{costFverf}.}
  From a physical perspective, the cost functional $\Upsilon(\Z;\u)$ is designed to balance three competing objectives over the control horizon $[t,T]$.
  The first term
  \begin{align*}
  	\frac{1}{2}\int_t^T \|\nabla \Z(\zeta)\|_{\H}^2\,\d\zeta,
  \end{align*}
  coincides with the enstrophy of the flow, $\int_t^T \|\nabla\times \Z(\zeta)\|_{\H}^2\,\d\zeta$. Enstrophy quantifies the intensity of rotational structures in the fluid and is closely related to the rate of viscous energy dissipation. Penalizing this term therefore reflects the physical goal of suppressing the growth of vorticity and rotational activity in the flow over $[t,T]$.
  
  The second term
  \begin{align*}
  	\frac{1}{2}\int_t^T \|\u(\zeta)\|_{\H}^2\,\d\zeta,
  \end{align*}
  represents the energy cost associated with applying the control input $\u$. Its inclusion prevents the optimal control from becoming unboundedly large, reflecting the physical fact that actuation is never free: the controller must regularize the flow economically rather than through arbitrarily strong intervention.
  
  The third term
  \begin{align*}
  	\frac{1}{2}\|\Z(T)\|_{\H}^2,
  \end{align*}
  penalizes the magnitude of the state at the terminal time $T$, driving the system toward a low-energy configuration by the end of the control horizon.
  
  \emph{Thus, minimizing $\Upsilon(\Z;\u)$ corresponds to seeking a control $\u$ that steers the flow so as to suppress vortical activity and bring the system to a low energy state at time $T$, while doing so with the least possible expenditure of control effort.}

	Any solution to Problem \ref{redopt} is called an \emph{optimal solution}, denoted by $(\Z^{*},\u^{*})$, which represents the optimal pair. Here, $\Z^{*}(\cdot)$ is the unique strong solution of the system \eqref{stapPMP} (called optimal trajectory) associated with the optimal control $\u^{*}\in\mathscr{U}$.
	
We now state the existence result for such an optimal pair. Its proof proceeds along the same lines as that of \cite[Theorem 3.3]{Op1}.
	
	\begin{theorem}[Existence of an optimal pair]
		Assume that $\z\in\D(\mathcal{A})$.  Then, for $r\geq1$ in $d=2$ and $r\geq3$ in $d=3$, there exists at least one admissible pair $(\Z^{*},\u^{*})$ for which the functional 
			$\Upsilon(\Z; \u)$ attains its minimum.
		\end{theorem}
	
We define the value function as
\begin{align}\label{vaLvar}
	\mathpzc{V}(t,\z):=\min\limits_{\u(\cdot)\in\mathscr{U}}
	\Upsilon(\Z; \u). 
\end{align}
It can be shown that the value function defined above formally yields the following HJB equation:
   \begin{equation}\label{detHJB1var}
	\left\{
	\begin{aligned}
		\mathpzc{V}_t-\mathcal{F}(\z,\mathcal{D} \mathpzc{V})&=0, \ \text{ for } (t,\z)\in(0,T)\times\D(\mathcal{A}), \\
		\mathpzc{V}(T,\z)&=\frac12\|\z\|_{\H}^2, \ \text{ for } \ \z\in\H.
	\end{aligned}
	\right.
\end{equation}
The Hamiltonian function $\mathcal{F}$ in \eqref{detHJB1var} is obtained from the \emph{pseudo-Hamiltonian function} in the following way:
\begin{align}\label{hamfunvar}
	\mathcal{F}(\z,\q)=\sup\limits_{\u\in\U} \wi{\mathcal{F}}(\z,\q,\u),
\end{align} 
where 
\begin{align}\label{pseudo}
\wi{\mathcal{F}}(\z,\q,\u):=
\big(\mu\mathcal{A}\z+\mathfrak{B}(\z)+\alpha\z+\beta\mathfrak{C}(\z),\q\big) -(\u,\q)-\frac12\big[\|\nabla\z\|_{\H}^2+\|\u\|_{\H}^2\big].
\end{align}
Then, the Hamiltonian function is given by 
\begin{align*}
		\mathcal{F}(\z,\q)=\big(\mu\mathcal{A}\z+\mathfrak{B}(\z)+\alpha\z+\beta\mathfrak{C}(\z),\q\big)+\Gamma(\q)-\frac12\|\nabla\z\|_{\H}^2,
\end{align*}
where 
	\begin{align*}
	\Gamma(\q)=\left\{\begin{array}{cc}
		\frac{1}{2}\|\q\|_{\H}^2, 
		&\text{ if } \ \|\q\|_{\H}\leq\mathtt{R},\\
		\mathtt{R}\|\q\|_{\H}-\frac{\mathtt{R}^2}{2}, &\text{ if } \ \|\q\|_{\H}>\mathtt{R},\end{array}\right.
\end{align*}
for some $\mathtt{R}>0$.  	Moreover, we remark that the optimal feedback control is given \emph{formally} as 
  	\begin{align}\label{optfeed}
	  		\wi\u(\cdot)=\upsigma\big(\q(\cdot)\big),
	  	\end{align}
  	where the function $\upsigma(\cdot)$ is given by
  	\begin{align}\label{upsigma}
	  		\upsigma(\q)=
	  		\begin{cases}
		  			-\q, \ &\text{ if } \ \ \|\q\|_{\H}\leq\mathtt{R},\\
		  			-\frac{\mathtt{R}}{\|\q\|_{\H}}\q, \ &\text{ if } \ \ \|\q\|_{\H}>\mathtt{R},
		  		\end{cases}
	  	\end{align}
  	and $\q(\cdot)=\mathcal{D}\mathpzc{V}(\cdot,\Z(\cdot))$, if $\mathpzc{V}$ is smooth. Since, we are dealing with a highly non-smooth solution (that is, viscosity solutions), therefore the above expressions require careful justification. To this end, we formulate the argument in the viscosity solution framework, where these steps properly justified. 
  	
  The following result establishes the continuity of the value function \eqref{vaLvar} and the Bellman principle of optimality. Its proof relies on arguments analogous to those in \cite{SSS} (see also \cite{FGSSA1}). Since the reasoning is standard and does not entail any additional technical difficulties, we omit the details.
  	\begin{theorem}
  		The value function $\mathpzc{V}$ is continuous, that is, $\mathpzc{V}\in\C([0,T];\H)$. Moreover, it satisfies the following properties:
  		
  		\textbf{(i) Locally Lipschitz in space variable.} For all $\mathpzc{R}>0$, there exists a constant $k=k(\mathpzc{R})>0$ such that
  		\begin{align}\label{Lpsch1}
  			|\mathpzc{V}(s,\z_1)-\mathpzc{V}(s,\z_2)|\leq k\|\z_1-\z_2\|_{\H},
  		\end{align}
  		for all $\z_1,\z_2\in\H$ with $\|\z_1\|_{\H},\|\z_2\|_{\H}\leq\mathpzc{R}$ and all $s\in[0,T]$.
  		
  			\textbf{(ii) Locally Lipschitz in time variable.} For all $\mathpzc{R}>0$, there exists a constant $k=k(\mathpzc{R})>0$ such that
  		\begin{align}\label{Lpsch2}
  			|\mathpzc{V}(t_1,\z)-\mathpzc{V}(t_2,\z)|\leq k|t_1-t_2|,
  		\end{align}
  		for all $\z\in\V$ with $\|\z\|_{\V}\leq\mathpzc{R}$ and all $t_1,t_2\in[t,T]$.
  
  			\textbf{(iii) Continuity in the negative norm.} For all $\mathpzc{R}>0$, there exists a constant $k=k(\mathpzc{R})>0$ such that
  			\begin{align*}
  				|\mathpzc{V}(s,\z_1)-\mathpzc{V}(s,\z_2)|\leq k\|\z_1-\z_2\|_{\V^{*}},
  			\end{align*}
  			for all $\z_1,\z_2\in\V$ with $\|\z_1\|_{\V}, \|\z_2\|_{\V}\leq\mathpzc{R}$ and all $s\in[t,T]$.
  			
  			 \textbf{(iv) Dynamic Programming Principle (DPP):}
  			For all $0 \leq t \leq s \leq T$ and $\z \in \V$, we have:
  			\begin{align}\label{vdpp}
  				\mathpzc{V}(t, \z) = \inf_{\u(\cdot)\in\mathscr{U}} \left\{ 
  			\frac12\int_t^s
  			\big[\|\nabla\Z(\xi)\|_{\H}^2 +\|\u(\xi)\|_{\H}^2\big]\d\xi+ 
  			\mathpzc{V}(s, \Z(s)) \right\}.
  			\end{align}
  	\end{theorem}
  	After establishing DPP, the next step is to show that the value function $\mathpzc{V}$ is the viscosity solution of the HJB equation \eqref{detHJB1var}. Since our goal is to obtain a verification theorem  tailored to the specific optimal control problem  \eqref{redopt}, we adopt a similar framework as in \cite{SSS}. In particular, this requires a slight adaptation of the definition of the test function class given below.
  	\begin{definition}\label{testDfunc}
  	 A class of functions $\phi:[t,T]\times\H\to\R$ is called \emph{a test function of class $\mathscr{D}$} if the following holds:
  	 
  	 (i) $\phi(\cdot,\cdot):[t,T]\times\H\to\R$ is locally Lipschitz.
  	 
  	 (ii) The Fr\'echet derivative is locally Lipschitz 
  	       \begin{align*}
  	       	\mathcal{D}\phi=(\phi_s,\mathcal{D}_{\z}\phi):[t,T]\times\V\to\R\times\V
  	       \end{align*}
  	        is locally Lipschitz.
  	\end{definition}
  	
  	\begin{definition}\label{vaLuevarf}
  		We say that a value function $\mathpzc{V}:[0,T]\times\H\to\R$ is a \emph{viscosity solution} to the HJB equation \eqref{detHJB1var}  if for all $\phi\in\mathscr{D}$, 
  		
  		(i) $\mathpzc{V}-\phi$ has a local maximum at $(t,\z)\in(0,T)\times\D(\mathcal{A})$, then 
  		\begin{align*}
  			-\phi_s(t,\z)+\mathcal{F}(\z,\mathcal{D}_{\z}\phi)\leq0;
  		\end{align*}
  		
  		(ii) $\mathpzc{V}-\phi$ has a local minimum at $(t,\z)\in(0,T)\times\D(\mathcal{A})$, then 
  		\begin{align*}
  			-\phi_s(t,\z)+\mathcal{F}(\z,\mathcal{D}_{\z}\phi)\geq0.
  		\end{align*}
  	\end{definition}
  	We state the following result (without proof) regarding the existence of a viscosity solution to the HJB equation \eqref{detHJB1var}. Its proof can be obtained in a manner similar to that in \cite{SSS1,SSS}, with a slight modification.
  	\begin{theorem}
  		The value function $\mathpzc{V}:[0,T]\times\H\to\R$ is a viscosity solution to the HJB equation \eqref{detHJB1var} in the sense of Definition \ref{vaLuevarf}.
  	\end{theorem}
  	
\subsection{A remark on the uniqueness of the viscosity solution}
It should be noted that the test function class defined in \eqref{testDfunc} is not sufficiently rich to guarantee the uniqueness of viscosity solutions for \eqref{detHJB1var}. Indeed, proving uniqueness under this framework remains an open problem, as pointed out in \cite{SSS1}.
We now discuss the main obstruction that arises from Definition \ref{testDfunc} and prevents the standard uniqueness arguments from being applied.

The uniqueness of viscosity solutions is usually established via a comparison principle. Such a principle guarantees, under appropriate conditions, that any viscosity subsolution is bounded above by any viscosity supersolution. The proof generally employs the doubling of variables method and a suitable penalization argument. A central step in this analysis is the evaluation of the Hamiltonian $\mathcal{F}(\z_0,\p_0)$ at the extremal point $(t_0,\z_0)$ generated by the penalization procedure. For the CBF system (and, more generally, for HJB equations associated with the CBF equations), the Hamiltonian $\mathcal{F}$ contains the term $\big(\mu\mathcal{A}\z_0,\boldsymbol{\xi}\big)$ (see \eqref{hamfunvar}-\eqref{pseudo}) which is meaningful only when $\z_0\in\D(\mathcal{A})$. However, under Definition \ref{testDfunc}, there is no mechanism ensuring that the extremal point $\z_0$ arising from the doubling of variables or penalization argument belongs to $\D(\mathcal{A})$. Indeed, a generic locally Lipschitz test function $\phi$ may attains its extremum at a point $\z_0\in\H\setminus\D(\mathcal{A})$. Consequently, the quantity $\mathcal{A}\z_0$ may not be well defined, and the viscosity inequalities in Definition \ref{vaLuevarf} cannot even be formulated at $(t_0,\z_0)$. As a result, the comparison argument breaks down at its very first step.

Motivated by \cite{FGSSA,FGSSA1}, we overcome the above obstruction by modifying the class of test functions, following the approach developed in \cite{smtm2}. This modification ensures that the extremal points arising in the comparison argument are regular enough for the Hamiltonian to be well defined. While the uniqueness problem for the HJB equation associated with the 3D NSE remains open, the additional dissipation induced by the absorption term in the CBF system allows us to establish uniqueness in the three-dimensional setting. We now recall the corresponding class of test functions.
   \begin{definition}\label{testD}
	A function $\psi:[t,T]\times\H\to\R$ is said to be a test function of class $\mathscr{D}^+$ if $$\psi(t,\z)=\psi_0(t,\z)+\mathfrak{h}(t)\|\z\|_{\V}^2,$$ where 
		\begin{enumerate}
		\item\label{Dp1}
		$\psi_0\in\mathrm{C}^{1,1}([t,T]\times\H)$ and
		the partial derivative $(\psi_0)_t$ and the derivative $\mathcal{D}\psi_0$ are uniformly continuous on every closed subinterval $[\lambda,T-\lambda]\times\H$ for every $\lambda>0$.
		\item\label{Dp2}
		$\mathfrak{h}(\cdot)\in\mathrm{C}^1([t,T])$ and satisfies $\mathfrak{h}(s) > 0$ for all $s\in[t, T]$,
		\item\label{Dp3}
		the formal derivative
		\begin{align}\label{formader}
	\mathcal{D}_{\z} \|\z\|_{\V}^2 = 2(\mathcal{A} + \I)\z
		\end{align}
		is interpreted as an operator well-defined for $\z\in\D(\mathcal{A})$.
	\end{enumerate}
\end{definition}
\vskip 2mm
\noindent
\textbf{Relationship between $\mathscr{D}$ and $\mathscr{D}^+$.} Note that every $\psi\in\mathscr{D}^+$ belongs to $\mathscr{D}$ in the following sense: the
smooth part $\psi_0\in C^{1,1}\subset\mathscr{D}$ trivially, and the radial
part $\mathfrak{h}(t)\|\z\|_{\V}^2$ is locally Lipschitz on $(0,T)\times\V$. and its formal Fr\'echet derivative \eqref{formader} is locally Lipschitz from $\D(\mathcal{A})$ to $\H$. Hence
\begin{align*}
	\mathscr{D}^+\subset\mathscr{D}.
\end{align*}
The inclusion is strict because the class $\mathscr{D}$ contains locally Lipschitz functions with locally Lipschitz Fr\'echet derivative but with no coercive
radial part, for example, $\phi(t,\z)=\|\z\|_{\V}^2\cdot\chi(t)$ for some
$\chi\in\C^1([t,T])$.
Such functions belong to $\mathscr{D}$ but not to $\mathscr{D}^+$
(since they cannot be written in the form as mentioned in Definition \ref{testD}).
	\vskip 2mm
	\noindent
	\textbf{Why $\z_0\in\D(\mathcal{A})$ is now guaranteed.}
	\label{rem:v0inDA}
	The key feature of Definition \ref{testD} is that the radial
	term $\mathfrak{h}(t)\|\z\|_{\V}^2$ is coercive in the
	$\|\cdot\|_{\V}$-norm. Therefore, whenever $u-\psi$ (where $u$ is the viscosity subsolution of the associated HJB equation) attains a global maximum at $(t_0,\z_0)$, the bound
	\begin{align*}
		u(t_0,\z_0)-\psi(t_0,\z_0)\ge u(t,\z)-\psi(t,\z) \ \text{ for all }  (t,\z)\in(0,T)\times\V,
	\end{align*}
	combined with the polynomial growth bound on $u$ (of the form
	$|u(t,\z)|\le C(1+\|\z\|_{\V}^k)$ for some $k>0$, analogous to
	\cite[Theorem 5.2]{smtm2} forces $\|\z_0\|_{\V}<\infty$, that is, $\z_0\in\V$.  The additional regularity $\z_0\in\D(\mathcal{A})$ then
	follows from a standard arguments based on energy estimates (Proposition \ref{weLLp})  exploiting the CBF dynamics.  Once $\z_0\in\D(\mathcal{A})$, then $\mathcal{A}\z_0$,  $\mathfrak{B}(\z_0)$, and $\mathfrak{C}(\z_0)$ are
	well-defined in $\H$, and the Hamiltonian $\mathcal{F}(\z_0,\mathcal{D}_{\z}(t_0,\z_0))$ can be evaluated rigorously.

We now provide  the definition of viscosity solution and state the comparison principle which is compatible with the test function class $\mathscr{D}^+$. A detailed proof is given in \cite[Section 5]{smtm2}. 
\begin{definition}[Viscosity solution in class $\mathscr{D}^+$]\label{viscositynew}
	A function $u:(0,T)\times\V\to\R$, weakly sequentially
	upper-semicontinuous (respectively, lower-semicontinuous) on
	$(0,T)\times\V$, is a \emph{viscosity subsolution} (respectively, \emph{viscosity supersolution}) of \eqref{detHJB1var} with respect to $\mathscr{D}^+$ if, for every
	$\phi\in\mathscr{D}^+$, whenever $u-\phi$ attains a global maximum
	(respectively, $u+\phi$ attains a global minimum) over
	$(0,T)\times\V$ at $(t_0,\z_0)$, then
	$\z_0\in\D(\mathcal{A})$ and
	\begin{align*}
		-\partial_t\phi(t_0,\z_0)+\mathcal{F}(\z_0,\partial_v\phi(t_0,\z_0))
		&\leq 0\\
		(\text{respectively,}
		-\partial_t\phi(t_0,\z_0)
		+\mathcal{F}(\z_0,\mathcal{D}_{\z}\phi(t_0,\z_0))
		&\geq 0).
	\end{align*}
	A function is a \emph{viscosity solution} if it is both a
	viscosity subsolution and a viscosity supersolution.
\end{definition}
The Hamiltonian $\mathcal{F}(\z,p)$  given in \eqref{hamfunvar} satisfies the continuity hypotheses of \cite[Hypothesis 5.1, Section 5]{smtm2}. Therefore, the comparison principle established in \cite[Theorem 5.2]{smtm2} is applicable in our framework and can be stated as follows.
\begin{theorem}[Comparison principle: Uniqueness]
	Let $u,v:(0,T)\times\V\to\R$ be, respectively, a viscosity
	subsolution and a viscosity supersolution of \eqref{detHJB1var}
	in the sense of Definition \ref{detHJB1var}, with polynomial growth
	\begin{align*}
		u(t,\z),\,-v(t,\z)\leq C(1+\|\z\|_{\V}^k)
	\end{align*}
	for some $k>0$, and satisfying the terminal conditions
	\begin{align*}
	\lim_{t\uparrow T}(u(t,\z)-g(\z))^+=0 \ \text{ and }
	\lim_{t\uparrow T}(v(t,\z)-g(\z))^-=0, \ \text{ where } \ g(\z)=\frac12\|\z\|_{\V}^2,
	\end{align*}
	uniformly on bounded subsets of $\V$. Then, $u\leq v$ on $(0,T]\times\V$.
\end{theorem}
  \begin{remark}
 Since $\mathscr{D}^+\subset\mathscr{D}$, every test function $\psi\in\mathscr{D}^+$ is in particular a test function in $\mathscr{D}$.  Therefore:
 \begin{align*}
 	\mathpzc{V}\text{ is a viscosity solution w.r.t.\ }\mathscr{D}
 	\implies
 	\mathpzc{V}\text{ is a viscosity solution w.r.t.\ }\mathscr{D}^+.
 \end{align*}
 In other words, restricting the class of test functions from $\mathscr{D}$ to the smaller class $\mathscr{D}^+$ weakens the notion of viscosity solution rather than strengthening it. Furthermore, all subsequent results established in the paper remain valid without modification. Indeed, since these results are proved for the larger test function class $\mathscr{D}$, they continue to hold a fortiori when the smaller class $\mathscr{D}^+\subset\mathscr{D}$ is employed. The only genuinely new consequence of introducing the refined class $\mathscr{D}^+$ is the validity of the comparison principle.
% The only result for which the specific structure of the test class
% matters is uniqueness, precisely because a comparison argument must
% evaluate the Hamiltonian at an arbitrary extremal point not
% necessarily on the optimal trajectory and it is only $\mathscr{D}^+$ that
% forces such points into $\D(\mathcal{A})$.
  \end{remark}
  	\section{Feedback analysis}\setcounter{equation}{0}\label{feedbackana}
 In this section, we introduce a functional $\mathcal{W}$, analogous to the value function $\mathpzc{V}$ but evaluated along the trajectory corresponding to a fixed admissible control, and establish its regularity properties namely continuity, local Lipschitz continuity, and Fr\'echet differentiability. These properties will be essential in relating $\mathcal{W}$ to $\mathpzc{V}$ via the viscosity solution framework in order to derive the feedback form of the optimal control.
  	
  		Consider the following system
  		\begin{equation}\label{stapocp}
  		\left\{
  		\begin{aligned}
  			\frac{\d\Z(s)}{\d s}&= -\mu\mathcal{A}\Z(s)-\mathfrak{B}(\Z(s))-\alpha\Z(s)-
  			\beta\mathfrak{C}(\Z(s))+\wi\u(s),  \  \text{ in } \ (\tau,T)\times\H, \\
  			\Z(\tau)&=\z\in\D(\mathcal{A}),
  		\end{aligned}
  		\right.
  	\end{equation}
  	where $\wi\u\in\mathscr{U}$ is an optimal control in $[t,T]$ with initial data $(t,\z)$. Note that in general $\Z(s,\tau,\z,\wi\u)$ would be optimal trajectory only when $\tau=t$ (for $\tau>t$, the control $\wi\u$ on $[\tau,T]$ is admissible only but not necessarily optimal). Let us define a functional $\mathcal{W}:[t,T]\times\H\to\R$ as 
  	\begin{equation}\label{wfunc}
  		\left\{
  		\begin{aligned}
  			\mathcal{W}(\tau,\z)&:=
  		\frac12\|\Z(T)\|_{\H}^2+\frac12\int_{\tau}^T 
  		\big[\|\nabla\Z(\xi)\|_{\H}^2+\|\wi\u(\xi)\|_{\H}^2\big]\d \xi,\\
  		\text{ with } \ \mathcal{W}(T,\z)&=\frac12\|\z\|_{\H}^2,
  		\end{aligned}
  		\right.
  	\end{equation}
  	where $\Z$ is a strong solution to the system \eqref{stapocp}. Let us examine some properties of $\mathcal{W}$. 
Define functions $\Phi(\cdot)$ and $\X(\cdot)$ for some $s\in[\tau,T]$ as 
\begin{equation}\label{phx}
\left\{
\begin{aligned}
\Phi(s)&:=\mathcal{D}_{\z}\Z(s,\tau;\z;\wi\u)\x, \ \x\in\H,\\
\X(s)&:=\Z_{\tau}(s,\tau;\z;\wi\u).
\end{aligned}
\right.
\end{equation}
It can be shown that $\Phi(\cdot)$ defined above satisfies the following linearized system in $\V^{*}+\L^{\frac{r+1}{r}}$: 
\begin{equation}\label{phx1}
\left\{
\begin{aligned}
\frac{\d\Phi(s)}{\d s}+ \mu\mathcal{A}\Phi(s)+\mathfrak{B}'(\Z(s))\Phi(s)+
\alpha\Phi(s)+\beta\mathfrak{C}'(\Z(s))\Phi(s)&=\boldsymbol{0},   \\
\Phi(\tau)&=\x\in\H,
\end{aligned}
\right.
\end{equation}
for a.e.  $s\in[\tau,T],$  in the weak sense. From Theorem \ref{soLvaLin}, the above linearized system \eqref{phx1} has a unique weak solution $\Phi(\cdot)\in\C([\tau,T];\H)\cap\mathrm{L}^2(\tau,T;\V)$ with 
$\frac{\d\Phi}{\d s}\in\mathrm{L}^2(\tau,T;\V^{*})$ and satisfies the following energy estimate:
\begin{align}\label{phiengest1}
 	&\|\Phi(s)\|_{\H}^2+\mu\int_{\tau}^s\|\nabla\Phi(\xi)\|_{\H}^2\d\xi+
 \alpha\int_{\tau}^s\|\Phi(\xi)\|_{\H}^2
 \d\xi+\beta\int_{\tau}^s \||\Z(\xi)|^{\frac{r-1}{2}}\Phi(\xi)\|_{\H}^2\d\xi
 \nonumber\\&\quad+
 2\beta(r-1)\int_{\tau}^s \||\Z(\xi)|^{\frac{r-3}{2}}(\Phi(\xi)\cdot\Z(\xi))\|_{\H}^2\d\xi\leq
\|\x\|_{\H}^2 e^{k s},
\end{align} 
for all $s\in[\tau,T]$, where $k$ is defined in \eqref{value-k}. Moreover, the function $\X(\cdot)$ defined in \eqref{phx} satisfies the following linearized system:
\begin{equation}\label{phx2}
\left\{
\begin{aligned}
\frac{\d\X(s)}{\d s}+ \mu\mathcal{A}\X(s)&+\mathfrak{B}'(\Z(s))\X(s)+
\alpha\X(s)\\ +\beta\mathfrak{C}'(\Z(s))\X(s)&=\boldsymbol{0},  \  \text{ for a.e.  } s\in[\tau,T], \\
\X(\tau)&=\mu\mathcal{A}\z+\mathfrak{B}(\z)+\alpha\z+
\beta\mathfrak{C}(\z)-\wi\u(\tau)\in\H.
\end{aligned}
\right.
\end{equation}
Note that the initial condition in \eqref{phx2} makes sense since $\z\in\D(\mathcal{A})$. Furthermore, from Theorem \ref{soLvaLin}, the above linearized system \eqref{phx2} has a unique weak solution $\X(\cdot)\in\C([\tau,T];\H)\cap\mathrm{L}^2(\tau,T;\V)$ with 
$\frac{\d\X}{\d s}\in\mathrm{L}^2(\tau,T;\V^{*})$ and satisfies the following energy estimate:
\begin{align}\label{Xengest1}
	&\|\X(s)\|_{\H}^2+\mu\int_{\tau}^s\|\nabla\X(\tau)\|_{\H}^2\d\zeta+
	\alpha\int_{\tau}^s\|\X(\zeta)\|_{\H}^2
	\d\tau+\beta\int_{\tau}^s \||\Z(\zeta)|^{\frac{r-1}{2}}\X(\zeta)\|_{\H}^2\d\zeta
	\nonumber\\&\quad+
	2\beta(r-1)\int_{\tau}^s \||\Z(\zeta)|^{\frac{r-3}{2}}(\X(\zeta)\cdot\Z(\zeta))\|_{\H}^2\d\zeta
	\nonumber\\&\leq
	C(\mu,\alpha,\beta,T, \mathtt{R},\|\mathcal{A}\z\|_{\H}),
\end{align} 
for all $s\in[\tau,T]$. 

The following propositions play a fundamental role in establishing the local Lipschitz continuity of the functional $\mathcal{W}$ (see Proposition \ref{WpropertyLip}). Moreover, they provide crucial estimates in negative norms for both the solutions and the differences of solutions to the linearized systems \eqref{phx1} and \eqref{phx2}. These estimates are indispensable for the subsequent analysis.
\begin{proposition}\label{phinegatest}
The weak solution $\Phi$ to the linearized system \eqref{phx1} satisfies the following energy estimate for the values of $r$ given in Table \ref{Table1}:
\begin{align}\label{phinegatest1}
\sup\limits_{s\in[\tau,T]}\|(\mathcal{A}+\I)^{-\frac12}\Phi(s)\|_{\H}^2+
\int_{\tau}^T \|\Phi(\xi)\|_{\H}^2\d\xi
%+\int_{\tau}^T\|(\mathcal{A}+\I)^{-\frac12}\Phi(\xi)\|_{\H}^2\d\xi
\leq 
C\big(\mu,\alpha,\beta,T,\|\z\|_{\V}\big)
\|(\mathcal{A}+\I)^{-\frac12}\x\|_{\H}^2.
\end{align}
Moreover, if $\Phi_1$ and $\Phi_2$ are two weak solutions of the system \eqref{phx1} with $\Phi_1(\tau)=\x_1$ and $\Phi_2(\tau)=\x_2$ corresponding to the strong solutions $\Z_1$ and $\Z_2$ of the system \eqref{stapPMP}, respectively, then we have
\begin{align}\label{phinegatest2}
&\sup\limits_{s\in[\tau,T]}
\|(\mathcal{A}+\I)^{-\frac12}(\Phi_1-\Phi_2)(s)\|_{\H}^2+
\int_{\tau}^T \|(\Phi_1-\Phi_2)(\xi)\|_{\H}^2\d\xi
%+\int_{\tau}^T\|(\mathcal{A}+\I)^{-\frac12}(\Phi_1-\Phi_2)(\xi)\|_{\H}^2\d\xi
\nonumber\\&\leq 
C\big(\mu,\alpha,\beta,T,\|\z\|_{\V}\big)\big[
\|(\mathcal{A}+\I)^{-\frac12}(\x_1-\x_2)\|_{\H}^2+
\|\z_1-\z_2\|_{\V}^4+\|\z_1-\z_2\|_{\V}^2\big].
\end{align}
\end{proposition}
  		
\begin{proof}
Let us take the inner product with $(\mathcal{A}+\I)^{-1}\Phi$ in the first equation of \eqref{phx1}, we get
\begin{align}\label{Phi1Phi2}
&\frac12\frac{\d}{\d s}\|(\mathcal{A}+\I)^{-\frac12}\Phi\|_{\H}^2+
\mu\|\Phi\|_{\H}^2+\alpha\|(\mathcal{A}+\I)^{-\frac12}\Phi\|_{\H}^2
\nonumber\\&=
\mu\|(\mathcal{A}+\I)^{-\frac12}\Phi\|_{\H}^2-
\big(\mathfrak{B}'(\Z)\Phi,(\mathcal{A}+\I)^{-1}\Phi\big)-
\beta\big(\mathfrak{C}'(\Z)\Phi,(\mathcal{A}+\I)^{-1}\Phi\big).
\end{align}
Let us now estimate the terms in the right hand side of \eqref{Phi1Phi2}. By using H\"older's, Ladyzhenskaya and Young's inequalities, we calculate
\begin{align}\label{bphi}
\left|\big(\mathfrak{B}'(\Z)\Phi,(\mathcal{A}+\I)^{-1}\Phi\big)\right|
&\leq
\left|\big(\mathfrak{B}(\Phi,(\mathcal{A}+\I)^{-1}\Phi),\Z\big)\right|+
\left|\big(\mathfrak{B}(\Z,(\mathcal{A}+\I)^{-1}\Phi),\Phi\big)\right|
\nonumber\\&\leq
2\|\Phi\|_{\H}\|\mathcal{A}^{\frac12}(\mathcal{A}+\I)^{-1}\Phi\|_{\wi\L^4}
\|\Z\|_{\wi\L^4}
\nonumber\\&\leq
2C\|\Phi\|_{\H}\|\mathcal{A}(\mathcal{A}+\I)^{-1}\Phi\|_{\H}^{\frac{d}{4}}
\|\mathcal{A}^{\frac12}(\mathcal{A}+\I)^{-1}\Phi\|_{\H}^{1-\frac{d}{4}}
\|\Z\|_{\wi\L^4}
\nonumber\\&\leq
\frac{\mu}{4}\|\Phi\|_{\H}^2+C
\|(\mathcal{A}+\I)^{-\frac12}\Phi\|_{\H}^2\|\Z\|_{\wi\L^4}^{\frac{8}{4-d}}.
\end{align}
%where $\kappa_1:=2^{\frac{8}{4-d}}\left(\frac{4+d}{2\mu}\right)^{\frac{4+d}{4-d}}
%\frac{4-d}{8}$. 
In \eqref{bphi}, we have used the fact that $\|\mathcal{A}^{\frac12}\cdot\|_{\H}\leq
\|(\mathcal{A}+\I)^{\frac12}\cdot\|_{\H}$. Similarly, we find
\begin{align}\label{cphi}
\left|\big(\mathfrak{C}'(\Z)\Phi,(\mathcal{A}+\I)^{-1}\Phi\big)\right|
&\leq\|\mathfrak{C}'(\Z)\Phi\|_{\wi\L^{\frac65}}
\|(\mathcal{A}+\I)^{-1}\Phi\|_{\wi\L^6}
\nonumber\\&\leq
r\||\Z|^{r-1}\Phi\|_{\wi\L^{\frac65}}\|(\mathcal{A}+\I)^{-1}\Phi\|_{\wi\L^6}
\nonumber\\&\leq
r\|\Z\|_{\wi\L^{3(r-1)}}^{r-1}\|\Phi\|_{\H}\|(\mathcal{A}+\I)^{-1}\Phi\|_{\wi\L^6}
\nonumber\\&\leq
\frac{\mu}{4}\|\Phi\|_{\H}^2+C
\|(\mathcal{A}+\I)^{-\frac12}\Phi\|_{\H}^2\|\Z\|_{\wi\L^{3(r-1)}}^{2(r-1)}.
\end{align}
 Combining \eqref{bphi}-\eqref{cphi} into \eqref{Phi1Phi2}, we obtain
\begin{align*}
&\frac12\frac{\d}{\d s}\|(\mathcal{A}+\I)^{-\frac12}\Phi\|_{\H}^2+
\frac{\mu}{2}\|\Phi\|_{\H}^2+\alpha\|(\mathcal{A}+\I)^{-\frac12}\Phi\|_{\H}^2
\nonumber\\&\leq
C\big[\|\Z\|_{\wi\L^4}^{\frac{8}{4-d}}+ \|\Z\|_{\wi\L^{3(r-1)}}^{2(r-1)}\big]\|(\mathcal{A}+\I)^{-\frac12}\Phi\|_{\H}^2.
\end{align*}
Integrating the above inequality  from $\tau$ to $s$ and then employing Gr\"onwall's inequality, we get
\begin{align}\label{bphicphi1}
&\|(\mathcal{A}+\I)^{-\frac12}\Phi(s)\|_{\H}^2+\mu\int_{\tau}^s \|\Phi(\zeta)\|_{\H}^2\d\zeta+2\alpha\int_{\tau}^s 
\|(\mathcal{A}+\I)^{-\frac12}\Phi(\zeta)\|_{\H}^2\d\zeta
\nonumber\\&\leq 
\|(\mathcal{A}+\I)^{-\frac12}\x\|_{\H}^2
\exp\bigg[C\int_{\tau}^s\|\Z(\zeta)\|_{\wi\L^4}^{\frac{8}{4-d}}
\d\zeta+
C\int_{\tau}^s\|\Z(\zeta)\|_{\wi\L^{3(r-1)}}^{2(r-1)}
\d\zeta\bigg],
\end{align}
for all $s\in[\tau,T]$. Using Proposition \ref{supVnorm}, for the values of $r$ given in Table \ref{Table1}, we have 
\begin{align}\label{bphicphi2}
\int_{\tau}^T\|\Z(\zeta)\|_{\wi\L^{3(r-1)}}^{2(r-1)}
\d\zeta
\leq C(T,\|\Z\|_{\mathrm{L}^{\infty}(\tau,T;\V)},
\|\Z\|_{\mathrm{L}^{r+1}(\tau,T;\wi\L^{3(r+1)})}).
\end{align} 
Combining \eqref{bphicphi1}-\eqref{bphicphi2} together yields \eqref{phinegatest1}.
Let us set $\Psi:=\Phi_1-\Phi_2$. Then, $\Psi$ satisfies in $\V^*+\wi\L^{\frac{r+1}{r}}$
\begin{equation}\label{phi1phi2}
\left\{
\begin{aligned}
\frac{\d\Psi(s)}{\d s}&+\mu\mathcal{A}\Psi(s)+
\big(\mathfrak{B}'(\Z_1(s))\Phi_1(s)-\mathfrak{B}'(\Z_2(s))\Phi_2(s)\big)+
\alpha\Psi(s)
\\&+
\beta\big(\mathfrak{C}'(\Z_1(s))\Phi_1(s)-\mathfrak{C}'(\Z_2(s))
\Phi_2(s)\big)
=\boldsymbol{0},   \\
\Psi(\tau)&=\x_1-\x_2\in\H,
\end{aligned}
\right.
\end{equation}
for a.e.  $s\in[\tau,T]$. On taking the inner product with $(\mathcal{A}+\I)^{-1}\Psi$ in the first equation of \eqref{phi1phi2}, we obtain
\begin{align}\label{phi2phi3}
&\frac12\frac{\d}{\d s}\|(\mathcal{A}+\I)^{-\frac12}\Psi\|_{\H}^2+
\mu\|\Psi\|_{\H}^2+\alpha\|(\mathcal{A}+\I)^{-\frac12}\Psi\|_{\H}^2
\nonumber\\&=
\mu\|(\mathcal{A}+\I)^{-\frac12}\Psi\|_{\H}^2
-\underbrace{\big(\mathfrak{B}'(\Z_1)\Phi_1-\mathfrak{B}'(\Z_2)\Phi_2,
(\mathcal{A}+\I)^{-1}\Psi\big)}_{:=\mathcal{T}_3}
\nonumber\\&\quad-
\beta\underbrace{\big(\mathfrak{C}'(\Z_1)\Phi_1-\mathfrak{C}'(\Z_2)\Phi_2,
(\mathcal{A}+\I)^{-1}\Psi\big)}_{:=\mathcal{T}_4}.
\end{align}
We now estimate $\mathcal{T}_3$ and $\mathcal{T}_4$, separately. By the properties of bilinear operator along with H\"older's, Ladyzhenskaya's and Young's inequalities, we estimate $\mathcal{T}_3$ as
\begin{align}\label{bphi2}
\big|\mathcal{T}_3\big|&\leq
\big|\big(\mathfrak{B}'(\Z_1)\Psi,(\mathcal{A}+\I)^{-1}\Psi\big)\big|+
\big|\big(\mathfrak{B}'(\Z_1-\Z_2)\Phi_2,(\mathcal{A}+\I)^{-1}\Psi\big)\big|
\nonumber\\&\leq
2\|\Psi\|_{\H}\|\mathcal{A}^{\frac12}(\mathcal{A}+\I)^{-1}\Psi\|_{\wi\L^4}
\|\Z_1\|_{\wi\L^4}+
2\|\Phi_2\|_{\H}\|\mathcal{A}^{\frac12}(\mathcal{A}+\I)^{-1}\Psi\|_{\wi\L^4}
\|\Z_1-\Z_2\|_{\wi\L^4}
\nonumber\\&\leq
\underbrace{2C\|\Psi\|_{\H}
\|\mathcal{A}(\mathcal{A}+\I)^{-1}\Psi\|_{\H}^{\frac{d}{4}}
\|\mathcal{A}^{\frac12}(\mathcal{A}+\I)^{-1}\Psi\|_{\H}^{1-\frac{d}{4}}
\|\Z_1\|_{\wi\L^4}}_{\text{Young's inequality with exponents } \frac{8}{4-d} \ \text{ and } \frac{8}{4+d} \ }
\nonumber\\&\quad+
\underbrace{2C\|\Phi_2\|_{\H}
\|\mathcal{A}(\mathcal{A}+\I)^{-1}\Psi\|_{\H}^{\frac{d}{4}}
\|\mathcal{A}^{\frac12}(\mathcal{A}+\I)^{-1}\Psi\|_{\H}^{1-\frac{d}{4}}
\|\Z_1-\Z_2\|_{\wi\L^4}}_{\text{Young's inequality with exponents } \frac{8}{d}, \ \frac{8}{4-d}, \ 4 \text{ and } 4}
\nonumber\\&\leq
\frac{\mu}{4}\|\Psi\|_{\H}^2+C
\|(\mathcal{A}+\I)^{-\frac12}\Psi\|_{\H}^2\|\Z_1\|_{\wi\L^4}^{\frac{8}{4-d}}+
C\|(\mathcal{A}+\I)^{-\frac12}\Psi\|_{\H}^2+C
\|\Phi_2\|_{\H}^4\|\Z_1-\Z_2\|_{\wi\L^4}^4.
\end{align}
%where $\kappa_3:=2^{\frac{8}{4-d}}\left(\frac{4+d}{\mu}\right)^{\frac{4+d}{4-d}}
%\frac{4-d}{8}$, and $\kappa_4,\kappa_5$ are constants depending on $\mu$.
 From \eqref{T2est6} (see Appendix \ref{negnormest}), we estimate $\mathcal{T}_4$ as
\begin{align}\label{cphi2}
\big|\mathcal{T}_4\big|&\leq
\frac{\mu}{4}\|\Psi\|_{\H}^2+C
\|(\mathcal{A}+\I)^{-\frac12}\Psi\|_{\H}^2\|\Z_1\|_{\wi\L^{3(r-1)}}^{2(r-1)}+
C\|\Phi_2\|_{\V}^2\|\Z_1-\Z_2\|_{\V}^2
\nonumber\\&\quad+
C
\||\Z_1|+|\Z_2|\|_{\wi\L^{2(r-2)}}^{2(r-2)}\|(\mathcal{A}+\I)^{-\frac12}\Psi\|_{\H}^2.
\end{align}
On substituting \eqref{bphi2}-\eqref{cphi2} into \eqref{phi2phi3}, we deduce 
\begin{align*}
&\frac12\frac{\d}{\d s}\|(\mathcal{A}+\I)^{-\frac12}\Psi\|_{\H}^2+
\frac{\mu}{2}\|\Psi\|_{\H}^2+\alpha\|(\mathcal{A}+\I)^{-\frac12}\Psi\|_{\H}^2
\nonumber\\&\leq
\mu\|(\mathcal{A}+\I)^{-\frac12}\Psi\|_{\H}^2+
C
\|(\mathcal{A}+\I)^{-\frac12}\Psi\|_{\H}^2\|\Z_1\|_{\wi\L^4}^{\frac{8}{4-d}}+
C\|(\mathcal{A}+\I)^{-\frac12}\Psi\|_{\H}^2+C
\|\Phi_2\|_{\H}^4\|\Z_1-\Z_2\|_{\wi\L^4}^4
\nonumber\\&\quad+
C
\|(\mathcal{A}+\I)^{-\frac12}\Psi\|_{\H}^2\|\Z_1\|_{\wi\L^{3(r-1)}}^{2(r-1)}+
C\|\Phi_2\|_{\V}^2\|\Z_1-\Z_2\|_{\V}^2
\nonumber\\&\quad+
C\||\Z_1|+|\Z_2|\|_{\wi\L^{2(r-2)}}^{2(r-2)} \|(\mathcal{A}+\I)^{-\frac12}\Psi\|_{\H}^2.
\end{align*}
On integrating above from $\tau$ to $s$ and then employing Gr\"onwall's inequality, we arrive at 
\begin{align*}
&\|(\mathcal{A}+\I)^{-\frac12}\Psi(s)\|_{\H}^2+
\mu\int_{\tau}^s\|\Psi(\zeta)\|_{\H}^2\d\zeta+
2\alpha\int_{\tau}^s\|(\mathcal{A}+\I)^{-\frac12}\Psi(\zeta)\|_{\H}^2
\d\zeta
\nonumber\\&\leq
\bigg[\|(\mathcal{A}+\I)^{-\frac12}(\x_1-\x_2)\|_{\H}^2+2C
\int_{\tau}^s\|\Phi_2(\zeta)\|_{\H}^4
\|\Z_1(\zeta)-\Z_2(\zeta)\|_{\wi\L^4}^4
\nonumber\\&\qquad+2C
\int_{\tau}^s\|\Phi_2(\zeta)\|_{\V}^2
\|\Z_1(\zeta)-\Z_2(\zeta)\|_{\V}^2\d\zeta\bigg]
\nonumber\\&\quad\times
\exp\left(2\mu T+2C T+2C \int_{\tau}^s\|\Z_1(\zeta)\|_{\wi\L^4}^{\frac{8}{4-d}}\d\zeta+
2C\int_{\tau}^s
\|\Z_1(\zeta)\|_{\wi\L^{3(r-1)}}^{2(r-1)}\d\zeta\right.\\&\qquad\left.+
C\int_{\tau}^s 
\||\Z_1(\zeta)|+|\Z_2(\zeta)|\|_{\wi\L^{2(r-2)}}^{2(r-2)}
\d\zeta\right),
\end{align*}
for all $s\in[\tau,T]$. Note that $\V\hookrightarrow\wi\L^{2(r-2)}$ for $3\leq r\leq5$. Therefore, in view of Proposition \ref{supVnorm} and \eqref{bphicphi2}, the above inequality yields
\begin{align*}
&\|(\mathcal{A}+\I)^{-\frac12}\Psi(s)\|_{\H}^2+
\mu\int_{\tau}^s\|\Psi(\zeta)\|_{\H}^2\d\zeta+
\nonumber\\&\leq
C(\mu,\alpha,\beta,T,\|\z_1\|_{\V},\|\z_2\|_{\V},\|\x_1\|_{\H},\|\x_2\|_{\H})
\nonumber\\&\quad\times
\bigg[\|(\mathcal{A}+\I)^{-\frac12}(\x_1-\x_2)\|_{\H}^2+
\|\z_1-\z_2\|_{\V}^4+
\|\z_1-\z_2\|_{\V}^2\bigg],
\end{align*}
for all $s\in[\tau,T]$.
  \end{proof}

The proof of the following proposition proceeds along the same lines as the proof of Proposition \ref{phinegatest} and is therefore omitted. 
\begin{proposition}\label{Xnegatest}
The weak solution $\X$ to the linearized system \eqref{phx2} satisfies the following energy estimate for the values of $r$ given  in Table \ref{Table1}:
\begin{align}\label{Xnegatest1}
\sup\limits_{s\in[\tau,T]}\|(\mathcal{A}+\I)^{-\frac12}\X(s)\|_{\H}^2+
\int_{\tau}^T \|\X(\xi)\|_{\H}^2\d\xi
\leq 
C\big(\mu,\alpha,\beta,T,\|\z\|_{\V}\big).
\end{align}
Moreover, if $\X_1$ and $\X_2$ are two weak solutions of the system \eqref{phx2} with 
\begin{equation}\label{Xinitialdata}
\left\{
\begin{aligned}
\X_1(\tau):=\mu\mathcal{A}\z_1+\mathfrak{B}(\z_1)+\alpha\z_1+
\beta\mathfrak{C}(\z_1)-\u(\tau),\\
\X_2(\tau):=\mu\mathcal{A}\z_1+\mathfrak{B}(\z_1)+\alpha\z_1+
\beta\mathfrak{C}(\z_1)-\u(\tau).
\end{aligned}
\right.
\end{equation}
Then, we have
{\small\begin{align}\label{Xnegatest2}
&\sup\limits_{s\in[\tau,T]}
\|(\mathcal{A}+\I)^{-\frac12}(\X_1-\X_2)(s)\|_{\H}^2+
\int_{\tau}^T \|(\X_1-\X_2)(\xi)\|_{\H}^2\d\xi
\nonumber\\&\leq 
C\big(\mu,\alpha,\beta,T,\|\z_1\|_{\V},\|\z_2\|_{\V}\big)
\|\z_1-\z_2\|_{\V}.
\end{align}
}
\end{proposition}
\begin{remark}
	Note that the the right hand side of \eqref{Xnegatest1} depends only on the norm $\|\z\|_{\V}$.  This dependence mainly arises from the term $\|(\mathcal{A}+\I)^{-\frac12}(\mu\mathcal{A}\z+\mathfrak{B}(\z)+\alpha\z+
	\mathfrak{C}(\z))\|_{\H}$. Indeed,  by using the fact that (\cite[Chapter II, pp. 57]{Te3})
	\begin{align*}
	\big(\D(\mathcal{A}+\I)^{\frac12}\big)^{*}\cong
	\D(\mathcal{A}+\I)^{-\frac12} \ \text{ and } \ 
	\D(\mathcal{A}+\I)^{\frac12}\cong\V,
	\end{align*}
	we obtain $\D(\mathcal{A}+\I)^{-\frac12}\cong\V^{*}$. Therefore, we find
	\begin{align*}
		\|(\mathcal{A}+\I)^{-\frac12}\mathfrak{B}(\z)\|_{\H}=
		\|\mathfrak{B}(\z)\|_{\V^{*}}\leq C\|\z\|_{\V}^2,
	\end{align*}
	and for $r\in[3,5]$, we have
	\begin{align*}
	\|(\mathcal{A}+\I)^{-\frac12}\mathfrak{C}(\z)\|_{\H}=
	\|\mathfrak{C}(\z)\|_{\V^{*}}\leq C\||\z|^{r-1}\z\|_{\wi\L^{\frac{r+1}{r}}}\leq C\|\z\|_{\V}^{r}.
	\end{align*}
	Combining the above estimates, we obtain
	\begin{align*}
		\|(\mathcal{A}+\I)^{-\frac12}(\mu\mathcal{A}\z+\mathfrak{B}(\z)+\alpha\z+
		\mathfrak{C}(\z))\|_{\H}
		\leq C(\mu,\alpha,\beta,\|\z\|_{\V}).
	\end{align*}
	Using similar arguments, one can deduce the estimate in \eqref{Xnegatest2}. In particular, applying analogous estimates to the difference of the nonlinear terms yields
	\begin{align*}
	&\|(\mathcal{A}+\I)^{-\frac12}\big((\mu\mathcal{A}\z+\mathfrak{B}(\z)+\alpha\z+
	\mathfrak{C}(\z))-(\mu\mathcal{A}\z+\mathfrak{B}(\z)+\alpha\z+
	\mathfrak{C}(\z))\big)\|_{\H}
	\nonumber\\&\leq 
	C\big(\mu,\alpha,\beta,T,\|\z_1\|_{\V},\|\z_2\|_{\V}\big)
	\|\z_1-\z_2\|_{\V}.
	\end{align*}
\end{remark}
 We now establish the Lipschitz property of the functional $\mathcal{W}$ defined in \eqref{wfunc}. 
  	\begin{proposition}\label{WpropertyLip}
  		The functional $\mathcal{W}$ defined in \eqref{wfunc} satisfies following properties:
  		
  		(i) $\mathcal{W}:[t,T]\times\H\to\R$ is locally Lipschitz. 
  		
  		(ii) For a.e. $\tau\in[t,T]$, $\mathcal{W}_{\tau}(\tau,\cdot):\V\to\R$ is locally Lipschitz.
  		
  		(iii) For all $\tau\in[t,T]$ and all $\z\in\V$, the Fr\'echet derivative $\mathcal{D}_{\z}\mathcal{W}(\tau,\z)\in\V$. Also, the Fr\'echet derivative $\mathcal{D}_{\z}\mathcal{W}(\tau,\cdot):\V\to\V$ is locally Lipschitz.
  	\end{proposition}
  	
  	\begin{proof}
  	We prove all three parts step by step as follows:
  	\vskip 0.25cm
  	\noindent
  	\textbf{Proof of (i).}
  	For $\x\in\H$, we find the Fr\'echet derivative of the functional \eqref{wfunc} with respect to the initial data $\z$ as follows:
  	\begin{align}\label{phx3}
  		(\mathcal{D}_{\z}\mathcal{W}(\tau,\z),\x)=
  		(\Z(T),\Phi(T))+\int_{\tau}^T (\nabla\Z(\xi),\nabla\Phi(\xi))\d \xi,
  	\end{align}
  	for all $\x\in\H$. We then find by using the Cauchy--Schwarz inequality
  	\begin{align}\label{phx4}
  	\big|(\mathcal{D}_{\z}\mathcal{W}(\tau,\z),\x)\big|
  	&\leq\|\Z(T)\|_{\H}\|\Phi(T)\|_{\H}+\|\Z\|_{\mathrm{L}^2(\tau,T;\V)}
  	\|\Phi\|_{\mathrm{L}^2(\tau,T;\V)}
  	\nonumber\\&\leq
  	\big(\|\Z(T)\|_{\H}^2+\|\Z\|_{\mathrm{L}^2(\tau,T;\V)}^2\big)^{\frac12}
  	\big(\|\Phi(T)\|_{\H}^2+\|\Phi\|_{\mathrm{L}^2(\tau,T;\V)}^2\big)^{\frac12}.
  	\end{align}
By utilizing the energy estimates (see Proposition \ref{weLLp} and \eqref{phiengest1}), it follows from \eqref{phx4} that
  	\begin{align}\label{phx5}
  		\big|(\mathcal{D}_{\z}\mathcal{W}(\tau,\z),\x)\big|
  		\leq
  	C=C(\mathtt{R},\|\z\|_{\H})\|\x\|_{\H}, \ \text{ for all } \ \x\in\H.
  	\end{align}
  	Thus, for all $\z\in\D(\mathcal{A})$ and for all $\tau\in[t,T]$, we have 
  	\begin{align}\label{phx55}
  	\mathcal{D}_{\z}\mathcal{W}(\tau,\z)\in\H \ \text{ with norm bound } \  \|\mathcal{D}_{\z}\mathcal{W}(\tau,\z)\|_{\H}\leq C(\mathtt{R},\|\z\|_{\H}).
  	\end{align}
  	Similarly, for a.e. $\tau\in[t,T]$, we find 
  	\begin{align}
  	\mathcal{W}_\tau(\tau,\z)&=
  	(\Z(T),\X(T))+\int_{\tau}^T (\nabla\Z(\xi),\nabla\X(\xi))\d \xi- \frac12\|\nabla\z\|_{\H}^2-\frac12\|\wi\u(\tau)\|_{\H}^2\label{phx66}
     \\&\leq
  	\big(\|\Z(T)\|_{\H}^2+\|\Z\|_{\mathrm{L}^2(\tau,T;\V)}^2\big)^{\frac12}
  	\big(\|\X(T)\|_{\H}^2+\|\X\|_{\mathrm{L}^2(\tau,T;\V)}^2\big)^{\frac12}
  	+\frac12\|\nabla\z\|_{\H}^2+\frac12\|\wi\u(\tau)\|_{\H}^2
  	\nonumber\\&\leq
  	\mathpzc{C}(\mu,\alpha,\beta,T, \|\mathcal{A}\z\|_{\H})
  	+\frac12\|\wi\u(\tau)\|_{\H}^2,\label{phx6}
  	\end{align}
  	for all $\z\in\D(\mathcal{A})$ and for all
  	$\wi\u\in\mathrm{L}^2(\tau,T;\H)$. Thus, we deduce that $\mathcal{W}_{\tau}(\tau,\z)\in\mathrm{L}^1(\tau,T)$. 
  	Therefore, by virtue of \eqref{phx55}-\eqref{phx6} and an application of the mean value theorem \cite[Theorem 9.3-1, pp. 739]{PGCnew}, the functional
  	$\mathcal{W}:[t,T]\times\H\to\R$ is locally Lipschitz.
  	\vskip 2mm
  	\noindent
  	\textbf{Proof of (ii).}
  	Using the self-adjointness of $(\mathcal{A}+\I)^{\frac12}$, it follows from \eqref{phx66} that
  	\begin{align}\label{Wtauder}
  		\mathcal{W}_{\tau}(\tau,\z)&=
  		((\mathcal{A}+\I)^{\frac12}\Z(T),(\mathcal{A}+\I)^{-\frac12}\X(T))+
  		\int_{\tau}^T \langle\mathcal{A}\Z(r),\X(r)\rangle\d r
  		\nonumber\\&\quad-\frac12\|\nabla\z\|_{\H}^2-\frac12\|\wi\u(\tau)\|_{\H}^2,
  	\end{align}
  	for a.e. $\tau\in[t,T]$. By using the energy estimates \eqref{eqncont2} and \eqref{Xnegatest1}, we further estimate 
  		\begin{align}\label{phx9}
  		|\mathcal{W}_{\tau}(\tau,\z)|&\leq
  		\big(\|(\mathcal{A}+\I)^{\frac12}\Z(T)\|_{\H}^2+
  		\|\mathcal{A}\Z\|_{\mathrm{L}^2(\tau,T;\H)}^2\big)^{\frac12}
  		\big(\|(\mathcal{A}+\I)^{-\frac12}\X(T)\|_{\H}^2
  		+\|\X\|_{\mathrm{L}^2(\tau,T;\H)}^2\big)^{\frac12}
  		\nonumber\\&\quad+\frac12\|\nabla\z\|_{\H}^2+\frac12\|\wi\u(\tau)\|_{\H}^2
  		\nonumber\\&\leq
  		C\big(\mu,\alpha,\beta,T,\|\z\|_{\V}\big)
  		+\frac12\|\nabla\z\|_{\H}^2+\frac12\|\wi\u(\tau)\|_{\H}^2
  		\nonumber\\&\leq
  		C\big(\mu,\alpha,\beta,T,\|\z\|_{\V},\mathtt{R}\big),
  	\end{align}
  	 which shows that for a.e. $\tau\in[t,T]$ and for all $\z\in\V$, $\mathcal{W}_{\tau}(\tau,\z)\in\R$. Let us now prove the locally Lipschitz property of $\mathcal{W}_{\tau}(\tau,\cdot)$ for a.e. $\tau\in[t,T]$. 
%  	 From \eqref{phx7} and \eqref{phx9}, we have shown that
%  	\begin{align*}
%  	\big(\mathcal{W}_{\tau}(\tau,\cdot),\mathcal{D}_{\z}\mathcal{W}(\tau,\cdot)\big):
%  	\V\times\V\to\R\times\V.
%  	\end{align*}
 Let $\Z_1(\cdot)$ and $\Z_2(\cdot)$ be two strong solutions of \eqref{stapocp} with $\Z_1(\tau)=\z_1$ and $\Z_2(\tau)=\z_2$. Let $\X_1$ and $\X_2$ be two weak solutions of the linearized system \eqref{phx2} along $\Z_1$ and $\Z_2$, respectively, with initial data given in \eqref{Xinitialdata}. In view of H\"older's inequality and the energy estimates \eqref{eqncont2}-\eqref{supVnormest} and \eqref{Xnegatest1}-\eqref{Xnegatest2}, we obtain from \eqref{Wtauder} that 
   \begin{align*}
 	&\mathcal{W}_{\tau}(\tau,\z_1)-\mathcal{W}_{\tau}(\tau,\z_2)
 	\nonumber\\&=
 	((\mathcal{A}+\I)^{\frac12}\Z_1(T),(\mathcal{A}+\I)^{-\frac12}\X_1(T))-
 	((\mathcal{A}+\I)^{\frac12}\Z_2(T),(\mathcal{A}+\I)^{-\frac12}\X_2(T))
 	\nonumber\\&\quad+
 	\int_{\tau}^T \big[\langle\mathcal{A}\Z_1(r),\X_1(r)\rangle-
 	\langle\mathcal{A}\Z_2(r),\X_2(r)\rangle\big]\d r-
 	\frac12\|\nabla\z_1\|_{\H}^2+\frac12\|\nabla\z_2\|_{\H}^2
 	\nonumber\\&\leq
 	\|(\mathcal{A}+\I)^{\frac12}(\Z_1(T)-\Z_2(T))\|_{\H}
 	\|(\mathcal{A}+\I)^{-\frac12}\X_1(T)\|_{\H}
 	\nonumber\\&\quad+
 	\|(\mathcal{A}+\I)^{\frac12}\Z_2(T)\|_{\H}
 	\|(\mathcal{A}+\I)^{-\frac12}(\X_1(T)-\X_2(T))\|_{\H}
 	\nonumber\\&\quad+
 	\|\mathcal{A}(\Z_1-\Z_2)\|_{\mathrm{L}^2(\tau,T;\H)}
 	\|\X_1\|_{\mathrm{L}^2(\tau,T;\H)}+
 	\|\mathcal{A}\Z_2\|_{\mathrm{L}^2(\tau,T;\H)}
 	\|\X_1-\X_2\|_{\mathrm{L}^2(\tau,T;\H)}\nonumber\\&\quad+
 	\frac12\|\nabla(\z_1-\z_2)\|_{\H}^2
 	\nonumber\\&\leq
 	 C\big(\mu,\alpha,\beta,T,\mathtt{R}, \|\z_1\|_{\V},\|\z_2\|_{\V}\big)
 	 \|\z_1-\z_2\|_{\V}.
 	\end{align*}
 	Thus, $\mathcal{W}_{\tau}(\tau,\cdot):\V\to\R$ is locally Lipschitz for a.e. $\tau\in[t,T]$.
 	\vskip 2mm
 	\noindent
 	\textbf{Proof of (iii).}
 	By using the fact that $(\mathcal{A}+\I)^{\frac12}$ is self adjoint, we rewrite \eqref{phx3} as
 	\begin{align}\label{phx333}
 		\langle\mathcal{D}_{\z}\mathcal{W}(\tau,\z),\x\rangle=
 		\langle(\mathcal{A}+\I)^{\frac12}\Z(T),(\mathcal{A}+\I)^{-\frac12}\Phi(T)\rangle
 		+\int_{\tau}^T \langle\mathcal{A}\Z(\zeta),\Phi(\zeta)\rangle\d \zeta.
 	\end{align}
 	On applying the Cauchy--Schwarz and H\"older's inequalities along with the energy estimates \eqref{phinegatest1}, we calculate
 	\begin{align}\label{phx7}
 	&	\big|\langle\mathcal{D}_{\z}\mathcal{W}(\tau,\z),\x\rangle\big|
 	\nonumber\\	&\leq
 		\|(\mathcal{A}+\I)^{\frac12}\Z(T)\|_{\H}\|(\mathcal{A}+\I)^{-\frac12}\Phi(T)\|_{\H}
 		+\|\mathcal{A}\Z\|_{\mathrm{L}^2(\tau,T;\H)}\|\Phi\|_{\mathrm{L}^2(\tau,T;\H)}
 		\nonumber\\&\leq
 		\big(\|(\mathcal{A}+\I)^{\frac12}\Z(T)\|_{\H}^2+
 		\|\mathcal{A}\Z\|_{\mathrm{L}^2(\tau,T;\H)}^2\big)^{\frac12}
 		\big(\|(\mathcal{A}+\I)^{-\frac12}\Phi(T)\|_{\H}^2
 		+\|\Phi\|_{\mathrm{L}^2(\tau,T;\H)}^2\big)^{\frac12}
 		\nonumber\\&\leq
 		C\big(\mathtt{R},\|\z\|_{\V}\big)
 		\|(\mathcal{A}+\I)^{-\frac12}\x\|_{\H}^2,
 	\end{align}
 	for all $\x\in\D(\mathcal{A}+\I)^{-\frac12}$. Consequently, for all $\tau\in[t,T]$ and all $\z\in\V$, we conclude 
 	\begin{align}\label{phx77}
 		\mathcal{D}_{\z}\mathcal{W}(\tau,\z)\in\V \ \text{ with norm bound } \ 
 		\|\mathcal{D}_{\z}\mathcal{W}(\tau,\z)\|_{\V}\leq
 		C\big(\mathtt{R},\|\z\|_{\V}\big).
 	\end{align}
 	Let us now establish the locally Lipschitz property of  the Fr\'echet derivative $\mathcal{D}_{\z}\mathcal{W}(\tau,\cdot):\V\to\V$. Let $\Z_1(\cdot)$ and $\Z_2(\cdot)$ be two strong solutions of \eqref{stapocp} with $\Z_1(\tau)=\z_1$ and $\Z_2(\tau)=\z_2$. Let $\Phi_1$ and $\Phi_2$ be two weak solutions of the linearized system \eqref{phx1} along $\Z_1$ and $\Z_2$, respectively, with same initial data $\Phi_1(\tau)=\Phi_2(\tau)=\x$.
 	Using H\"older's inequality along with the energy estimates \eqref{eqncont2}-\eqref{supVnormest} and \eqref{phinegatest1}-\eqref{phinegatest2}, we derive from \eqref{phx333} the following estimates:
 	\begin{align*}
 		&\langle\mathcal{D}_{\z}\mathcal{W}(\tau,\z_1)-\mathcal{D}_{\z}\mathcal{W}
 		(\tau,\z_2),\x\rangle
 		\nonumber\\&=
 		((\mathcal{A}+\I)^{\frac12}\Z_1(T),(\mathcal{A}+\I)^{-\frac12}\Phi_1(T))-
 		((\mathcal{A}+\I)^{\frac12}\Z_2(T),(\mathcal{A}+\I)^{-\frac12}\Phi_2(T))
 		\nonumber\\&\quad+
 		\int_{\tau}^T \big[\langle\mathcal{A}\Z_1(r),\Phi_1(r)\rangle-
 		\langle\mathcal{A}\Z_2(r),\Phi_2(r)\rangle\big]\d r
 		\nonumber\\&=
 		((\mathcal{A}+\I)^{\frac12}\Z_1(T)-(\mathcal{A}+\I)^{\frac12}\Z_2(T),
 		(\mathcal{A}+\I)^{-\frac12}\Phi_1(T))
 		\nonumber\\&\quad+
 		((\mathcal{A}+\I)^{\frac12}\Z_2(T),(\mathcal{A}+\I)^{-\frac12}\Phi_1(T)-
 		(\mathcal{A}+\I)^{-\frac12}\Phi_2(T))
 		\nonumber\\&\quad+
 		\int_{\tau}^T \big[\langle\mathcal{A}\Z_1(r)-\mathcal{A}\Z_2(r),
 		\Phi_1(r)\rangle+
 		\langle\mathcal{A}\Z_2(r),\Phi_1(r)-\Phi_2(r)\rangle\big]\d r
 		\nonumber\\&\leq
 		\|(\mathcal{A}+\I)^{\frac12}(\Z_1(T)-\Z_2(T))\|_{\H}
 		\|(\mathcal{A}+\I)^{-\frac12}\Phi_1(T)\|_{\H}
 		\nonumber\\&\quad+
 		\|(\mathcal{A}+\I)^{\frac12}\Z_2(T)\|_{\H}
 		\|(\mathcal{A}+\I)^{-\frac12}(\Phi_1(T)-\Phi_2(T))\|_{\H}
 		\nonumber\\&\quad+
 		\|\mathcal{A}(\Z_1-\Z_2)\|_{\mathrm{L}^2(\tau,T;\H)}
 		\|\Phi_1\|_{\mathrm{L}^2(\tau,T;\H)}+
 		\|\mathcal{A}\Z_2\|_{\mathrm{L}^2(\tau,T;\H)}
 		\|\Phi_1-\Phi_2\|_{\mathrm{L}^2(\tau,T;\H)}
 		\nonumber\\&\leq
 		\left(\|(\mathcal{A}+\I)^{\frac12}(\Z_1(T)-\Z_2(T))\|_{\H}^2+
 		\|\mathcal{A}(\Z_1-\Z_2)\|_{\mathrm{L}^2(\tau,T;\H)}^2\right)^{\frac12}
 		\nonumber\\&\qquad\times
 		\left(\|(\mathcal{A}+\I)^{-\frac12}\Phi_1(T)\|_{\H}^2+
 		\|\Phi_1\|_{\mathrm{L}^2(\tau,T;\H)}^2\right)^{\frac12}
 		\nonumber\\&\quad+
 		\left(\|(\mathcal{A}+\I)^{\frac12}\Z_2(T)\|_{\H}^2+
 		\|\mathcal{A}\Z_2\|_{\mathrm{L}^2(\tau,T;\H)}^2\right)^{\frac12}
 		\nonumber\\&\qquad\times
 		\left(\|(\mathcal{A}+\I)^{-\frac12}(\Phi_1(T)-\Phi_2(T))\|_{\H}^2+
 		\|\Phi_1-\Phi_2\|_{\mathrm{L}^2(\tau,T;\H)}^2\right)^{\frac12}
 		\nonumber\\&\leq 	
 		C\big(\mu,\alpha,\beta,T,\mathtt{R}, \|\z_1\|_{\V},\|\z_2\|_{\V}\big)
 		\|\z_1-\z_2\|_{\V}\|\x\|_{\D(\mathcal{A}+\I)^{-\frac12}},
 	\end{align*}
 	for all $\x\in\D((\mathcal{A}+\I)^{-\frac12})$. Since $\big(\D(\mathcal{A}+\I)^{-\frac12}\big)^*\cong\D(\mathcal{A}+\I)^{\frac12}\cong\V$, we deduce 
 	\begin{align*}
 		\|\mathcal{D}_{\z}\mathcal{W}(\tau,\z_1)-
 		\mathcal{D}_{\z}\mathcal{W}(\tau,\z_2)\|_{\V}
 		\leq C\big(\mu,\alpha,\beta,T,\mathtt{R}, \|\z_1\|_{\V},\|\z_2\|_{\V}\big)
 		\|\z_1-\z_2\|_{\V}.
 	\end{align*}
 	Thus, for all $\tau\in[t,T]$, $\mathcal{D}_{\z}\mathcal{W}(\tau,\cdot):\V\to\V$ is locally Lipschitz.
 \end{proof}
 
 	\section{Pontryagin maximum principle and verification results} \setcounter{equation}{0}\label{Pontryagin}
	We now turn to the derivation of the Pontryagin maximum principle and a
verification theorem for the optimal control problem. The key idea is to
compare the functional $\mathcal{W}$, introduced and analyzed in the previous
section, with the value function $\mathpzc{V}$, using the characterization of
$\mathpzc{V}$ as a viscosity solution of the HJB equation.
 \begin{theorem}\label{PMPpr}
 Let $\wi\u\in\mathrm{L}^2(t,T;\U)$ be the optimal control for the initial data $(t,\z)\in[0,T]\times\D(\mathcal{A})$ and let $\Z(\cdot)\in\C([t,T];\V)\cap\mathrm{L}^2(t,T;\D(\mathcal{A}))$ be the optimal trajectory of the system \eqref{stapocp}. Then, for a.e. $s\in[t,T]$, we have
 \begin{align}\label{pmp234}
 	\mathcal{F}(\Z(s),\mathcal{D}_{\z}\mathcal{W}(s,\Z(s)))= \wi{\mathcal{F}}(\Z(s),\mathcal{D}_{\z}\mathcal{W}(s,\Z(s)),\wi\u(s)),
 \end{align}
where $\mathcal{F}$ and $\wi{\mathcal{F}}$ are defined in \eqref{hamfunvar} and \eqref{pseudo}, respectively.  Moreover, if $\p(\cdot)\in\C([\tau,T];\H)\cap\mathrm{L}^2(\tau,T;\V)$ is the unique weak solution to the following adjoint (or backward) equation in $\V^{*}+\wi\L^{\frac{r+1}{r}}$:
\begin{equation}\label{adjequation}
	\left\{
	\begin{aligned}
		-\frac{\d\p(s)}{\d s}+ \mu\mathcal{A}\p(s)+\big(\mathfrak{B}'(\Z(s))\big)^{*}\p(s)+
		\alpha\p(s)+\beta\mathfrak{C}'(\Z(s))\p(s)&=
		\mathcal{A}\Z(s),   \\
		\p(T)&=\Z(T)\in\H.
	\end{aligned}
	\right.
\end{equation}
 for a.e.  $s\in[t,T]$. Then, we have $\p(s)=\mathcal{D}_{\z}\mathcal{W}(s,\Z(s))$ for a.e. $s\in[t,T]$.
Furthermore, the optimal control is obtained as $\wi\u(\cdot)=\upsigma(\p(\cdot))$, where $\upsigma$ is defined in \eqref{upsigma} and we obtain the regularity
\begin{align}\label{optreg}
	\wi\u(\cdot)\in\C([t,T];\H) \ \text{ and } \ 
	\Z(\cdot)\in\C([t,T];\D(\mathcal{A}))\cap\C^1((t,T);\H).
\end{align} 
 \end{theorem}

\begin{proof}
	The proof is divided into following number of steps:
	\vskip 0.2cm
	\noindent
	\textbf{Step I:} \emph{Proof of maximum condition \eqref{pmp234}.}
	By using the semigroup property of the solution $\Z(\cdot)$ to the system \eqref{stapocp}, we write
	\begin{align*}
		\Z(\zeta,s;\Z(s,\tau;\z;\wi\u);\wi\u)=\Z(\zeta,\tau;\z,\wi\u)  \ \text{ for } \ \tau\leq s\leq \zeta\leq T.
	\end{align*}
	Utilizing the above identity, the functional \eqref{wfunc} can be written as 
	\begin{align*}
		\mathcal{W}(s,\Z(s))=
	\frac12\|\Z(T)\|_{\H}^2+\frac12\int_s^T 
	\big[\|\nabla\Z(\zeta)\|_{\H}^2+\|\wi\u(\zeta)\|_{\H}^2\big]\d \zeta.
	\end{align*}
	Let $\delta s$ be the small increment in time $s$. Then, from above, we write
	\begin{align*}
	\mathcal{W}(s,\Z(s))&=
	\frac12\|\Z(T)\|_{\H}^2+\frac12\int_s^{s+\delta s}
	\big[\|\nabla\Z(\zeta)\|_{\H}^2+\|\wi\u(\zeta)\|_{\H}^2\big]\d \zeta
	\nonumber\\&\qquad+
	\frac12\int_{s+\delta s}^T
	\big[\|\nabla\Z(\zeta)\|_{\H}^2+\|\wi\u(\zeta)\|_{\H}^2\big]\d \zeta
	\nonumber\\&=
	\mathcal{W}(s+\delta s,\Z(s+\delta s))+
	\frac12\int_s^{s+\delta s}
	\big[\|\nabla\Z(\zeta)\|_{\H}^2+\|\wi\u(\zeta)\|_{\H}^2\big]\d \zeta.
	\end{align*} 
	On simplifying further, we get
	\begin{align}\label{pmp1}
	&\frac{1}{\delta s}\big(\mathcal{W}(s+\delta s,\Z(s)) -\mathcal{W}(s,\Z(s))\big)
	+\frac{1}{\delta s}\big(\mathcal{W}(s+\delta s,\Z(s+\delta s)) -\mathcal{W}(s+\delta s,\Z(s))\big)
	\nonumber\\&\quad+
	\frac{1}{2\delta s}\int_s^{s+\delta s}
	\big[\|\nabla\Z(\zeta)\|_{\H}^2+\|\wi\u(\zeta)\|_{\H}^2\big]\d \zeta=0.
	\end{align}
	Passing to the limit as $\delta s\to0$ and using the differentiability of $\mathcal{W}$ in \eqref{pmp1}, we obtain
	\begin{align}\label{pmp2}
	\mathcal{W}_s(s,\Z(s))+\bigg[\big(\mathcal{D}_{\z}\mathcal{W}(s,\Z(s),\Z_s\big)
	+\frac12\|\nabla\Z(s)\|_{\H}^2+\frac12\|\wi\u(s)\|_{\H}^2\bigg]=0 \ 
	\text{ for a.e. } \ s\in[\tau,T].
	\end{align}
	The bracketed expression on the left hand side of \eqref{pmp2} corresponds to the pseudo-Hamiltonian (see \eqref{pseudo}), hence, \eqref{pmp2} can be rewritten as
	\begin{align}\label{pmp222}
	\mathcal{W}_s(s,\Z(s))+
	\wi{\mathcal{F}}(\Z(s),\mathcal{D}_{\z}\mathcal{W}(s,\Z(s),\wi\u(s))=0
	 \ \text{ for a.e. } \ s\in[\tau,T].
	\end{align}
	Since, $\wi\u\in\mathrm{L}^2(t,T;\U)$ is an optimal control, the above equality also yields
	\begin{align}\label{hamfunvar1}
	\mathcal{W}_s(s,\Z(s))+\sup\limits_{\v\in\U}
	\wi{\mathcal{F}}(\Z(s),\mathcal{D}_{\z}\mathcal{W}(s,\Z(s),\v)\geq0.
\end{align}
 Thus, by the definition of Hamiltonian \eqref{hamfunvar}, we infer that
 \begin{align}\label{pmp22}
 \mathcal{W}_s(s,\Z(s))+
 \mathcal{F}(\Z(s),\mathcal{D}_{\z}\mathcal{W}(s,\Z(s))\geq0.
 \end{align}
Note that for $\tau\in[t,T]$ and $\z\in\D(\mathcal{A})$, from \eqref{vaLvar} and \eqref{wfunc}, we have 
\begin{align}\label{pmp3}
\mathcal{W}(\tau,\z)\geq\mathcal{V}(\tau,\z) \ \text{ and } \ 
\mathcal{W}(t,\z)=\mathcal{V}(t,\z),
\end{align}
where the last equality follows from the fact that $\wi\u$ is an optimal control for $\tau=t$. Moreover, if $\wi\Z(s)=\wi\Z(s,t;\z;\wi\u)$ is the optimal trajectory, then 
\begin{align}\label{pmp4}
	\mathcal{W}(s,\wi\Z(s))=\mathcal{V}(s,\wi\Z(s)) \ \text{ for } \ s\in[t,T].
\end{align}
From \eqref{pmp3} and \eqref{pmp4} we conclude that 
\begin{align}\label{vmwmax}
	\mathcal{V}-\mathcal{W} \ \text{ attains a maximum of zero along each point of the trajectory } \ (s,\wi\Z(s)).  
\end{align}
Now using the fact that $\mathcal{V}$ is a viscosity subsolution (see Definition \ref{vaLuevarf}) we deduce
\begin{align}\label{pmp44}
 \mathcal{W}_s(s,\wi\Z(s))+
\mathcal{F}(\wi\Z(s),\mathcal{D}_{\z}\mathcal{W}(s,\wi\Z(s))\leq0.
\end{align}
On comparing \eqref{pmp22} and \eqref{pmp44} at $\tau=t$, we obtain the following
\begin{align}\label{pmp5}
 \mathcal{W}_s(s,\wi\Z(s))+
\mathcal{F}(\wi\Z(s),\mathcal{D}_{\z}\mathcal{W}(s,\wi\Z(s))=0.
\end{align}
Finally, on comparing \eqref{pmp5} with \eqref{pmp222} for $\tau=t$, we deduce \eqref{pmp234}.
\vskip 0.2cm
\noindent
\textbf{Step II:} \emph{Characterization of optimal control.} Let $\Phi$ be the solution to the linearized system \eqref{phx1}. Then, on taking the duality paring in the adjoint equation \eqref{adjequation} (along the optimal trajectory $\wi\Z(\cdot)$) with $\Phi$ and integrating from $s$ to $T$, we get
\begin{align*}
&\int_s^T \left\langle-\frac{\d\p(\zeta)}{\d s}+ \mu\mathcal{A}\p(\zeta)+\big(\mathfrak{B}'(\wi\Z(\zeta))\big)^{*}\p(\zeta)+
\alpha\p(\zeta)+\beta\mathfrak{C}'(\wi\Z(\zeta))\p(\zeta),\Phi(\zeta)\right\rangle
\d \zeta
\nonumber\\&=
\int_s^T \langle\mathcal{A}\wi\Z(\zeta),\Phi(\zeta)\rangle\d \zeta.
\end{align*}
On integrating by parts (it can justified by using the regularity on $\p$ and $\Phi$) and using the equation \eqref{phx1} satisfied by $\Phi$, we deduce further
\begin{align}\label{pmp6}
(\p(s),\Phi(s))-(\wi\Z(T),\Phi(T))=
	\int_s^T \big(\nabla\wi\Z(\zeta),\nabla\Phi(\zeta)\big)\d \zeta.
\end{align}
Since $\p\in \C([t,T];\H)$,  both sides of \eqref{pmp6} are continuous functions of $s$ and the equality holds for almost every $s\in[\tau,T]$, we conclude that \eqref{pmp6} actually holds for all $s\in[\tau,T]$. On comparing \eqref{pmp6} with \eqref{phx3} for $\tau=s$, we obtain
\begin{align*}
(\p(s),\x)=(\mathcal{D}_{\z}\mathcal{W}(s,\wi\Z(s)),\x),
\end{align*}
for every $\x\in\H$. Therefore, we conclude 
\begin{align*}
	\p(s)=\mathcal{D}_{\z}\mathcal{W}(s,\wi\Z(s)) \ \text{ for all } s\in[t,T].
\end{align*}
Moreover, from the regularity of the adjoint equation \eqref{adjequation}, we have
\begin{align}\label{adjregu}
	\p(\cdot)=\mathcal{D}_{\z}\mathcal{W}(\cdot,\wi\Z(\cdot))
	\in\C([t,T];\H)\cap\mathrm{L}^2(t,T;\V).
\end{align}
From \eqref{pmp234}, we also get the optimal control
\begin{align*}
\wi\u(\cdot)=\upsigma(\mathcal{D}_{\z}\mathcal{W}(\cdot,\wi\Z(\cdot)))\in
\C([t,T];\H),
\end{align*}
where the continuity of the optimal control follows from  \eqref{adjregu} and the fact that $\upsigma$ is locally Lipschitz. Moreover, in view of Proposition \ref{extregu}, we also conclude that $\Z(\cdot)\in\C([t,T];\D(\mathcal{A}))\cap\C^1((t,T);\H)$, which completes the proof of \eqref{optreg}.
\end{proof}

\begin{remark}\label{pbelongstosupdff}
In view of \eqref{vmwmax}, since $\mathcal{V}-\mathcal{W}$ has maximum at $(s,\wi\Z(s))$ and from Proposition \ref{WpropertyLip}, $\mathcal{W}$ has a Fr\'echet derivative $\mathcal{D}_{\z}\mathcal{W}(\cdot,\mathfrak{z})\in\V$ which is locally Lipschitz in $\mathfrak{z}$. Therefore, from the definition of superdifferential (see \ref{superdff}), we have 
		\begin{align}\label{DZDZ}
			\mathcal{D}_{\z}\mathcal{W}(\cdot,\wi\Z(\cdot))\in
			\mathcal{D}_{\z}^{+}\mathcal{V}(\cdot,\wi\Z(\cdot)).
	\end{align}
	Further, since $\p(\cdot)=\mathcal{D}_{\z}\mathcal{W}(\cdot,\Z(\cdot))$ is the unique solution of the adjoint system \eqref{adjequation}, therefore \eqref{DZDZ} yields the following
	\begin{align*}
	\p(\cdot)\in\mathcal{D}_{\z}^{+}\mathcal{V}(\cdot,\wi\Z(\cdot)).
	\end{align*}
\end{remark}

\subsection{A verification theorem}\label{nonsmoothverf}
Once the dynamic programming equation is solved, the next step is to determine how to construct an optimal feedback control directly from it. A classical way to obtain the optimal feedback control is by identifying the maximizer (or minimizer) of the Hamiltonian that appears in the dynamic programming equation. This approach, often referred to as a \emph{verification technique}, relies on the fact that the value function $\mathpzc{V}$ possesses sufficient smoothness. In contrast, the verification theorem formulated within the framework of viscosity solutions requires only mild assumptions on $\mathpzc{V}$. 
%As a consequence, these results apply to a much broader class of control problems than the classical smooth verification setting. 
Below we present a proof of viscosity solution based verification theorem.

\begin{theorem}\label{verifmain}
	The superdifferential $\mathcal{D}_{\z}^{+}\mathpzc{V}(s,\z)$ of the value function $\mathpzc{V}(s,\z)$ is non-empty for all $(s,\z)\in[t,T]\times\D(\mathcal{A})$, and further for all $\z\in\D(\mathcal{A})$, we have
	\begin{equation}\label{Dzpsup}
		\left\{
		\begin{aligned}
		\mathpzc{V}_s(s,\z)&=\mathcal{F}(s,\z,\p(s)) \ \text{ for a.e. } \ s\in[t,T],\\
		\text{ with } \ \p(s)&\in\mathcal{D}_{\z}^{+}\mathpzc{V}(s,\z), \ \z\in\H.
		\end{aligned}
		\right.
	\end{equation} 
	Moreover, the optimal control is given by the feedback relation:
	\begin{align*}
	\wi\u(\cdot)=\upsigma(\p(\cdot)) \ \text{ for some } \ 
	\p(\cdot)\in\mathcal{D}_{\z}^{+}\mathpzc{V}(\cdot,\wi\Z(\cdot)).
	\end{align*}
\end{theorem}

\begin{proof}
Note that the value function $\mathpzc{V}(\cdot,\z)$ is Lipschitz continuous in time for $\z\in\D(\mathcal{A})$. Thus, by an application of the Rademacher theorem \cite[Theorem 9.2-2, pp. 738]{PGCnew},  the value function $\mathpzc{V}(\cdot,\z)$ is differentiable a.e. in $[t,T]$. Let $s\in[t,T]$ be such a point of differentiability. Let $\wi\u$ be the optimal control and $\wi\Z(\cdot)=\wi\Z(\cdot;s;\z;\wi\u)$ be the optimal trajectory corresponding to the initial data $(s,\z)\in[t,T]\times\D(\mathcal{A})$. By definition of the value function, we have 
\begin{align}\label{valuedef1}
\mathpzc{V}(s,\z) =\frac12\int_s^T 
\big[\|\nabla\wi\Z(\zeta)\|_{\H}^2+\|\wi\u(\zeta)\|_{\H}^2\big]\d \zeta
+\frac12\|\wi\Z(T)\|_{\H}^2.
\end{align}
Moreover, we assume that $\Z(\cdot)=\Z(\cdot;s+\eps;\z;\wi\u)$ is the trajectory corresponding to the optimal control $\wi\u$ and initial data $(s+\eps,\z)\in[t,T]\times\D(\mathcal{A})$, for $\eps>0$. Therefore, again by the definition of value function, we find 
\begin{align}\label{valuedef2}
		\mathpzc{V}(s+\eps,\z)\leq\frac12\int_{s+\eps}^T 
	\big[\|\nabla\Z(\zeta)\|_{\H}^2+\|\wi\u(\zeta)\|_{\H}^2\big]\d \zeta
	+\frac12\|\Z(T)\|_{\H}^2.
\end{align}
From \eqref{valuedef1} and \eqref{valuedef2}, we write
\begin{align}\label{valuedef3}
\mathpzc{V}(s+\eps,\z)-\mathpzc{V}(s,\z)&\leq
\frac12\big[\|\Z(T;s+\eps;\z;\wi\u)\|_{\H}^2-\|\wi\Z(T;s;\z;\wi\u)\|_{\H}^2\big]
\nonumber\\&\quad+\frac12\int_{s+\eps}^T \big[\|\nabla\Z(\zeta;s+\eps;\z;\wi\u)\|_{\H}^2-\|\nabla\wi\Z(\zeta;s;\z;\wi\u)\|_{\H}^2\big]\d \zeta
\nonumber\\&\quad-\frac12\int_s^{s+\eps} \big[\|\nabla\wi\Z(\zeta;s;\z;\wi\u)\|_{\H}^2+\|\wi\u(\zeta)\|_{\H}^2\big]\d \zeta.
\end{align}
On dividing by $\eps>0$ in \eqref{valuedef3}, we obtain
\begin{align}\label{valuedef4}
\frac{1}{\eps}\big[\mathpzc{V}(s+\eps,\z)-\mathpzc{V}(s,\z)\big]&\leq
\frac{1}{2\eps}\big[\|\Z(T;s+\eps;\z;\wi\u)\|_{\H}^2-\|\wi\Z(T;s;\z;\wi\u)\|_{\H}^2\big]
\nonumber\\&\quad+\frac{1}{2\eps}\int_{s+\eps}^T \big[\|\nabla\Z(\zeta;s+\eps;\z;\wi\u)\|_{\H}^2-\|\nabla\wi\Z(\zeta;s;\z;\wi\u)\|_{\H}^2\big]\d \zeta
\nonumber\\&\quad-\frac{1}{2\eps}\int_s^{s+\eps} \big[\|\nabla\wi\Z(\zeta;s;\z;\wi\u)\|_{\H}^2+\|\wi\u(\zeta)\|_{\H}^2\big]\d \zeta.
\end{align}
On passing the limit as $\eps\to0$, we obtain
\begin{align}\label{valuedef9}
	\mathpzc{V}_s(s,\z)\leq
	\int_s^T \langle\mathcal{A}\wi\Z(\zeta),\X(\zeta)\rangle\d \zeta+
	\big(\wi\Z(T),\X(T)\big)-[\|\nabla\z\|_{\H}^2+\|\wi\u(s)\|_{\H}^2],
\end{align}
for a.e. $s\in[t,T]$. From the adjoint system \eqref{phx3} (along the optimal trajectory $\wi\Z(\cdot)$), we have in $\V^*+\wi\L^{\frac{r+1}{r}}$
\begin{equation}\label{phx3wZ}
	\left\{
	\begin{aligned}
		-\frac{\d\p(s)}{\d s}+ \mu\mathcal{A}\p(s)+\big(\mathfrak{B}'(\wi\Z(s))\big)^{*}\p(s)+
		\alpha\p(s)+\beta\mathfrak{C}'(\wi\Z(s))\p(s)&=
		\mathcal{A}\wi\Z(s),  \\
		\p(T)&=\wi\Z(T)\in\H,
	\end{aligned}
	\right.
\end{equation}
 for a.e.  $ s\in[t,T].$ On substituting \eqref{phx3wZ} into \eqref{valuedef9} and performing  integration by parts (this can be justified due to the regularity on $\p(\cdot)$ and $\X(\cdot)$), and then using \eqref{phx2}, we deduce 
\begin{align}\label{valuedef10}
	\mathpzc{V}_s(s,\z)\leq
	(\p(s),\X(s))-[\|\nabla\z\|_{\H}^2+\|\wi\u(s)\|_{\H}^2],
\end{align}
where $\X(s)$ is the initial value of solution $\X(\cdot)$ to the linearized system \eqref{phx2} at time $s$, and is given by
\begin{align}\label{LNs}
	\X(s)=\mu\mathcal{A}\z+\mathfrak{B}(\z)+\alpha\z+
	\beta\mathfrak{C}(\z)-\u(s).
\end{align}
By inserting \eqref{LNs} into \eqref{valuedef10} and then using \eqref{pseudo} (pseudo-Hamiltonian) along with \eqref{pmp234} (Proposition \ref{PMPpr}), we arrive at
\begin{align}\label{valuedef11}
	\mathpzc{V}_s(s,\z)\leq\mathcal{F}(s,\z,\p(s)), \ \text{ for a.e. } \ s\in[t,T].
\end{align}
We need to prove the reverse inequality in \eqref{valuedef11}. Let us assume that $\wi\u$ is the optimal control in $[s+\eps,T]$ corresponding to the initial data $\z\in\D(\mathcal{A})$ at $s+\eps$ and let $\wi\Z(\cdot)=\Z(\cdot;s+\eps;\z;\wi\u)$ be the corresponding trajectory. Let us set 
\begin{equation}\label{conscont}
\u(\zeta)=
\left\{
\begin{aligned}
\wi\u(\zeta), \  &\text{ if } \ \zeta\in(s+\eps,T],\\
\wi\u(s+\eps), \ &\text{ if } \ \zeta\in(s,s+\eps].
\end{aligned}
\right.
\end{equation} 
We denote by $\Z=\Z(\cdot;s;\z;\u)$, the corresponding trajectory for the control $\u$ and the initial data $(s,\z)\in[t,T]\times\D(\mathcal{A})$.  Then, by the definition of value function, we write
\begin{align}
	\mathpzc{V}(s+\eps,\z)&=\frac12\int_{s+\eps}^T 
	\big[\|\nabla\wi\Z(\zeta;s+\eps;\z;\wi\u)\|_{\H}^2+\|\wi\u(\zeta)\|_{\H}^2\big]\d \zeta
	+\frac12\|\wi\Z(T;s+\eps;\z;\wi\u)\|_{\H}^2,\label{valuedef12}\\
	\text{ and } 
	\mathpzc{V}(s,\z)&\leq\frac12\int_{s}^T 
	\big[\|\nabla\Z(\zeta;s;\z;\u)\|_{\H}^2+\|\u(\zeta)\|_{\H}^2\big]\d \zeta
	+\frac12\|\Z(T;s;\z;\u)\|_{\H}^2.\label{valuedef13}
\end{align}
On combining \eqref{valuedef12}-\eqref{valuedef13} and utilizing \eqref{conscont}, we deduce the following inequality:
\begin{align*}
	\mathpzc{V}(s+\eps,\z)-\mathpzc{V}(s,\z)&\geq
	\frac12\big[\|\wi\Z(T;s+\eps;\z;\wi\u)\|_{\H}^2-\|\Z(T;s;\z;\wi\u)\|_{\H}^2\big]
	\nonumber\\&\quad+\frac12\int_{s+\eps}^T \big[\|\nabla\wi\Z(\zeta;s+\eps;\z;\wi\u)\|_{\H}^2-\|\nabla\Z(\zeta;s;\z;\u)\|_{\H}^2\big]\d \zeta
	\nonumber\\&\quad-\frac12\int_s^{s+\eps} \big[\|\nabla\Z(\zeta;s;\z;\u)\|_{\H}^2+\|\u(\zeta)\|_{\H}^2\big]\d \zeta.
\end{align*}
Upon dividing by $\eps>0$, passing to the limit as $\eps\to0$ and  following the same argument used in \eqref{valuedef9}, we obtain
\begin{align}\label{valuedef14}
	\mathpzc{V}_s(s,\z)\geq\mathcal{F}(s,\z,\p(s)), \ \text{ for a.e. } \ s\in[t,T].
\end{align}
Moreover, from Remark \ref{pbelongstosupdff}, we have 
\begin{align}\label{pbelongstosupdff1}
	\p(\cdot)\in\mathcal{D}_{\z}^{+}\mathpzc{V}(\cdot,\z) \ \text{ for } \ \z\in\H.
\end{align}
Combining \eqref{valuedef11}, \eqref{valuedef14}, and \eqref{pbelongstosupdff1}, we obtain \eqref{Dzpsup}, which completes the proof.
\end{proof}

\begin{appendix}\renewcommand{\thesection}{\Alph{section}}
	\numberwithin{equation}{section} 
		\section{Some essential estimates in negative norms}\label{negnormest}	
		We estimate the bound of the following term:
\begin{align}\label{T2est1}
\mathcal{T}_2
	=\big(\mathfrak{C}'(\Z_1)\Phi_1-\mathfrak{C}'(\Z_2)\Phi_2,
	(\mathcal{A}+\I)^{-1}\Psi\big),
\end{align}
which is used in \eqref{cphi2} of Proposition \ref{phinegatest}. Let us rewrite $\mathcal{T}_2$ as
\begin{align}\label{T2est2}
\mathcal{T}_2
=\underbrace{\big(\mathfrak{C}'(\Z_1)\Psi,
(\mathcal{A}+\I)^{-1}\Psi\big)}_{\mathcal{T}_2^1}+
\underbrace{\big((\mathfrak{C}'(\Z_1)-\mathfrak{C}'(\Z_2))\Phi_2,
(\mathcal{A}+\I)^{-1}\Psi\big)}_{\mathcal{T}_2^2}.
\end{align}
The term $\mathcal{T}_2^1$ can be estimated similarly to \eqref{cphi}, yielding
	\begin{align}\label{T2est3}
	\left|\big(\mathfrak{C}'(\Z_1)\Psi,
	(\mathcal{A}+\I)^{-1}\Psi\big)\right|\leq
	\frac{\mu}{4}\|\Psi\|_{\H}^2+C
	\|(\mathcal{A}+\I)^{-\frac12}\Psi\|_{\H}^2\|\Z_1\|_{\wi\L^{3(r-1)}}^{2(r-1)}.
\end{align}
We now estimate $\mathcal{T}_2^2$. From \eqref{C1d}, we write
\begin{align}\label{T2est3.1}
&(\mathfrak{C}'(\Z_1)-\mathfrak{C}'(\Z_2))\Phi_2
\nonumber\\&=
\mathcal{P}((|\Z_1|^{r-1}-|\Z_2|^{r-1})\Phi_2)+
(r-1)\mathcal{P}\big(\Z_1|\Z_1|^{r-3}(\Z_1\cdot\Phi_2)-
\Z_2|\Z_2|^{r-3}(\Z_2\cdot\Phi_2)\big),
\end{align}
provided $r\geq3$. Let us define the function of one variable $\varphi(\cdot):[0,1]\to\R^d$ by 
\begin{align*}
	\varphi(\theta):=((1-\theta)\Z_2+\theta\Z_1)
	|(1-\theta)\Z_2+\theta\Z_1|^{r-3}
	\big(((1-\theta)\Z_2+\theta\Z_1)\cdot\Phi_2\big), \text{ for  } \ \theta\in[0,1],
\end{align*}
where $\Z_1,\Z_2$ and $\Phi_2\in\R^d$. It is easy to see that
\begin{align*}
	\varphi(1)-\varphi(0)=\Z_1|\Z_1|^{r-3}(\Z_1\cdot\Phi_2)-
	\Z_2|\Z_2|^{r-3}(\Z_2\cdot\Phi_2), \ r\geq3.
\end{align*}
The derivative of $\varphi$ is given by
\begin{align*}
\varphi'(\theta)&=(\Z_1-\Z_2)|(1-\theta)\Z_2+\theta\Z_1|^{r-3}
\big(((1-\theta)\Z_2+\theta\Z_1)\cdot\Phi_2\big)
\nonumber\\&\quad+
((1-\theta)\Z_2+\theta\Z_1)
|(1-\theta)\Z_2+\theta\Z_1|^{r-3}\big((\Z_1-\Z_2)\cdot\Phi_2\big)
\nonumber\\&\quad+
(r-3)((1-\theta)\Z_2+\theta\Z_1)|(1-\theta)\Z_2+\theta\Z_1|^{r-5}
\big(((1-\theta)\Z_2+\theta\Z_1)\cdot(\Z_1-\Z_2)\big)
\nonumber\\&\qquad\times\big(((1-\theta)\Z_2+\theta\Z_1)\cdot\Phi_2\big)
\end{align*}
By an application of the Mean Value Theorem \cite[Theorem 9.3-1, pp. 739]{PGCnew}, we estimate
\begin{align}\label{T2est4}
	&\left|\Z_1|\Z_1|^{r-3}(\Z_1\cdot\Phi_2)-
	\Z_2|\Z_2|^{r-3}(\Z_2\cdot\Phi_2)\right|
	\nonumber\\&=
	|\varphi(1)-\varphi(0)|
	\nonumber\\&\leq
	\max\limits_{\theta\in[0,1]}|\varphi'(\theta)|
	\nonumber\\&=
	\max\limits_{\theta\in[0,1]}
	\bigg|(\Z_1-\Z_2)|(1-\theta)\Z_2+\theta\Z_1|^{r-3}
	\big(((1-\theta)\Z_2+\theta\Z_1)\cdot\Phi_2\big)
	\nonumber\\&\quad+
	((1-\theta)\Z_2+\theta\Z_1)
	|(1-\theta)\Z_2+\theta\Z_1|^{r-3}\big((\Z_1-\Z_2)\cdot\Phi_2\big)
	\nonumber\\&\quad+
	(r-3)((1-\theta)\Z_2+\theta\Z_1)|(1-\theta)\Z_2+\theta\Z_1|^{r-5}
	\big(((1-\theta)\Z_2+\theta\Z_1)\cdot(\Z_1-\Z_2)\big)
	\nonumber\\&\qquad\times\big(((1-\theta)\Z_2+\theta\Z_1)\cdot\Phi_2\big)
	\bigg|
	\nonumber\\&\leq (r-1)\left(\max\limits_{\theta\in[0,1]}|(1-\theta)\Z_2+\theta\Z_1|^{r-2}\right)
	|\Z_1-\Z_2|\,|\Phi_2|
	\nonumber\\&\leq 
	(r-1)\left(|\Z_1|+|\Z_2|\right)^{r-2}|\Z_1-\Z_2|\,|\Phi_2|,
\end{align}
for $r\ge 2$. Analogously, one can establish the following bound too:
\begin{align}\label{T2est4.1}
\big||\Z_1|^{r-1}\Phi_2-|\Z_2|^{r-1}\Phi_2\big|
\leq(r-1)\left(|\Z_1|+|\Z_2|\right)^{r-2}
|\Z_1-\Z_2|\,|\Phi_2|.
\end{align}
By utilizing \eqref{T2est4}-\eqref{T2est4.1} into \eqref{T2est3.1} and using H\"older's inequality, we compute 
\begin{align}\label{T2est4.2}
\|(\mathfrak{C}'(\Z_1)-\mathfrak{C}'(\Z_2))\Phi_2\|_{\wi\L^{\frac65}}
\leq\mathpzc{C}_{r}
\||\Z_1|+|\Z_2|\|_{\wi\L^{2(r-2)}}^{r-2}
\|\Z_1-\Z_2\|_{\wi\L^6}\|\Phi_2\|_{\wi\L^6}, \ r\geq2,
\end{align}
where $\mathpzc{C}_{r}:=(r-1)+(r-1)^2$. From \eqref{T2est4.2}, we estimate
$\mathcal{T}_2^2$ as
\begin{align}\label{T2est5}
\big|\mathcal{T}_2^2\big|&\leq
\big|\big((\mathfrak{C}'(\Z_1)-\mathfrak{C}'(\Z_2))\Phi_2,
(\mathcal{A}+\I)^{-1}\Psi\big)\big|
\nonumber\\&\leq
\|(\mathfrak{C}'(\Z_1)-\mathfrak{C}'(\Z_2))\Phi_2\|_{\wi\L^{\frac65}}
\|(\mathcal{A}+\I)^{-1}\Psi\|_{\wi\L^6}
\nonumber\\&\leq
\mathpzc{C}_{r}\||\Z_1|+|\Z_2|\|_{\wi\L^{2(r-2)}}^{r-2}
\|\Z_1-\Z_2\|_{\wi\L^6}\|\Phi_2\|_{\wi\L^6}
\|(\mathcal{A}+\I)^{-1}\Psi\|_{\wi\L^6}
\nonumber\\&\leq
C\|\Phi_2\|_{\V}^2
\|\Z_1-\Z_2\|_{\V}^2+C
\||\Z_1|+|\Z_2|\|_{\wi\L^{2(r-2)}}^{2(r-2)}\|(\mathcal{A}+\I)^{-\frac12}\Psi\|_{\H}^2,
\ \ r\geq2.
\end{align}
On combining \eqref{T2est3} and \eqref{T2est5} into \eqref{T2est2}, we find
\begin{align}\label{T2est6}
\big|\mathcal{T_2}\big|&\leq
\frac{\mu}{4}\|\Psi\|_{\H}^2+\kappa_2
\|(\mathcal{A}+\I)^{-\frac12}\Psi\|_{\H}^2\|\Z_1\|_{\wi\L^{3(r-1)}}^{2(r-1)}+
C\|\Phi_2\|_{\V}^2
\|\Z_1-\Z_2\|_{\V}^2
\nonumber\\&\quad+
C
\||\Z_1|+|\Z_2|\|_{\wi\L^{2(r-2)}}^{2(r-2)}\|(\mathcal{A}+\I)^{-\frac12}\Psi\|_{\H}^2,
\ r\geq2.
\end{align}

\end{appendix}

	\medskip
	\noindent
	\textbf{Acknowledgments:} 
	The first author gratefully acknowledges the financial support provided by the Ministry of Education, Government of India (MHRD). M. T. Mohan's research is supported by the National Board of Higher Mathematics (NBHM), Department of Atomic Energy, Government of India, under Project No. 02011/13/2025/NBHM(R.P)/R\&D II/1137.

	\medskip\noindent	\textbf{Declarations:} 
	
	\noindent 	\textbf{Ethical Approval:}   Not applicable 
	
	%\noindent  \textbf{Competing interests: } The authors declare no competing interests. 
	
	\noindent  \textbf{Conflict of interest: }On behalf of all authors, the corresponding author states that there is no conflict of interest.
	
	\noindent 	\textbf{Authors' contributions: } All authors have contributed equally. 
	
	\noindent 	\textbf{Funding: } NBHM, India, 02011/13/2025/NBHM(R.P)/R\&D II/1137 (M. T. Mohan)
	
	\noindent 	\textbf{Availability of data and materials: } Not applicable.

\end{document}